\documentclass[11pt]{article}
\usepackage{amsfonts,amsmath,amssymb,amsthm,mathrsfs,colortbl}
\usepackage{url}
\usepackage[font=normalsize,textfont=normalfont]{caption}
\usepackage{subcaption}
\usepackage{epsfig,latexsym,graphicx}
\usepackage{esint}
\usepackage{graphicx,latexsym,amssymb,amsmath,amsfonts,fancyhdr}
\usepackage{euscript}
\usepackage[letterpaper,hmargin=1.0in,vmargin=1.0in]{geometry}

\newtheorem{maintheorem}{Theorem}

\newtheorem{theorem}{Theorem}[section]
\newtheorem{lemma}[theorem]{Lemma}

\newtheorem{proposition}[theorem]{Proposition}
\newtheorem{observation}[theorem]{Observation}

\newtheorem{definition}[theorem]{Definition}

\def\XXint#1#2#3{{\setbox0=\hbox{$#1{#2#3}{\int}$ }
\vcenter{\hbox{$#2#3$ }}\kern-.6\wd0}}

\renewcommand{\P}{\mathbb{P}}

\newcommand{\A}[1]{\mathcal{A}^{(#1)}}

\newcommand{\cd}{\mathcal{D}}
\newcommand{\ch}{\mathcal{H}}
\newcommand{\ce}{\mathcal{E}}
\newcommand{\cg}{\mathcal{G}}
\newcommand{\cj}{\mathcal{J}}
\newcommand{\cm}{\mathcal{M}}
\newcommand{\Z}{\mathbb{Z}}
\newcommand{\N}{\mathbb{N}}
\newcommand{\R}{\mathbb{R}}
\newcommand{\cf}{\mathcal{F}}

\begin{document}
\title{Lipschitz Embeddings of Random Fields}

\author{Riddhipratim Basu\thanks{Stanford University. Email: rbasu@stanford.edu.}
\and
Vladas Sidoravicius\thanks{Courant Institute of Mathematical Sciences, New York, NYU-ECNU Institute of Mathematical Sciences at NYU Shanghai and
Cemaden, Sao Jose dos Campos. Email: vs1138@nyu.edu}
\and
Allan Sly\thanks{University of California, Berkeley. Email: sly@stat.berkeley.edu}
}

\date{}
\maketitle

\begin{abstract}
We consider the problem of embedding one i.i.d.\ collection of Bernoulli random variables indexed by $\Z^d$ into an independent copy in an injective $M$-Lipschitz manner. For the case $d=1$, it was shown in \cite{BS14} to be possible almost surely for sufficiently large $M$. In this paper we provide a multi-scale argument extending this result to higher dimensions.  
\end{abstract}

\section{Introduction}\label{s:intro}
%\subsection{Lipschitz Embedding}
%This paper is the second in the sequence of three investigating Lipschitz embeddings and rough isometries of higher dimensional random objects. 

The question of finding Lipschitz embeddings of one random object into another has attracted a significant interest in recent years \cite{GLR:10, GriHol:10a, GriHol:10, GriHol:12, HolMar:12}, these questions also have close connections with co-ordinate percolation problems \cite{BBS:00, Winkler:00, BSS14+}. In this paper, we consider a natural question in this class for Euclidean lattices.  

Let $\mathbb{X}=\{X_{v}\}_{v\in \Z^d}$ and $\mathbb{Y}=\{Y_{v}\}_{v\in \Z^d}$ be collections of binary fields on $\Z^d$. We say $\mathbb{X}$ can be \textbf{$M$-embedded} in $\mathbb{Y}$ if there exists an injective map $\phi: \Z^d \rightarrow \Z^d$ such that $X_{v}=Y_{\phi(v)}$ $\forall v\in \Z^d$ and $||\phi(v_1)-\phi(v_2)||\leq M||v_1-v_2||$  $\forall v_1,v_2\in \Z^d$ where $||\cdot||$ denotes the Euclidean norm in $\Z^d$. We denote the event that $\mathbb{X}$ embeds into $\mathbb{Y}$ by $\mathbb{X}\hookrightarrow _{M} \mathbb{Y}$.

The primary question we investigate is the following. Suppose $\mathbb{X}$ and $\mathbb{Y}$ are independent collection of i.i.d.\ Bernoulli variables. Does there exist $M$ sufficiently large such that $\mathbb{X}\hookrightarrow _{M} \mathbb{Y}$ almost surely? This question was answered affirmatively for $d=1$ in \cite{BS14}. Our contribution is to extend this result to higher dimensions.

\begin{maintheorem}
\label{t:embedh}
Let $\mathbb{X}=\{X_{v}\}_{v\in \Z ^2}$ and $\mathbb{Y}=\{Y_{v}\}_{v\in \Z ^2}$ be independent collections of i.i.d.\ $\mbox{Ber}(\frac{1}{2})$ random variables. There exists $M>0$ such that $\mathbb{X}$ can be $M$-embedded in $\mathbb{Y}$ almost surely.  
\end{maintheorem}

A couple of remarks are in order before we proceed further. 

Observe that by ergodicity, the event $\mathbb{X}\hookrightarrow _{M} \mathbb{Y}$ is a $0-1$ event, and hence to prove Theorem \ref{t:embedh} it suffices to prove that $\P[\mathbb{X}\hookrightarrow _{M} \mathbb{Y}]>0$ for $M$ sufficiently large.

Notice also that we state our result for $d=2$, and strictly speaking that is what we shall prove. However, it shall be clear following our proof that the same argument, modulo minor modifications and suitable changes in parameters, will go through for $d>2$ as well; where the Lipschitz constant $M=M(d)$ will depend on $d$. For the sake of clarity and notational convenience we have chosen to write the proof for the case $d=2$ only.

% We shall focus on the case $d=2$, and it will be clear that the same arguments will go through in higher dimensions with suitable changes of parameters and constants.

\subsection{Related Works}
In early 1990s Winkler introduced a fascinating class of dependent percolation problems, the so-called co-ordinate percolation problems, where the vertices are open or closed depending on variables on co-ordinate axes. Long-range dependence makes these problems not amenable to the tools of Bernoulli percolation. It turns out that several natural questions about embedding one random sequence into another following certain rules can be reformulated as problems in this class (see e.g.~\cite{BBS:00,CTP:00,Grimmett:09,Winkler:00}). In particular, Grimmett, Liggett and Richthammer~\cite{GLR:10} asked whether there exists a Lipschitz embedding of one Bernoulli sequence (indexed by $\Z$) into an independent copy. This question was recently answered in \cite{BS14} (see also \cite{Gacs:12}). The problem we investigate in this paper is a natural generalisation of the above question to higher dimensions.

In \cite{BS14}, a multi-scale structure was developed which was flexible enough to solve a number of embedding question. The novelty in that multi-scale approach is to ignore the anatomy of difficult
to embed configurations and instead obtain recursive tail estimates for the probability that they can be embedded into
a random block. A similar approach has been used to solve a co-ordinate percolation problem in \cite{BSS14+}. Our approach here is also essentially similar to that in \cite{BS14}, however substantially more complicated geometry of blocks in higher dimensions, makes it much more technically challenging. 

In a  similar vein is the question of a rough, (or  quasi-), isometry of two independent Poisson processes.  Informally, two metric spaces are roughly isometric if their metrics are equivalent up to multiplicative and additive constants. Originally Ab\'{e}rt~\cite{Abert:08} asked whether two independent infinite components of bond percolation on a Cayley graph are roughly isometric. Szegedy and also Peled~\cite{Peled:10} asked the problem when these sets are independent Poisson process in $\R$  (see~\cite{Peled:10} for a fuller description of the history of the problem). This problem was settled in \cite{BS14} as well. To settle the question in higher dimension (i.e., to show that for large enough choices of parameters, two independent copies of rate one Poisson processes on $\R^d$ are rough isometric almost surely) requires more work than we do in this paper and is subject of a future work in preparation. A crucial ingredient needed for this extension comes from \cite{BSS14+++}.

% 
%
%The question of Lipschitz embedding in one one dimension could be represented as a ``co-ordinate'' percolation type problem (see e.g.~\cite{BBS:00,CTP:00,Grimmett:09,Winkler:00}) which were popularized by Peter Winkler since early 1990's. Problems in this class can be interpreted  as oriented percolation problems in $\Z^2$ where the sites are open or closed according to random variables on the co-ordinate axes. 

In a related direction, a series of works by Grimmett, Holroyd and their collaborators~\cite{DDGHS:10,Grimmett:09,GriHol:10a,GriHol:10,GriHol:12} investigated a number of problems including when one can embed $\mathbb{Z}^d$ into site percolation in $\mathbb{Z}^D$ and showed that this was possible almost surely for $M=2$ when $D>d$ and the the site percolation parameter was sufficiently large but almost surely impossible for any $M$ when $D \leq d$.  Recently Holroyd and Martin \cite{HolMar:12} showed that a comb can be embedded in $\mathbb{Z}^2$.  Another series of work in this area involves embedding words into higher dimensional percolation clusters~\cite{BenKes:95,HLNS14,KSV:98,KSZ:01}.

\subsection{Outline of the proof}
\label{s:outline}  
Like essentially all results in this area, our argument is also based on multi-scale analysis. The main challenge is to match the difficult to embed regions in $\mathbb{X}$ to their suitable partners in $\mathbb{Y}$ simultaneously at all scales. Our approach here is in spirit similar to the one taken in \cite{BS14}, but it is technically much more challenging because the difficult-to-embed regions can have many different shapes and complicated geometries in higher dimensions. 

We divide the collections $\mathbb{X}$ and $\mathbb{Y}$ into blocks on a series of doubly exponentially growing length scales $L_j=L_0^{\alpha^j}$ for $j\geq 0$. A block of level $j$ is approximately a square of side length $L_j$, though we also allow blocks of more complicated shapes and larger sizes. At each of these levels we define a notion of a ``good'' block. Single elements in  $\mathbb{X}$ constitute the level 0 blocks and in $\mathbb{Y}$ squares of a fixed (large) size make level 0 blocks.

Suppose that we have constructed the blocks up to level $j$. In Section \ref{s:prelimh}, we give a construction of $(j+1)$-level blocks as a union of $j$-level sub-blocks in such a way that the blocks are identically distributed, non neighbouring blocks are independent and there are no bad $j$-level subblocks very close to the boundary of a $(j+1)$-level block. To ensure the last condition we need to allow blocks to be of larger size, and in certain cases blocks will approximate a connected union of squares of size $L_j$. For more details, see Section \ref{s:prelimh}.

At each level we distinguish a set of blocks to be good.  In particular this will be done in such a way that at each level $(j+1)$ for \emph{any} good block $X$ in $\mathbb{X}$ and \emph{any} good block $Y$ in $\mathbb{Y}$, their $j$-level bad sub-blocks can be matched with suitable partners via a bi-Lipschitz map of Lipschitz constant $(1+10^{-(j+4)})$ (this is termed as embedding at level $(j+1)$). Flexibility in choosing this map gives us an improved chance to find suitable partners for difficult to embed blocks at higher levels. We describe how to define good blocks in Section \ref{s:goodhd}. We also define components which are unions of blocks such that different components containing bad sub-blocks are separated by good components which are just single good blocks.

The proof then involves a series of recursive estimates at each level given in Section \ref{s:rech}. We ask that at level $j$ the probability that a block is good is at least $1-L_j^{-\gamma}$, conditioned on a subset (possibly empty) of other level $j$ blocks and hence the vast majority of blocks are good.  Furthermore, we show tail bounds on the embedding probabilities showing that for $0<p\leq 1-L_j^{-1}$,
\[
\P(S_j^{\mathbb{X}}(X)\leq p, V_X\geq v)\leq p^{m+2^{-j}}L_j^{-\beta}L_{j}^{-\gamma(v-1)}
\]
where $S_j^{\mathbb{X}}(X)$ denotes the $j$-level embedding probability of a $j$ level component $X$, and $V_X$ denote the number of squares of size $L_j$ that $X$ approximates, see Section \ref{s:dembedh} for a formal definition. We show the analogous bounds for $\mathbb{Y}$-blocks as well. 
The full inductive step is given in Section \ref{inductionh}. Proving this constitutes the main work of the paper.

The key quantitative estimate in the paper is Proposition~\ref{l:totalSizeBoundh} which follows directly from the recursive estimates, and bounds the chance of a block having a large size, many bad sub-blocks or a particularly difficult collection of sub-blocks measured by the product of their embedding probabilities.  In order to improve the embedding probabilities at each level we need to take advantage of the flexibility in mapping a small collection of bad blocks to a large number of possible partners in a Lipschitz manner with appropriate Lipschitz constants. To this effect we define families of maps between blocks to describe such potential maps. Because $m$ is large and we take many independent trials  the estimate at the next level improves significantly. Our analysis is split into four different cases.

To show that good blocks at level $(j+1)$ have the required properties, we construct them so that the total size of bad subcomponents contained in them is at most $k_0$ and all of which are ``semi-bad" (defined in Section \ref{s:dembedh}) in particular with embedding probability close to $1$.  We also require that every semi-bad block maps into a large proportion of the sub-blocks in every $L_j^{3/2}\times L_j^{3/2}$ square of $j$ level blocks contained in a $(j+1)$. Under these conditions we show that good blocks can always be mapped to any other good block.

To complete the proof we note that with positive probability the blocks surrounding the origin are good at each level. The proof is then completed using a standard compactness argument.

\subsubsection{Parameters}\label{s:parameters}
Our proof involves a collection of parameters $\alpha,\beta,\gamma,k_0,m$ and $v_0$ which must satisfy a system of constraints.  The required constraints are
\begin{eqnarray*}
\alpha>6, \gamma>40\alpha, \beta>1500\alpha\gamma, k_0>6000\alpha \gamma, v_0>3000\alpha,\\
8\gamma(v_0-1)>3\alpha\beta, m\geq 9\alpha \beta +3\alpha \gamma v_0, \gamma k_0> 300\alpha \beta, k_0>10\gamma, (1-10^{-10})^{4v_0}>\frac{9}{10}.
\end{eqnarray*}
%\textcolor{blue}{Write the correct constraints and parameters}

To fix on a choice we will set
\begin{equation}\label{e:parameters}
\alpha =8, \gamma =350, \beta =4500000, v_0=45000, m=15\times 10^7, k_0=13\times 10^6. 
\end{equation}
Given these choices we then take $L_0$ to be a sufficiently large integer.  We did not make a serious attempt to optimize the parameters or constraints, often with the aim of keeping the exposition more transparent.

\subsection*{Acknowledgements}
This work was completed when R.B. was a graduate student at the Department of Statistics at UC Berkeley. He gratefully acknowledges the support of UC Berkeley graduate fellowship. The result in this paper appeared in Chapter 4 of the Ph.D. dissertation of R.B. at UC Berkeley: \emph{Lipschitz Embeddings of Random Objects and Related Topics, 2015}. A.S. was supported by an Alfred Sloan Fellowship and NSF grants DMS-1208338, DMS-1352013.

\subsection*{Organization of the paper}
Rest of this paper is organised as follows. In Section~\ref{s:prelimh} we describe our block constructions and formally define good blocks.  In Section~\ref{s:rech} we state the main recursive theorem and show that it implies Theorem~\ref{t:embedh}.  In Section~\ref{s:construction} we construct a collection of bi-Lipschitz functions which we will use to describe our mappings between blocks.  In Section~\ref{s:tailestimateh} we prove the main recursive tail estimates on the embedding probabilities.  In Section~\ref{s:goodh} we show that good blocks have the required inductive properties thus completing the induction.

\section{The Multi-scale Structure}\label{s:prelimh}
As mentioned before, we shall restrict ourselves only to the two dimensional setting, but the reader can observe that the construction described here naturally goes through for in higher dimensions with minimal changes. For reasons of notational convenience that will momentarily be clear, without loss of generality, we shall take our sequence to be indexed by a translate of $\Z^2$ rather than $\Z^2$ itself. Let $\iota =(\frac{1}{2}, \frac{1}{2})$. Let $\mathbb{X}=\{X_{v}\}_{v\in \iota +\Z^2}$ and $\mathbb{Y}=\{Y_{v}\}_{v\in \iota + \Z^2}$ be collections of i.i.d.\ $\mbox{Ber}(\frac{1}{2})$ random variables.

As mentioned above, our argument for proof of Theorem \ref{t:embedh} is multi-scale and depends of partitioning $\mathbb{X}$ and $\mathbb{Y}$ into blocks at level $j$-for each $j\geq 0$. The blocks are constructed recursively. For the purpose of our construction we shall work with $\R^2$ rather than $\Z^2$. At each level $j$ we shall partition $\R^2$ into disjoint (except at the boundary) random regions $\{\mathscr{B}^{j,\mathbb{X}}_{\alpha}\}_{\alpha\in I_1}$ and $\{\mathscr{B}^{j,\mathbb{Y}}_{\alpha}\}_{\alpha\in I_2}$ respectively for $\mathbb{X}$ and $\mathbb{Y}$.

\textbf{We shall interchangeably use the term blocks at level $j$ (for $\mathbb{X}$, say) to refer to the regions $\mathscr{B}^{j,\mathbb{X}}_{\alpha}$ or the collection of random variables contained in these regions: $\{X_{u}:u\in \mathscr{B}^{j,\mathbb{X}}_{\alpha}\}$.}  

Our blocks will be indexed by elements in a random partition of $\Z^2$. 

\subsection{Blocks at level $0$}
We start with describing the construction of blocks at level $0$. Construction of blocks at level $0$ are different for $\mathbb{X}$ and $\mathbb{Y}$. Also level $0$ blocks are deterministic (i.e.\ the regions corresponding to them are deterministic) and indexed by vertices in $\Z^2$. 

For each $u=(u_1,u_2)\in \Z^2$, the $\mathbb{X}$-block at level $0$ indexed by $u$, denoted by $X^0(u)$ corresponds to the region $[u_1,u_1+1]\times [u_2,u_2+1]$. 

Let $M_0\in \N$ denote some large constant to be determined later. For each $u=(u_1,u_2)\in \Z^2$, the $\mathbb{Y}$-block at level $0$ indexed by $u$, denoted by $Y^0(u)$ corresponds to the region $[u_1M_0,(u_1+1)M_0]\times [u_2M_0,(u_2+1)M_0]$.

For $U\subseteq \Z^2$, the collection of blocks $\{X^0(u):u\in U\}$ will be denoted by $X^0_U$ (and similarly for $Y^0_U$).

Observe that level $0$ blocks are independent for both $\mathbb{X}$ and $\mathbb{Y}$. Level $0$ blocks are fundamental units of our multi-scale structure. All the blocks at higher scales will be unions of blocks at level $0$. For the rest of this construction, we rescale space for $\mathbb{Y}$ such that blocks at level $0$ become unit squares. Under this rescaling construction of higher level blocks are performed identically for $\mathbb{X}$ and $\mathbb{Y}$.

% Fix $M_0\in \N$ sufficiently large, to be chosen later. For $\alpha=(\alpha_1,\alpha_2)\in \Z^2$, let 
%$$U^{0,\mathbb{X}}_{\alpha}:= [\alpha_1,\alpha_1+1]\times [\alpha_2, \alpha_2+1];~ U^{0,\mathbb{Y}}_{\alpha}:= [\alpha_1M_0,(\alpha_1+1)M_0]\times [\alpha_2M_0, (\alpha_2+1)M_0].$$
%
%That is, $\{U^{0,\mathbb{X}}_{\alpha}\}_{\alpha\in \Z^2}$ partitions $\R^2$ into unit squares and $\{U^{0,\mathbb{Y}}_{\alpha}\}_{\alpha\in \Z^2}$ partitions $\R^2$ into squares of size $M_0$. We shall call $U^{0,\mathbb{X}}_{\alpha}$ (resp.\ $U^{0,\mathbb{Y}}_{\alpha}$) to be $\mathbb{X}$-blocks (resp.\ $\mathbb{Y}$-blocks) at level $0$. Interchangeably, depending on the context, we shall also call $\{X_{V^0_{\alpha}}\}$ (resp.\ $\{Y_{V^0_{\alpha}}\}$) $\mathbb{X}$-blocks (resp.\ $\mathbb{Y}$-blocks) at level $0$ where $V^{0}_{\alpha}$ is otained from $U^{0,\mathbb{X}}_{\alpha}$ (resp.\ $U^{0,\mathbb{Y}}_{\alpha}$) as described above. Note that since $\{V^0_{\alpha}\}:=\{T^{-1}(U^{0,\mathbb{X}}_{\alpha})\}$ (resp.\ $\{V^0_{\alpha}\}:=\{T^{-1}(U^{0,\mathbb{Y}}_{\alpha})\}$) forms a partition of $\Z^2$, this definition in unambiguous. Notice also that blocks at level $0$ are mutually independent.   

\subsubsection{Good Blocks at level 0}
As we have mentioned before, at each scale of the multi-scale construction, we shall designate a set of blocks in both $\mathbb{X}$ and $\mathbb{Y}$ to as {\bf good}. At level $0$, each $\mathbb{X}$-block will be good. For $u\in \Z^2$, $Y^0(u)$ is called good if we have the fraction of both 0's and 1's contained in $Y^0(u)$ is at least $1/3$ i.e.,

$$\#\{v\in Y^0(u): Y_v=1\}\wedge \#\{v\in Y^0(u): Y_v=0\} \geq \frac{M_0^2}{3}.$$

\subsection{An Overview of the Recursive Construction}
After rescaling blocks at level $0$ the recursive construction of blocks at higher levels is identical for both $\mathbb{X}$ and $\mathbb{Y}$. Without loss of generality, we shall restrict ourselves to construction of the blocks for $\mathbb{X}$ for levels $j\geq 1$. Our recursive block construction algorithm is fairly complex and has many elements to it. To facilitate the reader, before giving the formal definition, in this subsection we give a rough description of how the construction goes and make a list of different terms associated with the construction for easy reference. 

\begin{enumerate}
\item[$\bullet$] {\bf Cells:} Cells at level $j$ are basic units of construction at level $j$. These are squares, indexed by $\Z^2$, of size $L_j$, that $\R^2$ is divided into. Denote the cell corresponding to $u\in \Z^2$ by $B^j(u)$. Recall that $L_j=L_0^{\alpha^j}$ is the doubly exponentially increasing length scale. 
\item[$\bullet$] {\bf Buffer Zones:} Buffer zones are regions around the boundary of a cell, which should be thought of as fattened versions of the boundaries of cells.
\item[$\bullet$]{\bf Lattice Blocks:} At each level $j$ we partition $\Z^2$ as a (random) union of lattice animals (connected finite subsets). The elements of this are called lattice blocks. Let the set of lattice blocks at level $j$ be $\mathcal{H}_j=\mathcal{H}$.  The blocks at level $j$ are indexed by elements of $\ch$, typically denoted $X=X^j_{H}, H\in \ch$. For $H$ that is a union of elements in $\ch$, $X^j_{H}$ will denote the union of the corresponding blocks.
\item[$\bullet$]{\bf Ideal Multi-blocks:} For a lattice block $H$ at level $j$, we call  $\cup_{u\in H} B^j(u)$ an ideal multi-block.
\item[$\bullet$]{\bf Domains and Boundary Curves of Blocks:} Domains of blocks at level $j$ are small bi-Lipschitz perturbation of ideal multi-blocks. These are formed in such a way that the boundaries of the domains are nice (in some sense to be specified later). For a block $X$, we typically denote its domain by $\hat{U}_X$ and the curve corresponding to the boundary of $\hat{U}_{X}$ by $C_{X}$.
\item[$\bullet$]{\bf Blocks:} We shall define regions $\tilde{U}_X$, that are unions of smaller level blocks and these will define blocks. The regions $\tilde{U}_X$ will be defined as approximations of the regions $\hat{U}_X$ defined above. We shall denote the term block interchangeably for the region defining it as well as the collection of random variables in the region. For a block $X$, we shall denote by $V_{X}$ the size of a block, i.e., the size of the lattice block corresponding to it.
\item[$\bullet$]{\bf Good and Bad Blocks:} At each level, we designate some of the blocks to be good (depending on the configuration), other blocks are called bad. Good blocks will always correspond to lattice blocks of size 1, but the converse need not be true. 
\item[$\bullet$]{\bf Components of blocks:} We also form components of blocks at each level, where a component is a connected union of a number of blocks such that two components containing bad blocks are not neighbouring. Components are deterministically determined given the blocks and the identity of good blocks. For a component $X$, the size of it, i.e., the total size of all lattice blocks contained in that component will be denote by $V_{X}$.

\item[$\bullet$]{\bf Semi-bad Components:}  Components are called bad if they contain one or more bad blocks. Some of the bad components are designated as semi-bad component, depending on the configuration.    
\end{enumerate}

Observe that, at level $0$, lattice blocks are all singletons. Cells, domains and blocks are all the same and boundary curves are just the boundaries of cells. Now we give a detailed description of how we construct each of the steps above for $j\geq 1$.

\subsection{Cells and Buffer Zones}
%Now we come to the recursive construction of blocks at level $j$ for $j\geq 1$. The procedure followed is identical for $\mathbb{X}$ and $\mathbb{Y}$ so we restrict ourself to $\mathbb{X}$ from now on. The only difference is the fundamental unit for construction (i.e., level $0$ blocks) is a unit square for $\mathbb{X}$ whereas it is a square of side length $M_0$ for $\mathbb{Y}$, so everything has to be scaled up appropriately while dealing with $\mathbb{Y}$. We start with the following definition.

\begin{definition}[Cells at level $j$]
\label{d:ib}
For $j\geq 1$, set $L_j=L_{j-1}^{\alpha}=L_0^{\alpha^j}$. For $j\geq 1$ and $u=(u_1,u_2)\in \Z^2$, we define $B^{j}(u)=[u_1L_j,(u_1+1)L_j]\times [u_2L_j, (u_2+1)L_j]$. These squares which partition $\R^2$, will be called {\bf cells at level $j$}. 
\end{definition}

Observe that cells at level $j$ are squares of doubly exponentially growing length $L_j$. Also observe that cells are nested across $j$, i.e., a cell at level $j\geq 1$ is a union of $L_j^{2\alpha-2}$ many cells at level $(j-1)$. The above definition is illustrated in Figure \ref{f:ib}. 

\begin{figure}[h]
\begin{center}
\includegraphics[width=0.6\textwidth]{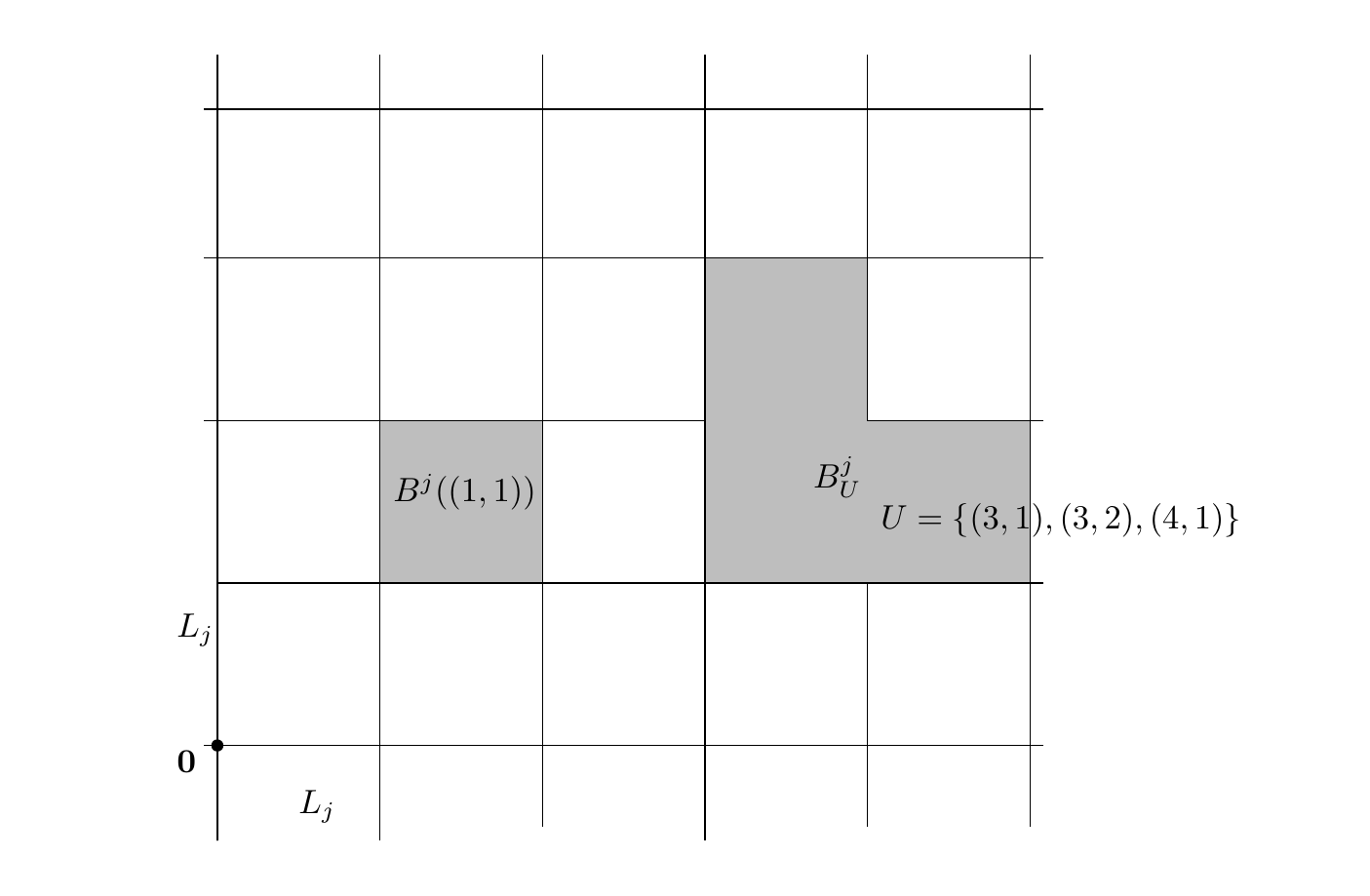}
\end{center}
\caption{Cells and multi-cells at level $j$}
\label{f:ib}
\end{figure}

The basic philosophy of constructing the blocks here is similar to that in \cite{BS14}: we want the region around the boundary of the blocks at level $j$ to consist of `good' subblocks at level $(j-1)$. Because of the more complicated geometry of $\R^2$ (as compared to the real line considered in \cite{BS14}) we shall need to consider cells of different shapes and sizes at a given level. This motivates the following sequence of definitions.   

\begin{definition}[Lattice animals and Shapes]
\label{d:animal}
A connected finite subset of vertices in $\Z^2$ is called a {\bf lattice animal}. Two lattice animals $U$ and $U'$ are said to have the same {\bf shape} if there is a translation from $\Z^2$ to itself that takes $U$ to $U'$.  
\end{definition}

We shall use the term shape also to identify equivalence classes of lattice animals having the same shape.

Two cells $B^{j}(u)$ and $B^{j}(u')$ are called {\bf neighbouring} if they share a common side, i.e., if $u$ and $u'$ are neighbours in $\Z^2$.  

\begin{definition}[Multi-cells at level $j$]
\label{d:mib}
For a lattice animal $U\subset \Z^2$, we call $B^j_{U}:=\cup_{u\in U} B^j(u)$ a {\bf multi-cell} at level $j$ corresponding to the lattice animal $U$.  
\end{definition}

The size of a multi-cell at level $j$ is defined to be $|U|$, i.e., the number of cells contained in it. The topological boundary of $B^j_{U}$ shall be denoted by $\partial B^{j}_{U}$.

%For each $j\geq 1$, a set of the form $B^{j}_{U}=\cup_{u\in U} B^{j}(u)$, for some  finite connected subset $U\subset \Z^2$ is called a {\bf lattice animal at level $j$}. That is, lattice animals are unions of ideal blocks that are connected  through neighbous. Size of a lattice animal at level $j$ will be defined as the number of $j$ level ideal blocks contained in the lattice animal. For a lattice animal $B^{j}_U$ at level $j$, $\partial B^{j}_{U}$ will denote the boundary of $B^{j}_U$. 
 
Our blocks at levels $j\geq 1$ will be suitable perturbations of certain $j$ level multi-cells (ideal multi-blocks) that ensure that there are no bad $(j-1)$ level subblocks near the boundary. To define the appropriate notion of perturbation we need to consider slightly thinned and fattened versions of cells at levels $j\geq 1$. 

\begin{definition}[Buffer zones of cells]
\label{d:bufferzone} Consider the squares 
$$B^{j,{\rm int}}(\mathbf{0}):=[L_{j-1}^5, L_j-L_{j-1}^5]^2;$$
$$B^{j,{\rm ext}}(\mathbf{0}):=[-L_{j-1}^5, L_j+L_{j-1}^5]^2.$$
For $j\geq 1$, call $B^{j,{\rm int}}(\mathbf{0})$ the {\bf interior} and $B^{j,{\rm ext}}(\mathbf{0})$ the {\bf blow up} of the $j$-level cell $B^j(\mathbf{0})$. 

For $u=(u_1,u_2)\in \Z^2$, define the interior and blow up of the cell $B^j(u)$ by setting 

$$B^{j,{\rm int}}(u):= (u_1L_j, u_2L_j)+B^{j,{\rm int}}(\mathbf{0});$$
$$B^{j,{\rm ext}}(u):=(u_1L_j, u_2L_j)+B^{j,{\rm ext}}(\mathbf{0}).$$

We call $\Delta B^{j}(u):=B^{j,{\rm ext}}(u)\setminus B^{j,{\rm int}}(u)$ the {\bf buffer} for the cell $B^j(u)$. We write $\Delta B^{j}(u)$ as the (non-disjoint) union of 4 rectangles called the {\bf top, left, bottom} and {\bf right} buffer zone denoted $\Delta B^{j,T}(u)$, $\Delta B^{j,L}(u)$, $\Delta B^{j,B}(u)$ and $\Delta B^{j,R}(u)$ respectively. Define $\Delta B^{j,T}(u)=\Delta B^{j}(u)\cap \Delta B^j(u')$ where $u'=u+(0,1)$, rest are defined similarly. 
\end{definition}
%\marginpar{figure here}

\begin{figure}[h]
\begin{center}
\includegraphics[width=0.5\textwidth]{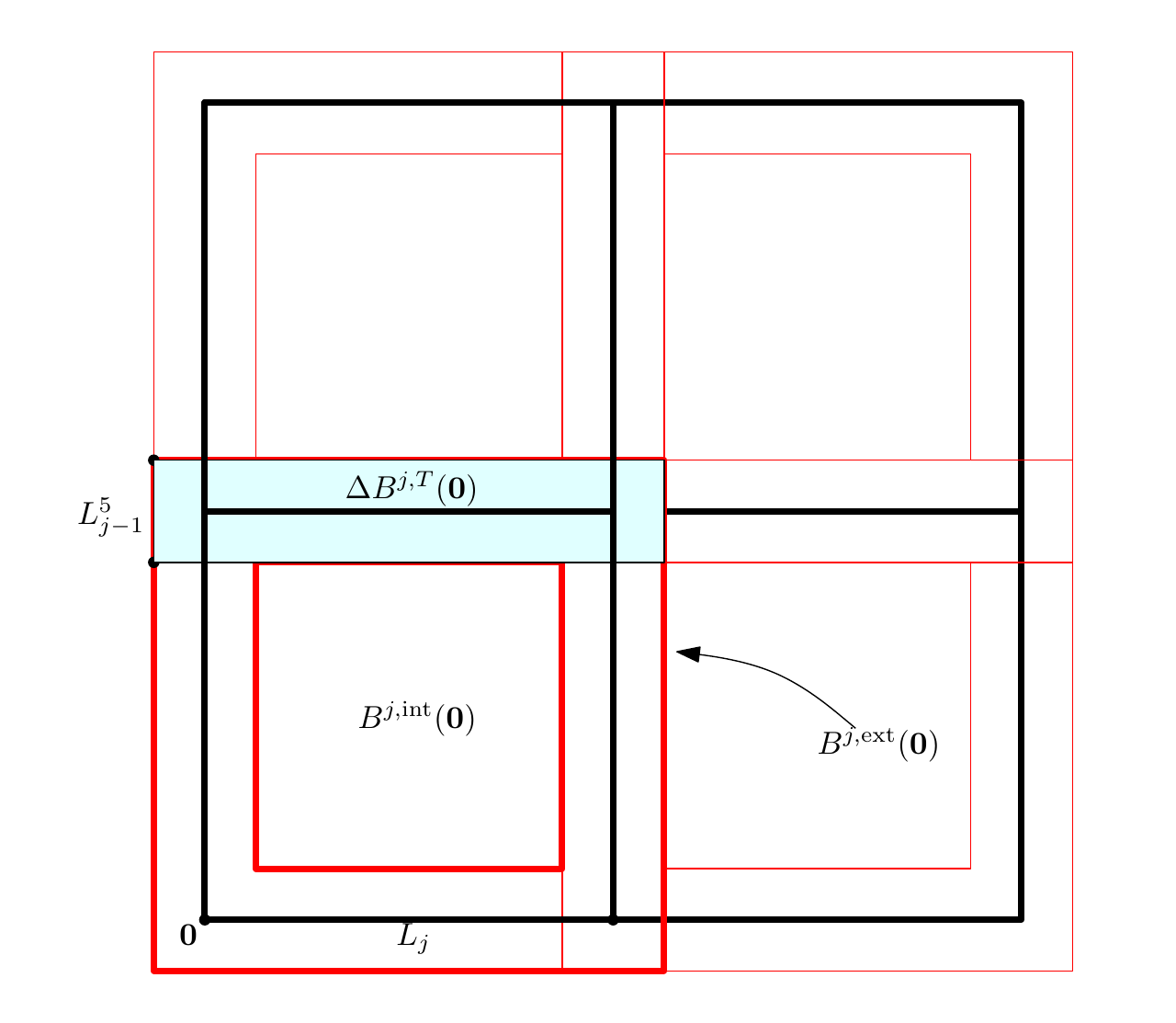}
\end{center}
\caption{Buffer Zones of Cells}
\label{f:buffer}
\end{figure}

Observe that if $u$ and $u'$ are neighbours in $\Z^2$, then $B^{j}(u)$ and $B^{j}(u')$ has one rectangular buffer zone (e.g.\ $\Delta B^{j,{T}}(u)$) in common, and conversely every rectangular  buffer zone is shared between two neighbouring cells. If $u$ and $u'$ are neighbours in the closed packed lattice of $\Z^2$ (i.e., $u$ and $u'$ are neighbours if $||u-u'||_{\infty}=1$) then also their buffer zones intersect. See Figure \ref{f:buffer} for illustration of this definition.

We next extend the definition of buffer zone to multi-cells at level $j$.

\begin{definition}[Buffer zones of Multi-cells]
\label{d:bufferzonem}
Fix $j\geq 1$, and a lattice animal $U\subset \Z^2$. Consider $B^{j}_{U}$, the multi-cell corresponding to $U$ at level $j$. For $u\in U$, and $\star \in \{T,L,B,R\}$, we call $\Delta B^{j,\star}(u)$ an {\bf outer buffer zone} of $B^{j}_{U}$ if this buffer zone is shared with a cell outside $B^{j}_{U}$. The buffer $\Delta B^{j}(U)$, of the multi-cell $B^{j}_{U}$ is defined as the union of all outer buffer zones of $B^{j}(u)$ for $u\in U$. The interior and blow up of $B^j_U$ is defined similarly as above. 
\end{definition}

This is illustrated in Figure \ref{f:buffer2}.

\begin{figure}[h]
\begin{center}
\includegraphics[width=0.5\textwidth]{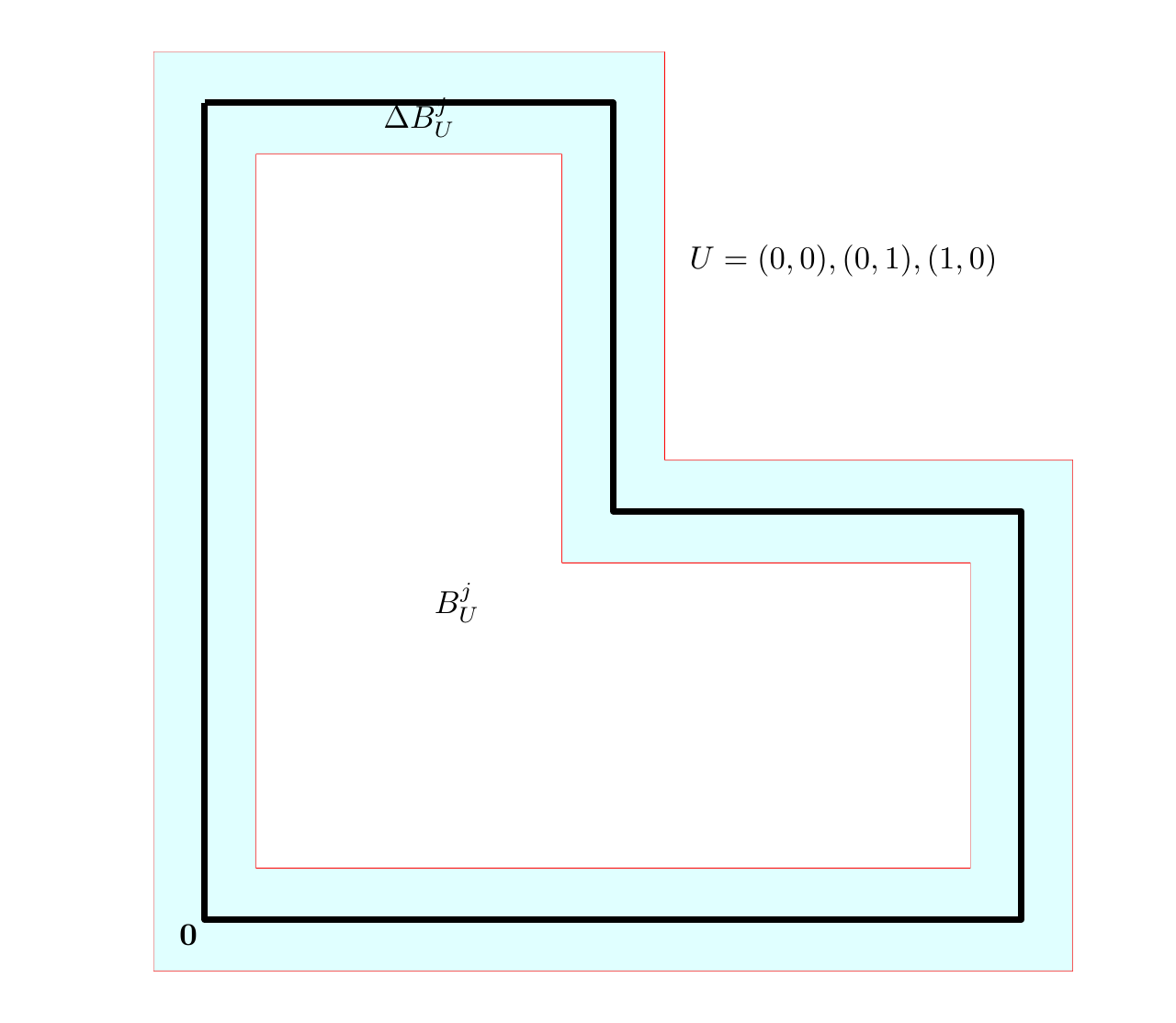}
\end{center}
\caption{Buffer Zones of a Multi-cell}
\label{f:buffer2}
\end{figure}

%\begin{definition}[Shape of lattice animals]
%\label{d:shape}
%Two lattice animals $B^{j}_{U}$ and $B^{j}_{U'}$ at level $j$ of size $k$ are said to be {\bf of the same shape} if there is a translation in $\R^2$ which takes $B^{j}_{U}$ to $B^{j}_{U'}$. It is clear that being of same shape is an equivalence relation. Let $\mathcal{H}_{k}$ denote the set of all equivalence classes, these will be called the {\bf shapes} of lattice animals of size $k$. Observe that $\mathcal{H}_{k}$ is independent of $j$.   
%\end{definition}

\subsection{Recursive construction of blocks I: Forming Ideal Multi-blocks}
In this subsection we start describing how to recursively construct the blocks at levels $j\geq 1$. Suppose that blocks have already been constructed for some $j\geq 0$. Also suppose that the good blocks at level $j$ have been specified. Further assume that other elements of the structure at level $j$ have also been constructed. In particular this means components have been identified with bad and semi-bad components also being specified at level $j$. We now describe how to construct the structure at level $(j+1)$.  Notice that the blocks and good blocks at level $0$ has already been defined. We postpone the precise definitions of components and semi-bad components for the moment.

%\begin{enumerate}
%\item[i.] Each block at level $j$ corresponds to a multi-cell $B^{j}_D$ at level $j$ (for some lattice animal $D$), in the sense that the block contains the interior of $B^{j}_D$ and is contained in the blow up of $B^{j}_D$. The multi-cells corresponding to the blocks will be called {\bf ideal multi-blocks at level $j$}. The size of a block will be the same as the size of the ideal multi-block it corresponds to. 
%
%\item[ii.] Each good block at level $j$ corresponds to a cell at level $j$, i.e., an ideal multiblock correspnding to each good block at level $j$ is of size 1.   
%\end{enumerate}
%
%With this, we describe below how to construct the blocks at level $(j+1)$. Which blocks at level $(j+1)$ are declared good is described later.  
%

\subsubsection{Conjoined Buffer Zones}
%Our first step is to construct the lattice blocks and ideal multi-blocks at level $(j+1)$.
%i.e., partition $\R^2$ into multi-cells at level $j+1$ such that the eventually construced blocks at level $(j+1)$ will correspond to these multi-cells (i.e., will contain the interiors and be contained in the blow-ups of these multi-cells).
Our first step is to construct the lattice blocks and ideal multi-blocks at level $(j+1)$. We start with the following observation. For each $u\in \Z^2$, by recursive construction, there exists a set $H(u)=H^j(u)\subseteq \Z^2$ containing such that $X^j_{H}$ is a component at level $j$.
%the cell $B^j(u)$ belongs to one of the multi-blocks at level $j$, which in turn corresponds to a unique block at level $j$. This level-$j$ block shall be denoted by $\mathbb{B}^{j,\mathbb{X}}(u)$ (we shall drop the superscript $\mathbb{X}$). 

To construct the ideal multiblocks at level $(j+1)$ we start with the following definition.

\begin{definition}[Conjoined buffer zone and Conjoined cells]
\label{d:conjoinedbuffer}
Fix neighbouring vertices $u,u'\in \Z^2$, consider the shared buffer zone denoted by $\Delta B^{j+1}(u,u')$ between cells $B^{j+1}(u)$ and $B^{j+1}(u')$. We call the buffer zone $\Delta B^{j+1}(u,u')$ {\bf conjoined} if one of the following conditions fail.

\begin{enumerate}
\item[\rm i.] Let $T\subseteq \Z^2$ be such that $B^j_{T}=\Delta B^{j+1}(u,u')$. Then we have
$$\#\{t\in T: X^{j}_{H(t)}~\text{is a bad component}\}\leq k_0.$$
%
%Number of ideal blocks $B^{j}(v)$ of level $j$ contained in $\Delta B^{j+1}(u,u')$ such that $\mathbb{B}^{j}(u')$ is {\bf bad} is at most $k_0$. 
That is, the total size of bad level $j$-components contained in the buffer zone is at most $k_0$. 
%\textbf{MAYBE WE NEED COMPONENTS HERE}
\item[\rm ii.] All the bad components contained in the buffer zone are semi-bad.
\end{enumerate}
Call the $(j+1)$-level cells $B^{j+1}(u)$ and $B^{j+1}(u')$ conjoined if $\Delta B^{j+1}(u,u')$ is {conjoined}.
\end{definition}

%\begin{definition}[Conjoined ideal blocks]
%\label{d:conjoined block}
%Let $u,u'\in \Z^2$ be neighbours. Call the ideal block $V^{*,j+1}(u)$ and $V^{*,j+1}(u')$ {\bf conjoined} if the buffer zone they share is conjoined. 
%\end{definition}
%

Using the notion of conjoined cells above we now define the ideal multi-blocks at level $j+1$ with the property that if two cells sharing a conjoined buffer zone are necessarily contained in the same ideal multi-block. More formally we define the following.

\begin{definition}[Lattice Blocks and Ideal multi-blocks at level $j+1$]
\label{d:idealmulti}
Consider the following bond  percolation on $\Z^2$. For $u,u'$ neighbours in $\Z^2$, we keep the edge between $u$ and $u'$ if $B^{j+1}(u)$ and $B^{j+1}(u')$ are conjoined. The connected components of this percolation are called the lattice blocks at level $(j+1)$. For a lattice block $U$ at level $(j+1)$, we call $B^{j+1}_{U}$ an ideal multi-block at level $(j+1)$.
%The ideal multi-blocks at level $j+1$ are the multi-cells $B^{j+1}_{D_{\alpha}}$ corresponding to the lattices animals $U_{\alpha}$.  
%
%For $u\in \Z^2$, Let $C^{j+1}(u)$ denote the component of $u$ in the graph described above. Then the lattice animal $\tilde{B}^{j+1}(u)= \cup_{v\in C^{j+1}(u)} B^{j+1}(v)$ is called an ideal multiblock at level $j$. It is clear that ideal multiblocks form a partition of $\R^2$. 
\end{definition}  

It will follow from our probabilistic estimates that almost surely all lattice blocks are finite. The definition of Ideal multi-blocks is illustrated in Figure \ref{f:conjoin}. The conjoined buffer zones and the ideal multi-blocks are marked in the figure.  

\begin{figure}
\begin{subfigure}{.5\textwidth}
  \centering
  \includegraphics[width=.8\textwidth]{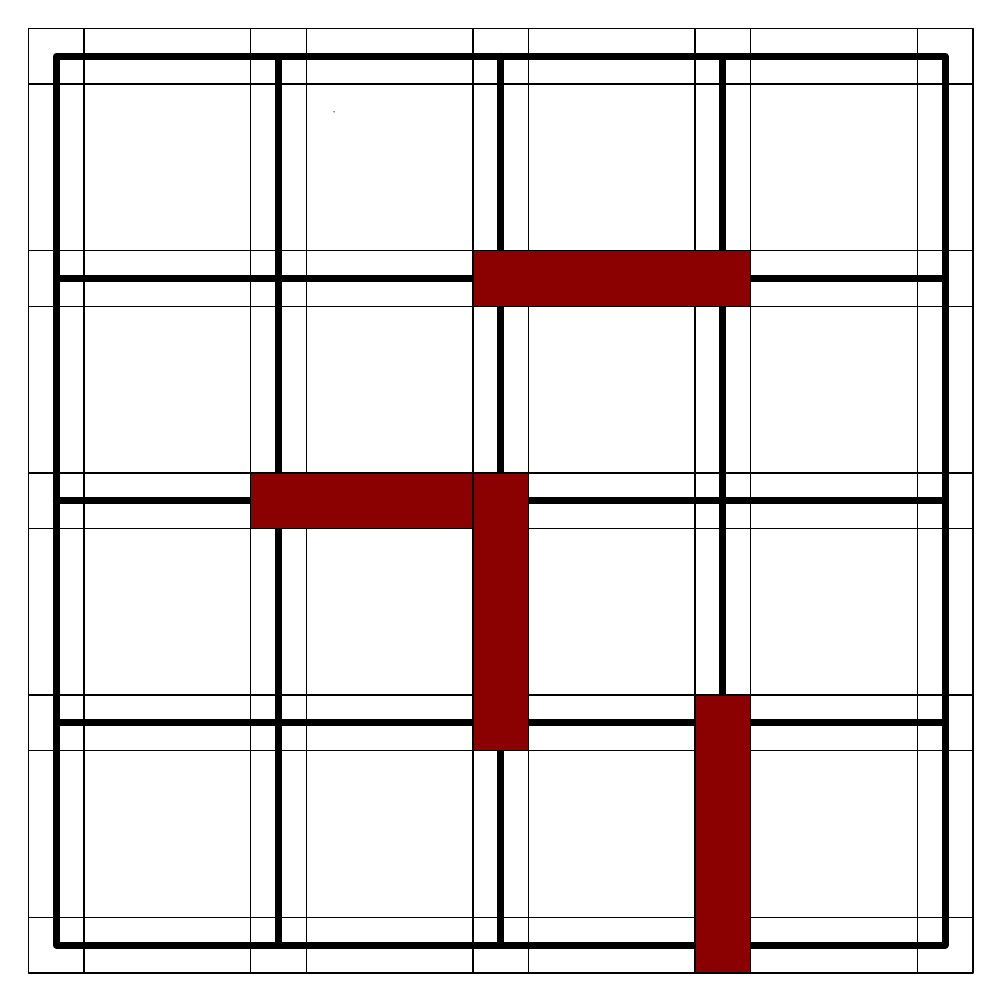}
  \vspace{0.1in}
  \caption{Conjoined buffer zones}
  \label{f:conjoin1}
  %\rule{0ex}{2ex}
\end{subfigure}%
\begin{subfigure}{.5\textwidth}
  \centering
  \includegraphics[width=.8\textwidth]{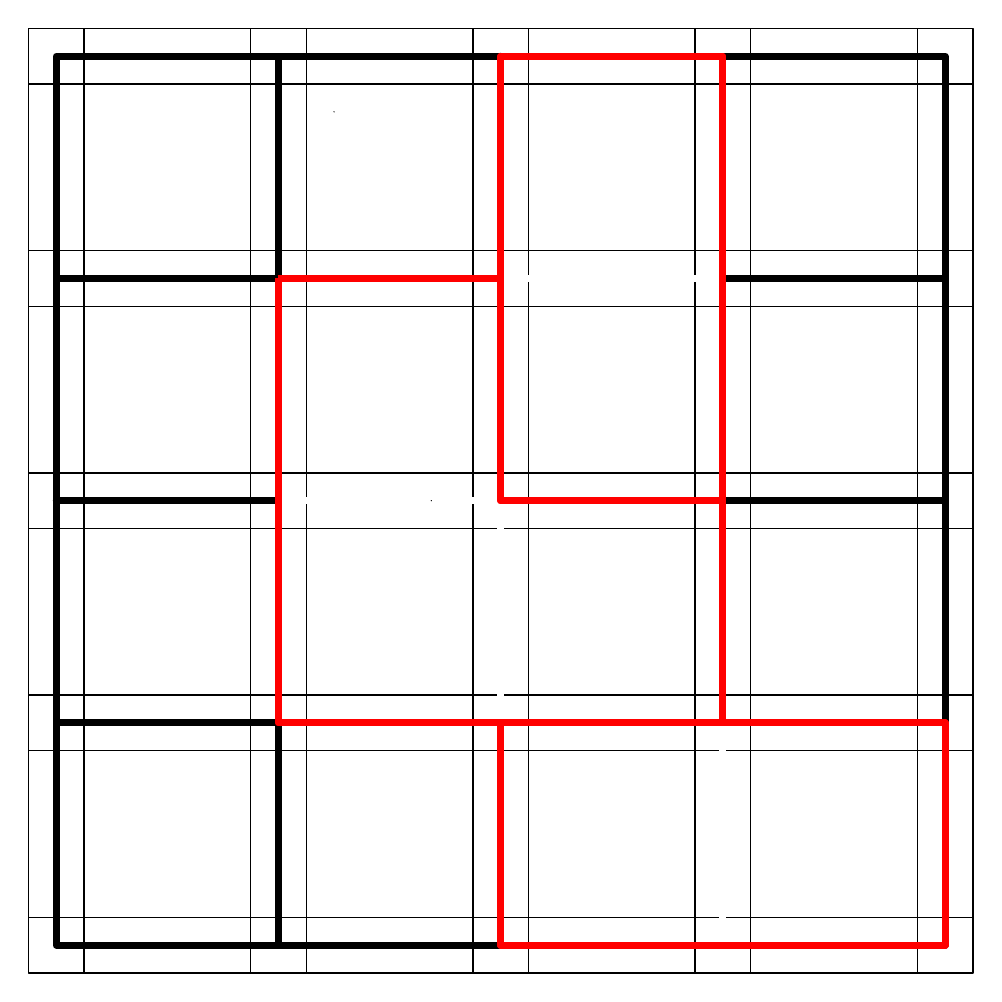}
  \vspace{0.1in}
  \caption{Ideal multi-blocks}
  \label{f:conjoin2}
  %\rule{0ex}{2ex}
\end{subfigure}
\caption{Formation of Ideal Multi-blocks. Ideal multi-blocks of size bigger than $1$ are marked} 
\label{f:conjoin}
\end{figure}

%\marginpar{figure here.}
%By the size of an ideal multi-block at level $(j+1)$ we mean the number of $j+1$ level ideal blocks contained in the multiblock. 

\subsection{Recursive construction of blocks II: Constructing Domains}
Our purpose in this and the next section is the following: we want to show that the boundaries of the ideal multiblocks can be locally randomly perturbed (with the randomness not having long-range dependence) in a bi-Lipschitz manner (with Lipschitz constant $1+O(10^{-j})$) such that the $(j+1)$-level blocks (unions of the $j$-level blocks approximately contained in that region) formed by this construction satisfy the property that it does not contain any $j$-level bad sub-block near the boundary. The property we need further is if there are no bad sub-blocks near the ideal multi-block boundary to start with then with overwhelmingly high probability we do not deform this boundary; also any boundary that can be chosen is chosen with probability that is at most exponentially small in $j$ (see Observation \ref{o:bdcurve} for a precise statement). The way we do it is roughly the following. We construct a large number of such bi-Lipschitz perturbations, that are not too close to one another. Since the total size of the bad sub-blocks in the buffer zone is not large, out of these large number of choices some must satisfy the property of having no bad sub-blocks near the boundary. Now we use some appropriate locally dependent randomness to choose one of the valid choices. The formal construction is technical, and in our opinion, not very illuminating. So the reader, who is willing to believe that such a construction is plausible, might want to skip these two sections for now and go to Observation \ref{o:blockprop} directly.
 
\bigskip
\bigskip
\bigskip

Let $\ch=\ch^{j+1,\mathbb{X}}$ denote the lattice blocks of $\mathbb{X}$ at level $(j+1)$ constructed as above. Clearly $\ch$ is a partition of $\Z^2$ and $\{B^{j+1}_{H}\}_{H\in \ch}$ is a partition of $\R^2$. As alluded to above, the blocks at level $(j+1)$ will be indexed by $\ch$ and will be ``approximations" to the ideal multi-blocks $B^{j+1}_{H}$. To construct the blocks at level $(j+1)$, we first start with constructing domains of blocks which will be some smooth perturbations of the ideal multiblocks $B^{j+1}_{H}$.

\subsubsection{Potential Boundary Curves}
Ideally we would have liked to use the ideal multi-blocks as our blocks at level $(j+1)$, but in that case it is not possible to guarantee that the $j$-level subblocks near the boundary will be good. Hence depending on the distribution of $j$-level subblocks in the buffer zone we would choose boundaries for our blocks. We want the number of possible curves that could serve as boundaries to be limited and hence we first construct a family of curves through buffer zones. 

Let $(\Z^2, E^2)$ denote the usual nearest neighbour lattice on $\Z^2$. The family of curves we construct would be indexed by $\{(\ell_{v}, s_{v}):v\in \Z^2, s_{e}: e\in E^2\}$ where each $\ell_{v}, s_{e}\in [2k_0]$ and each $s_{v}\in \{1,2\}$. Here is the rough meaning of the above indexing. Observe that the buffer zone is union of mutually parallel horizontal and vertical strips, which can be thought of as a fattened version of the graph $(\Z^2, E^2)$. That is, consider the horizontal strips $S^1_{v_1}=\R\times [v_1L_{j+1}-L_j^{5},v_1L_{j+1}+L_j^5]$ for $v_1\in \Z$ and the vertical strips $S^2_{v_2}= [v_2L_{j+1}-L_j^{5},v_2L_{j+1}+L_j^5]\times \R$ for $v_2\in \Z$. So the vertex $v=(v_1,v_2)$ corresponds to the square $S_{v}=S^1_{v_1}\cap S^2_{v_2}$ and an edge would correspond to the rectangle connecting two such squares. Roughly the parameters $\ell_v$ and $s_v$ determine the curve in the square $S_v$ whereas $s_{e}$ determines the curve in the region of the buffer zones corresponding to the edge $e\in E^2$.    

Curves we construct through $S^1_{v_1}$ (say) will be images of the horizontal line $y=v_1L_{j+1}$ under some mild perturbation, and a similar statement is true for vertical strips of buffer zones. Without loss of generality we describe the construction of these maps of $S^1_0$, rest are obtained by translation. Curves through vertical strips are defined similarly.

Let $v=(v_1,0)$. Define points $p_{v,\ell}^{-}=(v_1L_{j+1}-\ell 100^{-(j+5)}L_{j}^5,0)$ and $p_{v,\ell}^{+}=(v_1L_{j+1}+\ell 100^{-(j+5)}L_{j}^5,0)$ for $\ell\in [2k_0]$. Let $T_{\ell,v}$ denote the square whose centre is $(v_1L_{j+1},0)$ and has a side length $2\ell 100^{-(j+5)}L_{j}^5$. Also let $e$ denote the edge between $v$ and $v'=v+(1,0)$. Denote by $T_{\ell_1,\ell_2,e}$ the rectangle $[v_1L_{j+1}+\ell_1 100^{-(j+5)}L_{j}^5,0),(v_1+1)L_{j+1}-\ell_2 100^{-(j+5)}L_{j}^5,0)]\times[-L_j^5/2,L_j^5/2]$. Also let $R^1_{\ell,v}=R_{\ell,v}$ denote the straightline segment in the intersection of $T_{\ell,v}$ and the $x$-axis. Further let $R^1_{\ell_1,\ell_2,e}$ denote the straightline segment in the intersection of $T_{\ell_1,\ell_2,e}$ and the $x$-axis.  

Now suppose we choose $\ell_{v}$ and $\ell_{v'}$ to be corresponding parameters to our curves. Then the curve passes through points $p_1=p_{v,\ell_{v}}^{-}$ and $p_2=p_{v,\ell_v}^{+}$ (and also through points $p_3=p_{v',\ell_{v'}}^{-}$ and $p_4=p_{v',\ell_{v'}}^{+}$). The curve between the points $p_1$ and $p_2$ is determined by the choice of $s_{v}$ and the curve between the points $p_2$ and $p_3$ is determined by the choice of $s_{e}$. Fix $\ell_1$, $\ell_2\in [2k_0]$. Fix functions $F^s_{\ell_1,v}$ for $s\in \{1,2\}$ and $F^s_{\ell_1,\ell_2,e}$ for $s\in [2k_0]$ satisfying the following properties (we shall suppress the subscript $v$ and $e$ in the following):

\begin{enumerate}
\item[\rm i.] $F^s_{\ell}$ (resp.\ $F^s_{\ell_1,\ell_2}$) is a bijection from $T_{\ell}$ (resp.\ $T_{\ell_1,\ell_2}$) to itself.
\item[\rm ii.] $F^{s}_{\ell}$ (resp.\ $F^s_{\ell_1,\ell_2}$) is identity on the boundary of $T_{\ell}$ (resp.\ $T_{\ell_1,\ell_2}$) and is bi-Lipschitz with Lipschitz constant $1+10^{-(j+10)}$.
\item[\rm iii.] For all $\ell_1, \ell_2$ we have $F^{1}_{\ell_1}$ (resp.\ $F^{1}_{\ell_1,\ell_2}$) is the identity map.
\item[\rm iv.] Let $R^1_{\ell}$ (resp.\ $R^1_{\ell_1,\ell_2}$) denote the straight line segment formed by the intersection of the $x$-axis with $T_{\ell}$ (resp.\ $T_{\ell_1,\ell_2}$). We have that $R^2_{\ell}=F^2_{\ell} (R^1_{\ell})$ (resp.\ $R^{s}_{\ell_1,\ell_2}=F^s_{\ell_1,\ell_2} (R^1_{\ell_1,\ell_2})$ for each $s\in [2k_0]\setminus \{1\}$) is contained in the strip $\R\times [- 100^{-(j+6)}L_{j}^5,100^{-(j+6)}L_{j}^5]$.
\item[\rm v.] The $\ell_{\infty}$ distance between $R_{\ell}$ and $R_{\ell'}$ for $\ell \neq \ell'$ (resp.\ between $R^{s}_{\ell_1, \ell_2}$ and $R^{s'}_{\ell_1, \ell_2}$ for $s\neq s'$) is at least $10L_j^4$ on the interval $[p^{+}_{v,\ell_1} +L_j^4,p^{-}_{v+1,\ell_{2}}-L_j^{4}]$.
\end{enumerate}
   
We shall omit the proof of the following basic lemma which easily follows from the fact $L_0$ is sufficiently large and $L_j$ grows doubly exponentially.  

%Observe that the buffer zones at level $(j+1)$ can be written as (non-disjoint) union over mutually parallel hirizontal strips and mutually parallel vertical strips each of width $2L_j^{5}$. We describe how to construct a family of curves through one horizontal and one vertical strip, the family of curves through other strips are constructed by translations. 
%
%Consider the horizontal strip $S^1_0=\R\times [-L_j^{5}, L_j^5]$. Consider the sequence of points $\mathbf{p}^{0}_{-,\ell}= (-\ell 100^{-(j+5)}L_{j}^5,0)$ for $\ell=1,2,\ldots, 2k_0$. Consider the reflection of these points on the $y$-axis defined by $\mathbf{p}^{0}_{+,\ell}= (\ell 100^{-(j+5)}L_{j}^5,0)$. For $u\in \Z$, define the translates of these points by $\mathbf{p}^{u}_{\pm,\ell}=\mathbf{p}^{0}_{\pm,\ell}+(uL_{j+1},0)$.
%
%Our family of curves through $S^1_0$ will be determined as follows. For each $u\in \Z$ we shall pick $\ell_{u}\in [2k_0]$. Then we construct a a family of curves passing through $\mathbf{p}^{u}_{\pm,\ell_u}$ as follows. We shall pick one of the two possible curves between the pairs of points $(\mathbf{p}^{u}_{-,\ell_u},\mathbf{p}^{u}_{-,\ell_u})$ and $(\mathbf{p}^{u}_{+,\ell_u},\mathbf{p}^{u+1}_{-,\ell_{u+1}})$. We define these two curves for $u=0$, for other values of $u$ the curves are defined by translation. We use the following basic lemma about existence of certain maps which we state without proof. See Figure \ref{f:bd1} for an illustration. Notice that this lemma uses crucially that $L_0$ is sufficiently large and $L_j$ grows doubly exponentially.
%
\begin{lemma}
\label{l:bdexistence}
For all $\ell, \ell_1, \ell_2\in [2k_0]$, there exist functions $F^s_{\ell}$ and $F^s_{\ell_1,\ell_2}$ satisfying the properties listed above.
%For $\ell, \ell_1,\ell_2\in [2k_0]$, Let $T_{\ell}=[-\ell 100^{-(j+5)}L_{j}^5, \ell 100^{-(j+5)}L_{j}^5]^2$ and let $T_{\ell_1,\ell_2}=[\ell_1 100^{-(j+5)}L_{j}^5, L_{j+1}-\ell_2 100^{-(j+5)}L_{j}^5]\times [-L_j^5/2, L_j^5/2]$. For all $\ell_1,\ell_2\in [2k_0]$, $i\in \{1,2\}$ and $i'\in [2k_0]$, there exists bijections $F^i_{\ell_1}: T_{\ell_1}\rightarrow T_{\ell_1}$ and $F^{i'}_{\ell_1,\ell_1}: T_{\ell_1,\ell_2}\rightarrow T_{\ell_1,\ell_2}$ satisfying the following properties.
%\begin{enumerate}
%\item[\rm i.] $F^{i}_{\ell}$ (resp.\ $F^i_{\ell_1,\ell_2}$) is identity on the boundary of $T_{\ell}$ (resp.\ $T_{\ell_1,\ell_2}$) and is bi-Lipschitz with Lipschitz constant $1+10^{-(j+10)}$.
%\item[\rm ii.] For all $\ell_1, \ell_2$ we have $F^{1}_{\ell_1}$ (resp.\ $F^{1}_{\ell_1,\ell_2}$) is the identity map.
%\item[\rm iii.] Let $R^1_{\ell}$ (resp.\ $R^1_{\ell_1,\ell_2}$) denote the straight line segment formed by the intersection of the $x$-axis with $T_{\ell}$ (resp.\ $T_{\ell_1,\ell_2}$). We have $R^2_{\ell}=F^2_{\ell} (R_{\ell})$ (resp.\ $R^{s}_{\ell_1,\ell_2}=F^s_{\ell_1,\ell_2} (R^1_{\ell_1,\ell_2})$ for each $s\in [2k_0]\setminus \{1\}$) is contained in the strip $\R\times [- 100^{-(j+6)}L_{j}^5,100^{-(j+6)}L_{j}^5]$.
%\item[\rm iv.] The $\ell_{\infty}$ distance between $R_{\ell}$ and $R_{\ell'}$ for $\ell \neq \ell'$ (resp.\ between $R^{s}_{\ell_1, \ell_2}$ and $R^{s'}_{\ell_1, \ell_2}$ for $s\neq s'$) is at least $10L_j^4$ on the interval $[\mathbf{p}^{u}_{+,\ell_1} +L_j^4,\mathbf{p}^{u+1}_{-,\ell_{2}}-L_j^{4}]$.
%\end{enumerate} 
\end{lemma}

See Figure \ref{f:bd1} for an illustration of the above construction.
\begin{figure}[h]
\begin{center}
\includegraphics[width=\textwidth]{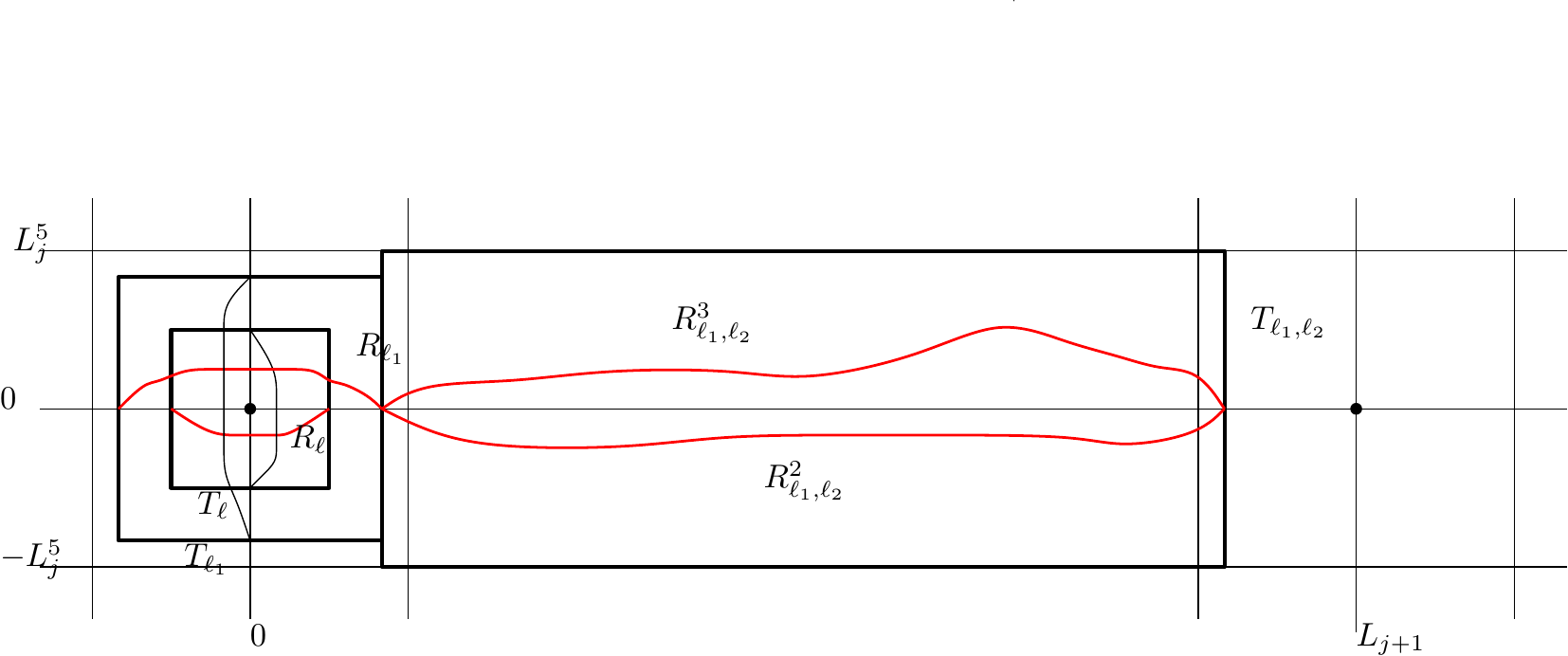}
\end{center}
\caption{Potential boundary curves through a buffer zone}
\label{f:bd1}
\end{figure}

We do similar constructions for vertical strips of buffer zones as well using the same maps $F^s_{\ell,v}$ for the squares $T_{\ell,v}$. Observe the following. For each choice of $\{\ell_{v},s_{v}\}_{v\in \Z^2}$ and $\{s_e\}_{e\in E^2}$ we get one curve contained in each horizontal and vertical buffer zone strip. The family of such curves are called {\bf potential boundary curves}. When we restrict to one buffer zone, the family is called {\bf potential boundary curves} through that buffer zone.

Fix $u\in \Z^2$. Now observe that if we restrict to the buffer zone $\Delta B^{j+1,B}(u)$, then a potential boundary curve through $\Delta B^{j+1,B}(u)$ is determined by $\ell_{u}, \ell_{u'}, s_{u}, s_{u'}$ and $s_{e}$ where $u'=u+(1,0)$ and $e$ is the edge joining $u$ and $u'$, (except at the extremities). In particular, a potential boundary curve through the buffer zone of a cell is determined by choices of $\ell$ and $s$ along the corners and edges of the cell. See Figure \ref{f:bd2}. 
%Similar definitions can be made for multi-cells as well.
  
%Observe that for each horizontal strip of buffer zone $S^1_{v}=\R\times [uL_{j+1}-L_j^{5}, uL_{j+1}+ L_j^5], v\in \Z$, the above lemma gives rise to a family of curves $\mathcal{C}^1_{v}$ parametrised by  $\{\ell_{u}\in [2k_0], s_{u}\in [2], s_{u,u+1}\in [2k_0],u\in \Z\}$, where the curve $C^1(\{\ell_u\}, \{s_{u}\}, \{s_{u,u+1}\},u\in \Z)\in \mathcal{C}^1_{v}$ is defined as follows. Without loss of generality take $v=0$. Define a curve passing through points $\mathbf{p}^{u}_{\pm,\ell_u}$ by taking the curve between $(\mathbf{p}^{u}_{-,\ell_u},\mathbf{p}^{u}_{-,\ell_u})$ an appropriate translate of $R^{s_{u}}(\ell_u)$ and the curve between $(\mathbf{p}^{u}_{+,\ell_u},\mathbf{p}^{u+1}_{-,\ell_{u+1}})$ an appropriate translate of $R^{s_{u,u+1}}(\ell_u,\ell_{u+1})$. The family of curves $\mathcal{C}^1_{v}$ is called the set of {\bf potential boundary curves through the strip} $S^1_{v}$.
%
%In a similar vein, we construct the family of potential boundary curves 
%$$\mathcal{C}^2_{v}=\{C^2(\{\ell_u\}, \{s_{u}\}, \{s_{u,u+1}\},u\in \Z):\ell_{u}\in [2k_0], s_{u}\in [2], s_{u,u+1}\in [2k_0]\}$$
%through each vertical strip of buffer zone $S^2_{v}= [vL_{j+1}-L_j^{5}, vL_{j+1}+ L_j^5]\times \R, v\in \Z$. Observe that we use the same bi-Lipschitz bijection from $T_{\ell}$ to itself while constructing boundary curves through horizontal and vertical buffer zones.  

\begin{definition}[Potential Boundary Curves of a multi-cell and Potential Domains] 
Fix a multi-cell $B^{j+1}_{U}$ at level $j+1$. Each choice of potential boundary curves through each of the outer buffer zones of $B^{j+1}_{U}$  determines a simple closed curve $C$ through the buffer zone of $B^{j+1}_{U}$. These curves are called the potential boundary curves of the multi-cell $B^{j+1}_{U}$. The region surrounded by $C$ is called a potential domain of the multi-cell $B^{j+1}_{U}$. 
\end{definition}   

It follows from the construction, that the number of potential boundary curves of the multi-cell $B^{j+1}_{U}$ is at most $(8k_0)^{16|U|k_0^2}$.

\begin{figure}
\begin{subfigure}{.5\textwidth}
  \centering
  \includegraphics[width=.8\linewidth]{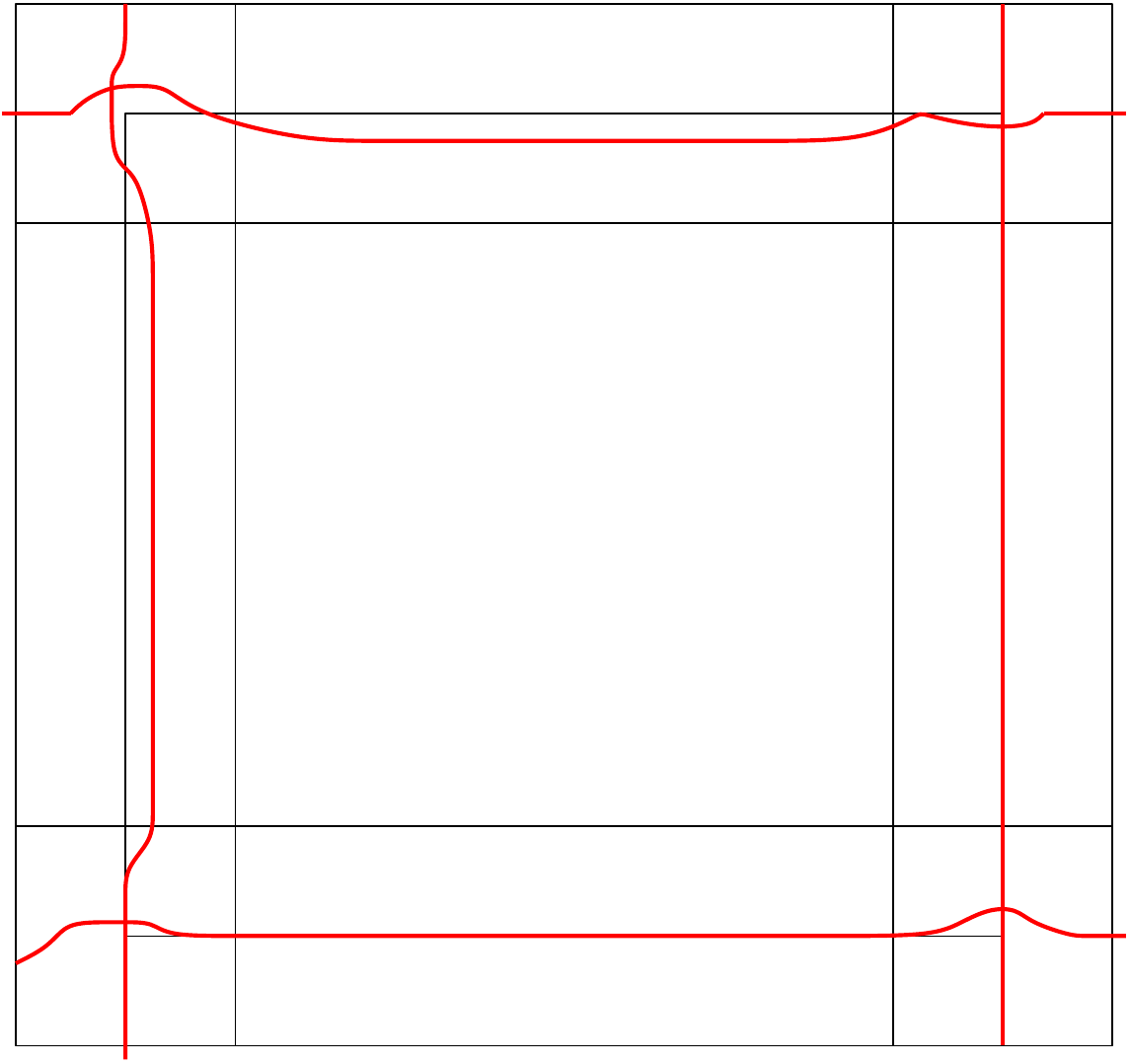}
  %\caption{Conjoined buffer zones}
  \label{f:potbd2}
  %\rule{0ex}{2ex}
\end{subfigure}%
\begin{subfigure}{.5\textwidth}
  \centering
  \includegraphics[width=.8\linewidth]{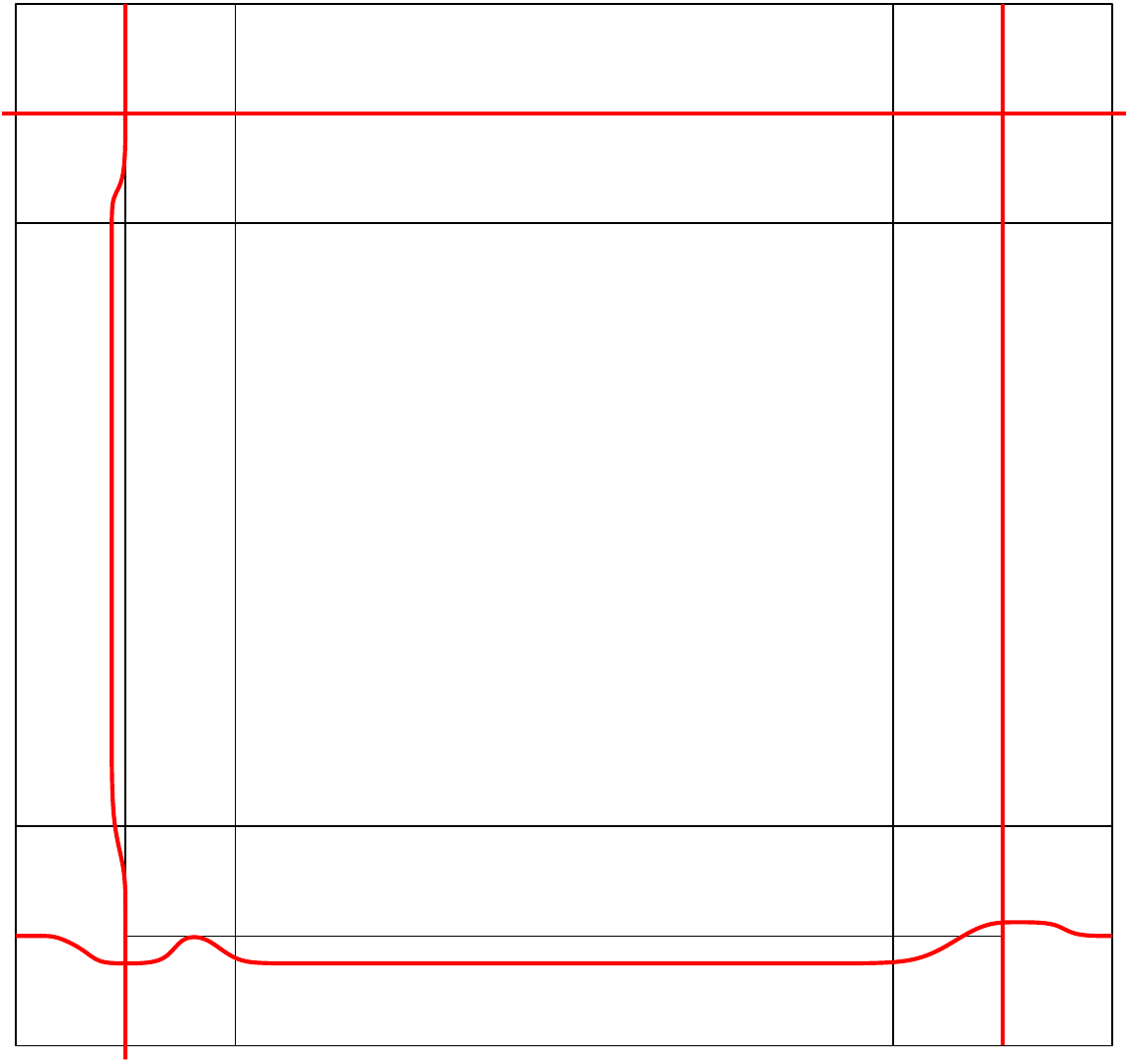}
  %\caption{Ideal multi-blocks}
  \label{f:potbd3}
  %\rule{0ex}{2ex}
\end{subfigure}
\caption{Two choices of potential boundary curves of a (multi) cell of size $1$} 
\label{f:bd2}
\end{figure}

It is clear from our construction that associated with each potential boundary curve there is a unique bijection from $\R^2$ to itself which is bi-Lipschitz with Lipschitz constant $(1+10^{-(j+5)})$.  Let $F$ denote such a map. Then for all multi-cell $B^{j+1}_{U}$, $F(\partial B^{j+1}_{U})$ is the potential boundary curve through the buffer zone of $B^{j+1}_{U}$ induced by the potential boundary curve corresponding to $F$
That is, potential boundary curves are small perturbations of the boundaries of multi-cells. We make a formal definition for this.

\begin{definition}[Canonical Maps]
\label{d:canonical}
For a multi-cell $B^{j+1}_{U}$ and for any potential boundary curve $C$ through $\Delta B^{j+1}_{U}$, there exists a unique bi-Lipschitz map $F=F_C$ on $B^{j+1}_{U}$ with Lipschitz constant $(1+10^{-(j+5)})$ such that $F(\partial B^{j+1}_{U})=C$. These maps and their inverses are called canonical maps. That is, a canonical map is a map that transforms a multi-cell $B^{j+1}_{U}$ to a potential domain $U_{C}$  and vice-versa. Observe also that the family of canonical maps only depend on the shape of $U$ upto translation. For two potential boundary curves $C_1,C_2$ of the multi-cell $B^{j+1}_{U}$, the maps $F_{C_2}\circ F_{C_1}^{-1}$ from $U_{C_1}$ to $U_{C_2}$ are also called canonical maps.   
\end{definition}

\subsubsection{Valid Boundary Curves and Domains} 
Recall that we have already constructed the ideal multi-blocks at level $(j+1)$. Our next order of business is to stochastically choose one boundary curve through the outer buffer zones of each ideal multi-block satisfying certain conditions. This curve will be called the boundary curve at level $(j+1)$ and the potential domain corresponding to this choice of boundary will be called domain. Since the outer buffer zones of ideal multi-blocks are not conjoined, the choice of a boundary curve through these boils down to choosing $\{(\ell_{v},s_{v})\}_{v\in V^*}$ and $\{s_{e}\}_{e\in E^*}$. Here $V^*\subseteq \Z^2$ is the set of all vertices corresponding to the squares (intersection of a horizontal and a vertical buffer zone)  such that not all of the four buffer zones intersecting at that square are conjoined and $E^*$ denotes the edges in $E^2$ that correspond to non-conjoined buffer zones.  

Recall that we want to choose our boundaries so that they are away from the $j$ level bad components. To this end we restrict our choices to {\bf valid} boundary curves defined below. 

For $v\in V^*$, we call $(\ell_v,s_{v})$ (where $\ell_{v}\in [2k_0]$ and $s_{v}\in [2]$) {\bf valid} if there does not exist any bad $j$ level component within distance $10L_j^4$ of the boundary of $T_{v,\ell}$ and $F^{s}(R^*)$, where $R^*$ is the intersection of the boundaries of $(j+1)$-level cells with $T_{v,\ell}$.

Let $e$ be the edge connecting neighbouring vertices $v,v'\in V^*$. For a valid choice of $(\ell_{v},s_{v})$ and $(\ell_{v'},s_{v'})$ we call $(\ell_v,s_{v},\ell_{v'},s_{v'}, s_{e})$ {\bf valid} if $R^{s}_{\ell_1,\ell_2,e}$ does not have any $j$ level bad-component within distance $L^4$ of it. 

The following observation is immediate from the definition of conjoined block.

\begin{observation}
\label{o:validexist}
For all $v\in V^*$, there exist valid choices of $(\ell_{v},s_{v})$. Also for all $e=(v,v')\in E^*$, and for all valid choices of $(\ell_{v},s_{v})$ and $(\ell_{v'},s_{v'})$ there exist $s_{e}$ such that $(\ell_v,s_{v},\ell_{v'},s_{v'}, s_{e})$ is valid.
\end{observation}

Given $V^*$ and $E^*$, we choose a valid boundary curve randomly independently of everything else as follows. 

\begin{enumerate}
\item[$\bullet$] For each $v\in V^*$, choose a valid $(\ell_v,s_{v})$.
\item[$\bullet$] If there exist valid $(\ell_v, s_{v})$ with $s_{v}=1$ choose one such with probability at least $(1-10^{-(j+10)})$.
\item[$\bullet$] For $e=(v,v')\in E^*$, choose $s_{e}$ such that $(\ell_v,s_{v},\ell_{v'},s_{v'}, s_{e})$ is valid.
\item[$\bullet$] If $s_{e}=1$ leads to a valid choice, then choose it with probability at least $(1-10^{-(j+10)})$.
\item[$\bullet$] The probability of each valid choice must be at least $(8k_0)^{-4k_0^2}100^{-(j+10)}$. 
\end{enumerate}

This choice leads to a boundary curve, which we shall call the boundary curve at level $(j+1)$. The following important properties of the boundary curve as chosen above is easy to see and recorded as an observation for easy reference.  

\begin{observation}[Domains]
\label{o:bdcurve}
Let $\ch=\ch^{j+1}$ denote the set of lattice blocks of $\mathbb{X}$ at level $(j+1)$. The boundary curve partitions $\R^2$ (in a weak sense) into closed connected regions $\{\hat{U}_{X}\}_{U\in \ch}$, called domains, which have the following properties.
\begin{enumerate}
\item[\rm i.] For each $U\in \ch$, $\hat{U}_{X}$ contains the interior of the ideal multi-block $B^{j+1}_{U}$ and is contained in the blow-up of the $B^{j+1}_{U}$.
\item[\rm ii.] Given $\ch$, for $U_1,U_2\in \ch$ such that $U_1$ and $U_2$ are non-neighbouring, the choice of $\hat{U_1}_{X}$ and $\hat{U_2}_{X}$ are independent.
\item[\rm iii.] There is a canonical map $F$, which is a bi-Lipschitz bijection from $\R^2\to \R^2$ with Lipschitz constant $(1+10^{-(j+5)})$ such that $F(B^{j+1}_{U})=U_{X}$ for all $U\in \ch$, and such that $F$ is identity everywhere except near the boundaries of ideal multi-blocks.
\item[\rm iv.] There are no $j$-level bad components near the boundaries of the domains.
\item[\rm v.] If there are no bad $j$ level component in the buffer zone of the ideal multi-block $B^{j+1}_{U}$, then with probability at least $(1-10^{-(j+10)})^{4|U|}$, the canonical map $F$ is identity on $B^{j+1}_{U}$.     
\end{enumerate}  
\end{observation}

See Figure \ref{f:domain} for an illustration of domain constructions. Bad level $j$ components are marked in red.

\begin{figure}[h!]
\begin{center}
\includegraphics[width=0.6\textwidth]{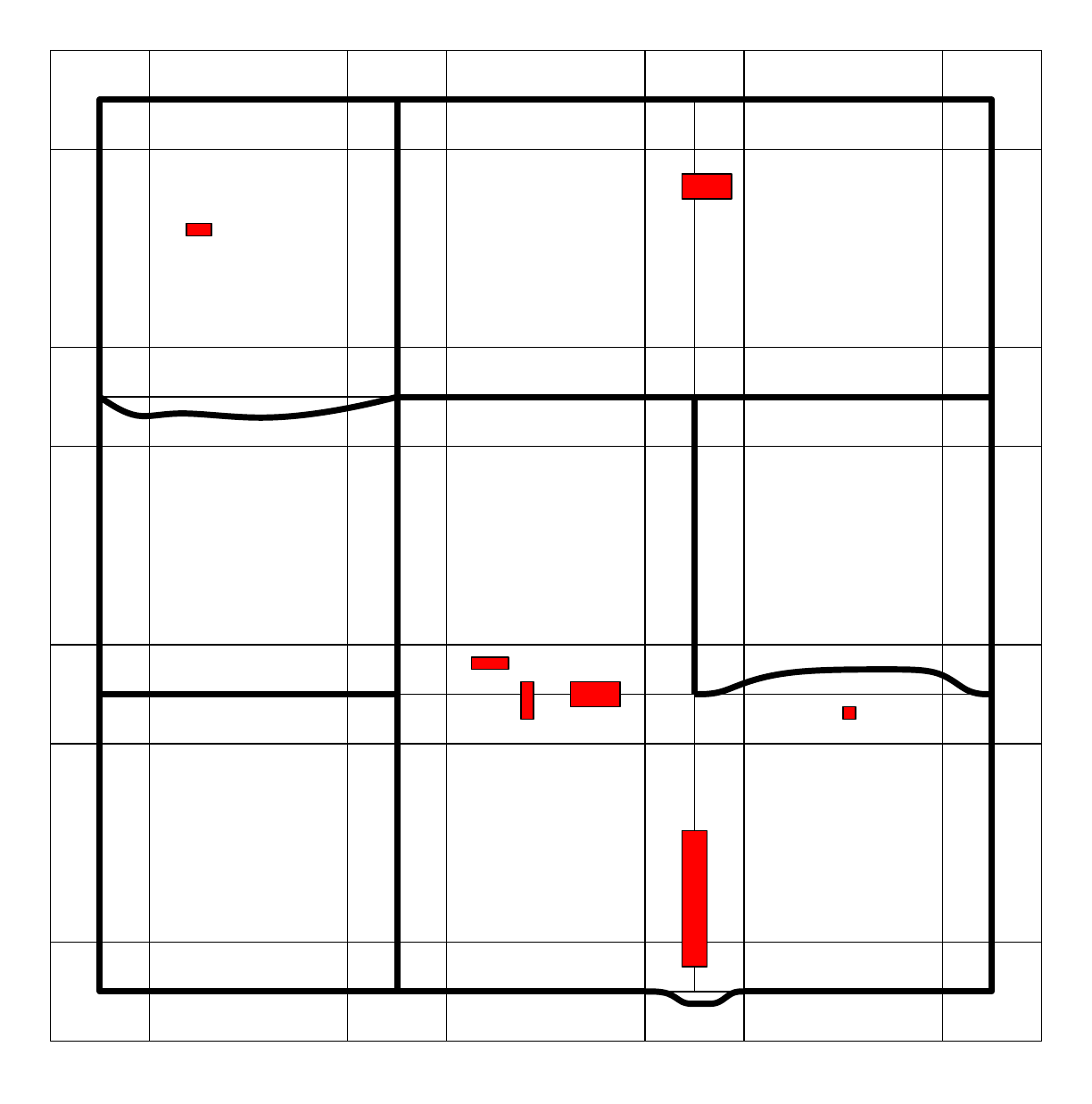}
\end{center}
\caption{Domains at level $j+1$}
\label{f:domain}
\end{figure}

\subsection{Recursive Construction of Blocks III: Forming Blocks out of Domains}

Notice that we have constructed the domains in such a way that boundaries of domains at level $(j+1)$ avoid the bad components at level $j$. However observe also that domains at level $j+1$ are not necessarily unions of blocks at level $j$. This is why we cannot use domains as blocks themselves and have to do one more level of approximation. 

Let $\{\hat{U}_{X}\}_{U\in \ch}$ denote the set of domains of $\mathbb{X}$ at level $(j+1)$. Define $\tilde{U}\subseteq \Z^2$ to be the set of all vertices $u$ of $\Z^2$ such that the $j$-level cell $B^{j}(u)$ is contained in $\hat{U}_{X}$ or the north east corner of $B^{j}(u)$ is contained in $\hat{U}_{X}$. Then define the block at level $j+1$ corresponding to the lattice block $U$, denoted by $X^{j+1}_{U}$ to be equal to $X^{j}_{\tilde{U}}$. Notice that this is well defined because by construction $\tilde{U}$ is a union of lattice blocks for $\mathbb{X}$ at level $j$. The set of all blocks at level $(j+1)$ is $\{X^{j+1}_{U}:U\in \ch^{j+1}\}$. Notice that blocks at level $(j+1)$ are union of blocks at level $j$ with none of the $j$ level bad subcomponents close to the boundary of the $(j+1)$ level blocks. Suppose $V\subseteq \Z^2$ is such that $X^j_{V}$ is a bad component at level $j$. Then the distance of $V$ from the boundary of $\tilde{U}$ is at least $L_j^3$. We record some useful properties of the blocks in the following observation.

\begin{observation}[Properties of Blocks]
\label{o:blockprop}
The blocks constructed as above satisfy the following conditions.
\begin{enumerate}
\item[\rm i.] Each block corresponds to a unique ideal multi-block, contains its interior and is contained in its blow-up.
\item[\rm ii.] The distance between any bad $j$-level subblock contained in a $j+1$-level block and its boundary is at least $L_j^3$ level $j$ cells. 
\item[\rm iii.] Suppose $B^{j+1}_{H}$ and $B^{j+1}_{H'}$ are two multi-cells that do not share a buffer zone. Condition on the event $\ce=\ce(H,H')$ that none of the external buffer zones of $B^{j+1}_{H}$ and $B^{j+1}_{H'}$ are conjoined. Clearly, on $\ce$ we have that $H$ and $H'$ are both unions of lattice blocks for $\mathbb{X}$ at level $(j+1)$. Then conditioned on $\ce$, we have that $\{X^{j+1}_{H}\}$ and $\{X^{j+1}_{H'}\}$ are independent.
\end{enumerate}
\end{observation}

\subsection{Geometry of a block: Components}
To complete the description of block construction, it remains to define good blocks at level $j\geq 1$. Before we give the recursive definition of the good blocks, it is necessary to introduce certain definitions and notations regarding the geometry of the multi-blocks. 

\subsubsection{Bad Components of blocks}
Fix $j\geq 0$. Suppose that blocks and good blocks are already defined up to level $j$. Recall that good blocks at level $j$ always correspond to lattice blocks of size $1$. Let $\mathcal{H}=\{H(u)\}_{u\in \Z^2}$ denote the family of lattice blocks at level $j$, i.e., $H(u)$ denotes the lattice block containing $u$. Our objective is to group the neighbouring bad blocks together. To this end we make the following definition.
%
%corresponding to the ideal multiblock containing the cell $B^j(u)=B(u)$. More generally for any lattice animal $U$ $X_U=X^{j}_U$ will denote the collection of $j$-level blocks corresponding to the smallest collection of ideal multi-blocks containing $B^j_U$.
%As mentioned before the good blocks at level $j$ will always correspond to lattice blocks of size $1$. Suppose the multiblocks of level $j$ has been formed and the good blocks have also been defined. The good blocks, as mentioned before, will always be multiblocks of size 1. 

\begin{definition}[Lattice Components]
\label{d:latticecomp}
Let $\mathcal{Q}=\{Q(u)\}_{u\in \Z^2}$ be the family of subsets having the following properties.
\begin{enumerate}
\item[\rm i.] $Q(u)= \cup_{v\in Q(u)} H(v)$, i.e., elements of $\mathcal{Q}$ form a partition of $\Z^2$, where each element is a union of lattice blocks.
\item[\rm ii.] If $|Q(u)|>1$, then $Q(u)$ must contain $H$ such that $X^j_{H}$ is a bad block at level $j$.
\item[\rm iii.] If $|Q(u)|>1$ or if $X^j_{Q(u)}$ is a bad block at level $j$, then for all neighbours $v$ of $Q(u)$ in the closed packed lattice of $\Z^2$, $X^j_v=X^j_{\{v\}}$ is a good block at level $j$.
%For all $v\in \Z^2$, $v\notin Q(u)$ and such that $v$ is neighbouring to some vertex of $Q(u)$ in the closed packed lattice of $\Z^2$, we have that $X^j_{v}$ is a good block at level $j$. 
%%
% of $C(u)$ (i.e., all vertices outside $C(u)$ and neighbouring to some vertex of $C(u)$) in the close packed lattice of $\Z^2$, correspond to good $j$ level blocks.
\item[\rm iv.] If there are are vertices $v,v'\in Q(u)$ which are not neighbours in the usual Euclidean lattice but neighbours in the close packed lattice of $\Z^2$, then the $2\times 2$ square containing $v$ and $v'$ is also contained in $Q(u)$.
\item[\rm v.] The family $\{Q(u)\}_{u\in \Z^2}$ is the maximal family having properties ${\rm i.-iv.}$ above, i.e., any other family having the same properties must consist of unions of elements of $\mathcal{Q}$.
\end{enumerate}
Elements of $\mathcal{Q}$ are called {\bf lattice components} at level $j$.
\end{definition}
%We record the trivial properties of the sets $Q(u)$ in the following observation. 
%
%\begin{observation}
%\label{o:components}
%Let $\{Q(u):u\in \Z^2\}$ denote the collection of the sets defined as above (we suppress the superscript $j$). Then
%\begin{enumerate}
%\item[\rm ii.] Fix $u\in \Z^2$. If for all $v$ such that $v$ is a neighbour of $u$ in the closed packed lattice of $\Z^2$ or $v=u$ we have $X^j_{v}$ is a good block at level $j$, we have that $Q(u)=\{u\}$. 
%\item[\rm iii.] For each $u$, $Q(u)$ is a union of lattice blocks at level $j$, and if $|Q(u)|> 1$, then $X^j_{Q(u)}$ must contain at least $1$ bad block. More generally, If $|Q(u)|\geq k>1$, then the total size of bad blocks contained in $X_{Q(u)}$ must be at least $\frac{k}{9}$.   
%\end{enumerate} 
%\end{observation}

It is easy to see that $\mathcal{Q}$ is well defined. For $Q\in \mathcal{Q}$, we call $X^j_{Q}$ a {\bf component} of $\mathbb{X}$ at level $j$.   
%Let $\{C_{\alpha}\}=\{C^{j}_{\alpha}\}$ be the disjoint collection of $C(u)$'s as described above. Then $X_{C_{\alpha}}$ are called the {\bf components} (of $\mathbb{X}$) at level $j$. 
Notice that a component is always a union of blocks at level $j$. We call $X^j_{Q}$ a {\bf bad component} at level $j$ if it contains a bad block at level $j$. Often we shall denote the component $X^j_{Q(u)}$ by $X^{*,j}(u)$. The following observation is easy but useful.

\begin{observation}
\label{o:components}
If $|Q(u)|\geq k>1$, there exists $Q^*\subseteq Q(u)$ with $|Q^*|\geq \lceil \frac{k}{25}\rceil$ such that elements of $Q^*$ are non-neighbouring and for all $v\in Q^*$, $X^j_{H(v)}$ is a bad block at level $j$.
\end{observation}

Notice that once we know the blocks at level $j$, and also know which blocks at level $j$ are good, we can work out what the components at level $j$ are, as the components only depend on the geometry of locations of the bad blocks at level $j$ and not on the anatomy of the blocks themselves. See Figure \ref{f:component} for an illustration. The bad blocks are marked as well as the boundary of the components.

\begin{figure}[h]
\begin{center}
\includegraphics[width=0.6\textwidth]{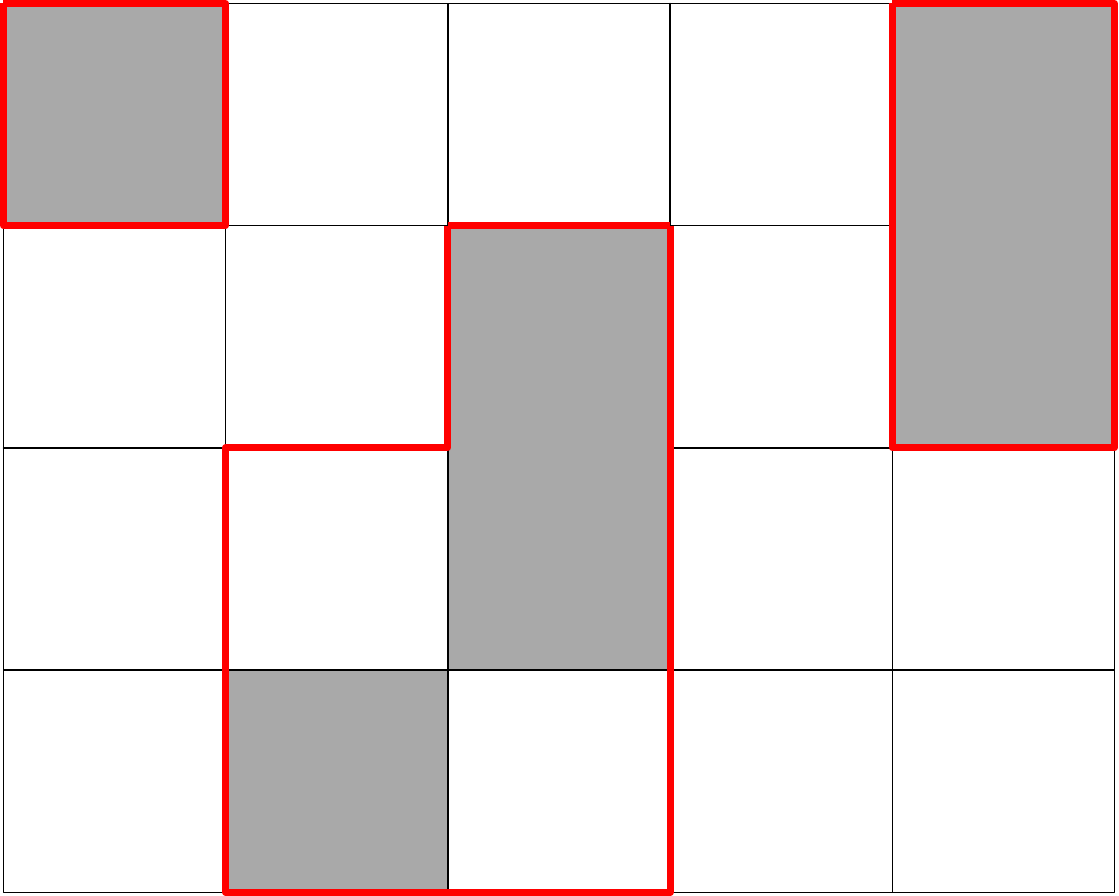}
\end{center}
\caption{Blocks and Components: Bad blocks are marked in gray, Boundaries of components are also marked}
\label{f:component}
\end{figure}

\subsubsection{Sub-blocks and Subcomponents}
Let $X^j_{U}$ be a $j$-level block or component. Then $|U|$ shall denote the size of the block/component. Now suppose $j\geq 1$. Let $U'\subset \Z^2$ be such that $X^{j-1}_{U'}=X^{j}_{U}$. For $V\subset U'$ such that $V$ is a lattice block (resp.\ lattice component at level $(j-1)$) we call $X^{j-1}_{V}$ a sub-block (resp.\ sub-component) at level $(j-1)$ of the the $j$-level block/component $X^j_{U}$. Notice that by construction we have that all bad subcomponents are away from the boundary of the component $X^j_{U}$. 
%
%Let $\{X^*_{\alpha}\}_{\alpha\in I}$ denote the collection of $j-1$ level components contained in $X_{U}$. These will be called the $(j-1)$ level sub-components of $X_U$. The $(j-1)$ level sub-blolcks of $X_U$ are defined similarly. Observe that by definition of a block/component there are no level $j-1$ bad sub-components near the boundary of $X_{U}$.    
%
%
%
%Set $u\sim u'$ if $\mathcal{H}(u)$ and $\mathcal{H}(u')$ are neighbouring in the closed packed lattice of $\Z^2$ and both. Now look at the component of $\mathcal{H}(u)$ in this graph. Let the lattice animal corresponding to this be denoted by $\tilde{\mathcal{H}}(u)$. Then the collection of $j$ level multiblocks corresponding to the collection of $j$ level ideal multiblocks $\tilde{\mathcal{H}}(u)$ is defined as the bad component of the $j$ level subblock $X^{j}(u)$. We shall denote it by $X^{*,j}(u)$, and often we will drop the superscript $j$ when it will be clear from the context that we are talking about $j$ level blocks.

%\textbf{Subblocks:}
%Let $j\geq 1$. Let $X^{j}(u)$ be a $j$ level multiblock. Let $\mathcal{H}$ be the level $j$ lattice animal corresponding to this block. Let $\mathcal{H}^*$ denote the $j-1$ level lattice animal corresponding to $\mathcal{H}$. Let $\{\mathcal{H}^*_{i}\}_{i\in I}$ denote the set of level $j-1$ lattice animals that corresponds to ideal $j-1$ level multiblocks in $\mathcal{H}^*$, i.e., $\cup_{i}\mathcal{H}^*_{i}=\mathcal{H}^*$. We call $X^{j-1}_{\mathcal{H}^*_i}$ to be the $j-1$ level subblocks of of the $j$ level block $X^j_{\mathcal{H}}$.
%
%\textbf{Bad components:} 

\subsection{Embedding, Embedding Probabilities and Semi-bad components}
\label{s:dembedh}
Now we need to make a recursive definition of embedding and define embedding probabilities for blocks and components at level $j+1$. We start with defining embedding at level $0$.
\subsubsection{Embedding at level $0$}
For $v,v'\in \Z^2$, suppose $X^0_{v}$ and $Y^0_{v'}$ are blocks at level $0$. We call $Y^0_{v'}\in \mathbf{0}$ if $Y^0_{v'}$ is not good and $Y^0_{v'}$ contains more $0$'s than $1$'s. Similarly $Y^0_{v'}\in \mathbf{1}$ if $Y^0_{v'}$ is not good and $Y^0_{v'}$ contains at least as many $1$'s as $0$'s.

For $v,v'\in \Z^2$, we call $X^0_v$ embeds into $Y^0_{v'}$, denoted $X^0_v\hookrightarrow Y^0_{v'}$, if one of the following three conditions hold.
\begin{enumerate}
\item[\rm i.] $Y^0_{v'}$ is a good block at level $0$.
\item[\rm ii.] $X_{\iota +v}=0$ and $Y^0_{v'}\in \mathbf{0}$.
\item[\rm iii.] $X_{\iota+v}=1$ and $Y^0_{v'}\in \mathbf{1}$.
\end{enumerate}

Let $U$ and $U'$ be lattice animals of the same shape. Let $h:\Z^2\rightarrow \Z^2$ denote the translation that sends $U$ to $U'$. Then we say $X^0_{U}\hookrightarrow Y^0_{U'}$ if $X_{u}\hookrightarrow Y_{h(u)}$ for all $u\in U$.

Notice that at level $0$, the component $X^0_U$ always corresponds to the ideal multi-block $B^0_{U}$. This is no longer true for $j\geq 1$ as the boundaries can have different shapes. So we need to make a more complicated recursive definition at levels $j\geq 1$.

\subsubsection{Embedding at higher levels}
Fix $j\geq 1$. Suppose $U$ is a union of lattice blocks for $\mathbb{X}$ at level $j$. Suppose also that $V\subseteq \Z^2$ is a union of lattice blocks for $Y$ at level $j$. Suppose further that $U$ and $V$ have the same shape. We want to define an event $X^j_{U}$ embeds into $Y^j_{V}$, denoted by $X^j_{U}\hookrightarrow Y^j_{V}$. 

Modulo a translation from $\R^2\to \R^2$ that takes $B^j_{U}$ to $B^j_{V}$, we can assume that $U=V$. Define the domain of $X^j_{U}$ to be the union of the domains of the $j$ level blocks contained in $X^j_{U}$, denote it by $\hat{U}_{X}$. Define $\hat{U}_{Y}$, the domain of $Y^j_{U}$, in a similar manner. To define the embedding we need to define bi-Lipschitz maps that take $\hat{U}_{X}$ to $\hat{U}_{Y}$. Notice that we already have one such candidate map, namely the canonical map that takes $\hat{U}_{X}$ to $\hat{U}_{Y}$. We shall consider small perturbations of that map.

\begin{definition}[$\alpha$-canonical maps]
\label{d:alphacano}
Let $X^{j}_{U}$, $Y^j_{U}$, $\hat{U}_{X}$, $\hat{U}_{Y}$ be as above. Let $T_1,T_2, \ldots ,T_{k}\subseteq \Z^2$ be such that $X^{(j-1)}_{T_{1}}, \ldots ,X^{(j-1)}_{T_{k}}$ are unions of blocks of $\mathbb{X}$ at level $(j-1)$ with domains $\hat{T}_{i,X}$ for $i\in [k]$. Similarly let $T'_1,T'_2,\ldots ,T'_{k'}$ be such that $Y^{(j-1)}_{T'_{1}}, \ldots ,Y^{(j-1)}_{T'_{k}}$ are unions of blocks of $\mathbb{Y}$ at level $(j-1)$ with domains $\hat{T'}_{i,Y}$ for $i\in [k']$. Let $F$ be the canonical map from $\hat{U}_{X}$ to $\hat{U}_{Y}$. Then we call $G_{\theta}=\theta \circ F$ to be an $\alpha$-canonical map from $\hat{U}_{X}$ to $\hat{U}_{Y}$ (with respect to $\mathcal{T}=\{\hat{T}_1,\hat{T}_2,\ldots ,\hat{T}_k\}$ and $\mathcal{T}'=\{\hat{T'}_1,\ldots \hat{T'}_{k'}\}$) if the following conditions are satisfied.
\begin{enumerate}
\item[\rm i.] $\theta$ is a bijection from $\hat{U}_{Y}$ to itself that is identity on the boundary of $\hat{U}_{Y}$ and is bi-Lipschitz with Lipschitz constant $(1+10^{-(j+10)})$.
\item[\rm ii.] There exists $\{S_{i}:i\in [k]\}$ (resp.\ $\{S'_{i}:i\in [k']\}$) such that $S_i$ has the same shape as $T_{i}$ (resp.\ $S'_{i}$ has the same shape as $T'_{i}$) such that $Y^{j-1}_{S_i}$ is a union of $j-1$  level blocks of $\mathbb{Y}$ with domain $\hat{S}_{i,Y}$ (resp.\ $X^{j-1}_{S'_i}$ is a union of $j-1$  level blocks of $\mathbb{X}$ with domain $\hat{S'}_{i,X}$) such that $G_{\theta}(\hat{T}_{i,X})=\hat{S}_{i,Y}$ for all $i\in [k]$ and $G_{\theta}(\hat{S'}_{i,X})=\hat{T'}_{i,Y}$ for all $i\in [k']$.   
\item[\rm iii.]  $G_{\theta}$ restricted to $hat{T}_{i}$ (resp.\ $\hat{S'}_i$) coincides with the canonical map from $\hat{T}_{i,X}$ to $\hat{S}_{i,Y}$ (resp.\ from $\hat{S}'_{i,X}$ to $\hat{T}'_{i,Y}$).
\end{enumerate}
\end{definition}

Notice that an $\alpha$-canonical map by definition is a bi-Lipschitz map with Lipschitz constant $(1+10^{-(j+5)})$. For the sake of notational convenience we shall often denote this as an $\alpha$-canonical map with respect to $\{T_1, T_2,\ldots, T_{k}\}$ and $\{T'_1, T'_2,\ldots, T'_{k}\}$, but it will always be implicitly assumed that we know domains of the $(j-1)$ level blocks corresponding to $T_{i}, T'_{i'}$. In the above setting, we shall also denote $S_i=G_\theta(T_i)$ and $S'_{i}=G_\theta^{-1}(T'_i)$.

Observe that an $\alpha$-canonical map maps a domain to a domain of same shape while matching up certain sub-blocks in $\mathbb{X}$ (resp.\ in $\mathbb{Y}$) to sub-blocks of same shapes in $\mathbb{Y}$ (resp.\ in $\mathbb{X}$). For embedding, we want to match up all bad sub-blocks by an $\alpha$-canonical map as above. We define embedding at level $j$ formally as follows. Assume that we have defined embedding at levels upto $(j-1)$.

\begin{definition}[Embedding at level $j$]
\label{d:embedding}
Let $X=X^{j}_{U}$, $Y=Y^j_{U}$, $\hat{U}_{X}$, $\hat{U}_{Y}$ be as above. Let $T_1,T_2, \ldots ,T_{k}\subseteq \Z^2$ be such that $X^{(j-1)}_{T_{1}}, \ldots ,X^{(j-1)}_{T_{k}}$ are unions of blocks of $\mathbb{X}$ at level $(j-1)$ containing all $(j-1)$ level bad sub-blocks. Similarly let $T'_1,T'_2,\ldots ,T'_{k'}$ be such that $Y^{(j-1)}_{T'_{1}}, \ldots ,Y^{(j-1)}_{T'_{k'}}$ are unions of blocks of $\mathbb{Y}$ at level $(j-1)$ containing all bad sub-blocks. We say $X$ embeds into $Y$, denoted $X\hookrightarrow Y$ if there exist $\mathcal{T}=\{T_1,T_2,\ldots ,T_k\}$, and $\mathcal{T}'=\{T'_1,\ldots T'_{k'}\}$ as above and there exists an $\alpha$-canonical map $G_{\theta}$ from $\hat{U}_{X}$ to $\hat{U}_{Y}$ with respect to $\mathcal{T}$ and $\mathcal{T}'$ such that for all $i\in [k]$, $X^{(j-1)}_{T_{i}}\hookrightarrow Y^{(j-1)}_{G_\theta(T_i)}$ and for all $i\in [k']$ we have $X^{(j-1)}_{G_\theta^{-1}(T'_{i})}\hookrightarrow Y^{(j-1)}_{T'_i}$.  
\end{definition}

In the situation of the above definition, we say \textbf{$G_{\theta}$ gives an embedding of $X$ into $Y$}. The following sufficient condition for embedding given in terms of components will be useful for us.

\begin{lemma}
\label{l:embedcond}
Let $X=X^{j}_{U}$, $Y=Y^j_{U}$, $\hat{U}_{X}$, $\hat{U}_{Y}$ be as above. Let $T_1,T_2, \ldots ,T_{k}\subseteq \Z^2$ be such that $X^{(j-1)}_{T_{1}}, \ldots ,X^{(j-1)}_{T_{k}}$ are all bad level $(j-1)$ components contained in $X$. Let $W\subseteq \Z^2$ be such that $Y=Y^{j-1}_{W}$. Suppose there exists an $\alpha$-canonical map $G_{\theta}$ from $\hat{U}_{X}$ to $\hat{U}_{Y}$ with respect to $\mathcal{T}=\{T_1,T_2,\ldots ,T_k\}$ and $\emptyset$ such that for all $i\in [k]$, $X^{(j-1)}_{T_{i}}\hookrightarrow Y^{(j-1)}_{G_\theta(T_i)}$ and for all $u\in W\setminus  \cup_{i} G_\theta(T_i)$, $Y^j_{u}$ is a good block at level $j-1$. Then $X\hookrightarrow Y$.
\end{lemma}

\begin{proof}
Follows immediately from Definition \ref{d:embedding}.
\end{proof}

\subsubsection{Random Blocks and Embedding Probabilities}
Observe that at a fixed level $j$ the distribution of the blocks and components is translation invariant. That is, there exist laws $\mu_j^{\mathbb{X}}$ (resp.\ $\mu_j^{\mathbb{Y}}$) such that for all $u\in \Z^2$, the $j$-level component $X^{*,j}(u)$ (resp.\ $Y^{*,j}(u)$) has the law $\mu_j^{\mathbb{X}}$ (resp.\ $\mu_j^{\mathbb{Y}}$).

Fix a component $X^*=X^{j}_{U}$ at level $j$. Let $A^{\mathbb{Y}}_{\rm valid}$ denote the event that the external buffer zones of $B^j_{U}$ are not conjoined in $\mathbb{Y}$. On $A^{\mathbb{Y}}_{\rm valid}$, clearly $Y^*=Y^j_{U}$ is a union of $j$ level blocks in $\mathbb{Y}$. Denote the embedding probability of the component $X^*$  

\begin{equation}
\label{e:defsx}
S_j^{\mathbb{X}}(X^*)=\P[X^*\hookrightarrow Y^*, A^{\mathbb{Y}}_{\rm valid}\mid X^*].
\end{equation} 

In a similar vein we define the embedding probability of a $j$-level $\mathbb{Y}$-component $Y^*$ by      
     
\begin{equation}
\label{e:defsy}
S_j^{\mathbb{Y}}(Y^*)=\P[X^*\hookrightarrow Y^*, A^{\mathbb{X}}_{\rm valid}\mid Y^*].
\end{equation} 

We shall drop the superscripts $\mathbb{X}$ or $\mathbb{Y}$ when it will be clear from the context which block we are talking about. The embedding probabilities are very important quantities for us. The key of our multi-scale proof rests on proving recursive power law tail estimates for $S_j(X^*)$ when $X^*$ is distributed according to $\mu_j^{\mathbb{X}}$ and similarly for $S_j(Y^*)$.

\subsubsection{Semi-bad Components and Airports}
It will be useful for us to classify the bad components at level $j$ into two types: {\bf semi-bad} and {\bf really bad}. A semi-bad component will be one which is not too large in size and has a sufficiently high embedding probability. We define it only for $\mathbb{X}$-components, semi-bad $\mathbb{Y}$-components are defined in a similar fashion. For a component $X$ we shall denote by $V_{X}$ its size, that is the size of the multi-cell it corresponds to. 

\begin{definition}[Semi-bad Components]
\label{d:semibad}
A component $X=X^j_{U}$ at level $j$ is said to be semi-bad if it satisfies the following conditions.
\begin{enumerate}
\item[\rm i.] $V_{X}=|U|\leq v_0$.
\item[\rm ii.] $S_j^{\mathbb{X}}(X)\geq 1-\frac{1}{v_0^5k_0^4100^j}$.
\end{enumerate}
\end{definition}

An \textbf{airport} is a region such that most locations in it can be embedded into any semi-bad component. The formal definition is as follows.

\begin{definition}[Airports]
\label{d:airport}
A square $S$ of $L_{j-1}^{3/2}\times L_{j-1}^{3/2}$ many $j-1$ level cells contained in a $j$ level component of $\mathbb{X}$ is called an airport if for all level $j-1$ semi-bad component $Y^{*}=Y^{j-1}_{U}$ the following condition holds.
\begin{enumerate}
\item[$\bullet$] Fix any $H\subseteq S$ having the same shape as $U$. Let the event that $X^*=X^{j-1}_{H}$ is a union of blocks at level $j-1$ be denoted by $A^{H}_{\rm valid}$. We have 
$$\#\{H: A^{H}_{\rm valid}, X^*\hookrightarrow Y^*\} \geq (1-v_0^{-2}k_0^{-4}100^{-j})N(S,U)$$ 
where $N(S,U)$, denotes the number of multi-cells in $S$ having the same shape as $U$.
\end{enumerate}   
\end{definition}

Airports are defined for $\mathbb{Y}$ blocks in an analogous manner.
%\begin{definition}
%\label{d:blockembedding}
%Fix $j\geq 0$, Let $u,u'\in \Z^2$. Let $X^j(u)$ and $Y^{j}(u')$ be the corresponding $j$ level $\mathbb{X}$-block and $j$ level $\mathbb{Y}$-block respectively. We say that $X^j(u)$ embeds into $Y^{j}(u')$, denoted $X^{j}(u)\hookrightarrow Y^{j}(u')$ if the following conditions hold.
%
%\begin{enumerate}
%\item 
%\end{enumerate}
%
%
%\end{definition}
% 
%
%
%
%
%

\subsection{Good Blocks}
\label{s:goodhd}
To complete the construction of the multi-scale structure, we need to define good blocks at level $j\geq 1$. With the preparations from the preceding subsections, we are now ready to give the recursive definition. Fix $j\geq 1$. Suppose we already have constructed the structure up to level $j-1$. As usual, in the following definition, we only consider $\mathbb{X}$ blocks; $j$-level good blocks for $\mathbb{Y}$ are defined similarly.

\begin{definition}[Good Blocks]
\label{d:good}
A block $X=X(u)$ at level $j$ is said to be good if the following conditions hold.
\begin{enumerate}
\item[\rm i.] $X$ has size $1$, i.e., $B^j(u)$ is an ideal multi-block. 
\item[\rm ii.] The total sizes of $(j-1)$ level bad components contained in $X$ is at most $k_0$.
\item[\rm iii.] All the bad components contained in $X$ are semi-bad.
\item[\rm iv.] All $L_{j-1}^{3/2}\times L_{j-1}^{3/2}$ squares of $(j-1)$ level cells contained in $X$ are airports.  
\end{enumerate} 
\end{definition}

\section{Recursive Estimates}
\label{s:rech}
Our proof of Theorem \ref{t:embedh} depends on a collection of recursive estimates, all of which are proved together by induction. In this section we list these estimates. The proof of these estimates are provided over the next few sections. We recall that for all $j>0$, we have $L_j=L_{j-1}^{\alpha}=L_0^{\alpha ^j}$.

\begin{itemize}
\item {\bf Tail Estimate:}
Let $j\geq 0$. Let $X=X^{j}_{U}$ be a $\mathbb{X}$-component at level $j$ (having the distribution $\mu_{j}^{\mathbb{X}}$) and  let $m_j=m+2^{-j}$. Recall $V_X=|U|$. Then
\begin{equation}
\label{e:tailx1h}
\mathbb{P}(S_j^{\mathbb{X}}(X)\leq x, V_{X}\geq v)\leq x^{m_j}L_j^{-\beta}L_{j}^{-\gamma(v-1)}~~\text{for}~~0< x\leq 1-L_{j}^{-1}~~\text{for all}~v\geq 1.
\end{equation}
Let $Y=Y^{j}_{U}$ be a $\mathbb{Y}$-component at level $j$ (having the distribution $\mu_{j}^{\mathbb{X}}$). Recall $V_X=|U|$. Then
\begin{equation}
\label{e:taily1h}
\mathbb{P}(S_j^{\mathbb{Y}}(Y)\leq x, V_{Y}\geq v)\leq x^{m_j}L_j^{-\beta}L_{j}^{-\gamma(v-1)}~~\text{for}~~0< x\leq 1-L_{j}^{-1}~~\text{for all}~v\geq 1.
\end{equation}

\item{\bf Size Estimate:}
Let $j\geq 0$. Let $X=X^{j}_{U}$ (resp.\ $Y=Y^{j}_{U}$) be a $\mathbb{X}$-component at level $j$ having the distribution $\mu_{j}^{\mathbb{X}}$ (resp.\ $\mathbb{Y}$-component at level $j$ having the distribution $\mu_{j}^{\mathbb{Y}}$). Then
\begin{equation}
\label{e:sizex1h}
\mathbb{P}(V_{X}\geq v)\leq L_{j}^{-\gamma(v-1)}~~\text{for}~~v\geq 1.
\end{equation}
\begin{equation}
\label{e:sizey1h}
\mathbb{P}(V_{Y}\geq v)\leq L_{j}^{-\gamma(v-1)}~~\text{for}~~v\geq 1.
\end{equation}

\item {\bf Good Block Estimate:}
\begin{itemize}
\item[$\bullet$] Good blocks embed into to good blocks, i.e., for all good $j$-level block $X$ and for all good $j$-level block $Y$ we have
\begin{equation}
\label{e:goodvgoodh}
X\hookrightarrow Y.
\end{equation}

%\item[$\bullet$] Most blocks are ``good". For each $u\in \Z^2$
%\begin{equation}
%\label{e:xgoodh}
%\mathbb{P}(X^j(u)~\text{is a good block at level $j$})\geq 1-L_j^{-\delta}.
%\end{equation}
%\begin{equation}
%\label{e:ygoodh}
%\mathbb{P}(Y^j(u)~\text{is a good block at level $j$})\geq 1-L_j^{-\delta}.
%\end{equation}

\item[$\bullet$] Conditioned on a partial set of outside blocks, blocks are good with high probability. Fix $u\in \Z^2$. Let $V\subseteq \Z^2\setminus \{u\}$. Let $\cf_{V}^{\mathbb{X}}$ (resp.\ $\cf_{V}^{\mathbb{Y}}$) denote the conditioning $\cf_{V}^{\mathbb{X}}=\{X^j_{V}, X^j(u)\notin X^j_V\}$ (resp.\ $\cf_{V}^{\mathbb{Y}}=\{Y^j_{V}, Y^j(u)\notin Y^j_V\}$), i.e. we condition on partial set of $j$ level blocks corresponding to ideal blocks excluding $B^j(u)$ such that these blocks are not the block corresponding to $B^j(u)$. Then we have the following for all $u\in \Z^2$ and for all $V\subseteq \Z^2\setminus \{u\}$.

\begin{equation}
\label{e:goodconditionx}
\P[X^j_u~\text{is good}\mid \cf_{V}^{\mathbb{X}}]\geq 1-L_{j}^{-\gamma}.
\end{equation}

\begin{equation}
\label{e:goodconditiony}
\P[Y^j_u~\text{is good}\mid \cf_{V}^{\mathbb{Y}}]\geq 1-L_{j}^{-\gamma}.
\end{equation}
\end{itemize}
\end{itemize}

\begin{theorem}[Recursive Theorem]
\label{inductionh}
For $\alpha$, $\beta$, $\gamma$, $m$, $k_0$ and $v_0$ as in equation~\eqref{e:parameters}, the following holds for all large enough $L_0$. If the recursive estimates \eqref{e:tailx1h}, \eqref{e:taily1h}, \eqref{e:sizex1h}, \eqref{e:sizey1h}, \eqref{e:goodvgoodh}, \eqref{e:goodconditionx} and \eqref{e:goodconditiony} hold at level $j$ for some $j\geq 0$ then all the estimates hold at level $(j+1)$ as well.
\end{theorem}

We shall prove Theorem \ref{inductionh} over the next few sections. Before that we prove that these estimates indeed hold at level $j=0$.

\begin{theorem}
\label{t:inductionbase}
Fix $\alpha$, $\beta$, $\gamma$, $m$, $k_0$, $v_0$ and $L_0$ such that the conclusion of Theorem \ref{inductionh} holds. Then for $M_0$ sufficiently large depending on all the parameters the estimates \eqref{e:tailx1h}, \eqref{e:taily1h}, \eqref{e:sizex1h}, \eqref{e:sizey1h}, \eqref{e:goodvgoodh}, \eqref{e:goodconditionx} and \eqref{e:goodconditiony} hold for $j=0$.
\end{theorem} 

\begin{proof} 
Observe that \eqref{e:goodvgoodh} for $j=0$ follows from the definition of good blocks at level $0$. Recall that blocks at level $0$ are independent and hence by taking $M_0$ sufficiently large we make sure that (\ref{e:goodconditiony}) holds for $j=0$, and \eqref{e:goodconditionx} holds vacuously. As a matter of fact, by taking $M_0$ sufficiently large, we can ensure that $\P[Y^0(u)~\text{is good}]\geq 1-L_0^{-20\beta}$. Notice that components of $\mathbb{X}$ all have size $1$ and hence \eqref{e:sizex1h} also holds trivially. For a component $X=X^0(u)$ we have $S_0^{\mathbb{X}}(X)\geq \P[Y^0(u)~\text{is good}]\geq 1-L_0^{-1}$, and hence \eqref{e:tailx1h} also holds.

Now look at the component $Y=Y^{*,0}(u)=Y^0_{U}$. If $V_{Y}=v>1$, there are at least $\frac{v}{9}$ many bad blocks contained in $Y$. Since blocks are independent, it follows by summing over all lattice animals containing $u$ of size $v$ that $\P[V_{Y}\geq v]\leq L_j^{-10\beta (v-1)}$. Also notice that, $S_0^{\mathbb{Y}}(Y)=1$ if $Y$ is good, $S_0^{\mathbb{Y}}(Y)=\frac{1}{2^{v}}$ otherwise. Hence it suffices to prove \eqref{e:taily1h} for $x=\frac{1}{2^v}$ and $v\geq 1$. For $x\leq \frac{1}{2}$, it follows that for  $\P[S_0^{\mathbb{Y}}(Y)\leq x, V_{Y}\geq v]\leq \P[V_{Y}\geq \max\{v,\log_{2} x\}]\leq x^{m_{j+1}}L_0^{-\beta}L_0^{-\gamma(v-1)}$ because $L_0$ is sufficiently large. This establishes \eqref{e:taily1h} for $j=0$.
%
%
%
%
%Recall that lattice blocks at level $0$ are always of size $1$. Notice now that for a component $X=X^0_{U}$ at level $0$, we have  $S_0{\mathbb{X}}(X)\geq \P[Y^0_{U}~\text{consists of good blocks}]\geq 1-L_0^{-1}$, if $M_0$ is chosen sufficiently large depending on $L_0$. Hence \eqref{e:tailx1h} follows for $j=0$. For \eqref{e:taily1h}, notice that, $S^{0,\mathbb{Y}}(Y)=1$ if $Y$ is good, $S^{0,\mathbb{Y}}(Y)=\frac{1}{2}$ otherwise. Hence it suffices to prove \eqref{e:taily1h} for $x=\frac{1}{2}$. It follows by taking $M_0$ sufficiently large depending on all other parameters. It is trivial to observe that \eqref{e:xgoodh} holds and \eqref{e:ygoodh} follows by taking $M_0$ sufficiently large. Since blocks at level $0$ are independent and has fixed boundaries (\ref{e:goodconditionx}) and (\ref{e:goodconditiony}) are immediate for $j=0$. 
\end{proof}

\subsection{Proof of the Main Theorem}

Before proceeding with the proof of Theorem \ref{inductionh}, we show how Theorem \ref{inductionh} and Theorem \ref{t:inductionbase} can be used to deduce Theorem \ref{t:embedh}. 

\begin{proof}[Proof of Theorem \ref{t:embedh}]
Notice that by ergodic theory considerations it suffices to prove that $\P[\mathbb{X}\hookrightarrow _{M} \mathbb{Y}]>0$ for some $M$. Fix $\alpha$, $\beta$, $\gamma$, $m$, $k_0$, $v_0$, $L_0$ and $M_0$ in such a way that conclusions of both Theorem \ref{inductionh} and Theorem \ref{t:inductionbase} hold. This implies that the recursive estimates (\ref{e:tailx1h}), (\ref{e:taily1h}), (\ref{e:goodvgoodh}), (\ref{e:sizex1h}), (\ref{e:sizey1h}), (\ref{e:goodconditionx}) and (\ref{e:goodconditiony}) hold for all $j\geq 0$. 

Let $u_1=(0,0)$, $u_2=(0,-1)$, $u_3=(-1,1)$ and $u_4=(-1,0)$. So $\{X^j_{u_{i}}:i\in \{1,2,3,4\}\}$ denote the blocks surrounding the origin at level $j$. Let us denote the domains of these blocks by $D^{j,\mathbb{X}}_{u_i}$ respectively. Define $D^{j,\mathbb{Y}}_{u_i}$ similarly.
For $j\geq 0$, let $\mathcal{A}^{\mathbb{X}}_j$ (resp.\ $\mathcal{A}^{\mathbb{Y}}_j$) denote the following event that for all $i\in [4]$ we have $D^{j,\mathbb{X}}_{u_i}=B^{j}_{u_i}$ and $X^j_{u_i}$ is good (resp.\ $D^{j,\mathbb{Y}}_{u_i}=B^{j}_{u_i}$ and $Y^j_{u_i}$ is good). The proof is now completed using the following three propositions. 
%
%
%
%. The $j$-level $\mathbb{X}$-blocks (resp.\ $\mathbb{Y}$-blocks) corresponding to the four $j$-level ideal blocks surrounding the origin are all good and have the straightline boundary. It follows from the estimates (see the next proposition) that
%
%\begin{equation}
%\label{e:allgood}
%\P\left[\bigcap_{j\geq 0} (\mathcal{A}^{\mathbb{X}}_j\cap \mathcal{A}^{\mathbb{Y}}_j)\right]>0.
%\end{equation}
%
%Observe that on $\bigcap_{j\geq 0} (\mathcal{A}^{\mathbb{X}}_j\cap \mathcal{A}^{\mathbb{Y}}_j)$ it follows by a compactness argument that $\mathbb{X}\hookrightarrow _{M} \mathbb{Y}$ for $M=10M_0$. See proposition \ref{p:compactness}
%
%This completes the proof.
\end{proof}

%Now, as promised, we provide a proof of \eqref{e:allgood}.

\begin{proposition}
\label{p:allgood}
Suppose (\ref{e:tailx1h}), (\ref{e:taily1h}), (\ref{e:goodvgoodh}), (\ref{e:sizex1h}), (\ref{e:sizey1h}), (\ref{e:goodconditionx}) and (\ref{e:goodconditiony}) hold for all $j\geq 0$. Then 
$$\P\left[\bigcap_{j\geq 0} (\mathcal{A}^{\mathbb{X}}_j\cap \mathcal{A}^{\mathbb{Y}}_j)\right]>0.$$
\end{proposition}

\begin{proof}
Notice that $\mathcal{A}_j^{\mathbb{X}}=\mathcal{A}_j^{\mathbb{X},G}\cap \mathcal{A}_j^{\mathbb{X},\partial}$ where $\mathcal{A}_j^{\mathbb{X},G}$ denotes the event the the four blocks around origin at level $j$ are good and $\mathcal{A}_j^{\mathbb{X},\partial}$ denotes the event that $D^{j,\mathbb{X}}_{u_i}=B^j(u_i)$ for all $i$. Clearly $\P[\mathcal{A}_0^{\mathbb{X},\partial}]=1$. Now for $j\geq 1$, let $\mathcal{A}_j^{\mathbb{X},B}$ denote the event that the  $j$-level buffer zones in these blocks (12 in number) only contain good $(j-1)$ level blocks. Clearly from construction of the blocks 
$$\P[\mathcal{A}_j^{\mathbb{X},\partial}]\geq (1-10^{-(j+2)})\P[\mathcal{A}_j^{\mathbb{X},B}].$$
It follows from \eqref{e:goodconditionx} that for $j\geq 1$
$$\P[\mathcal{A}_j^{\mathbb{X},B}]\geq 1-12L_{j-1}^{\alpha+3-\gamma}.$$
Since $\gamma >\alpha +3$ we get   
$$\P[\mathcal{A}_j^{\mathbb{X},\partial}] \geq 1-10^{-(j+3/2)}$$
by taking $L_0$ sufficiently large. Combining these estimates we get for all $j\geq 0$,
$$\P[\mathcal{A}_j^{\mathbb{X}}]\geq 1-10^{-(j+1)}.$$
By the obvious symmetry between $\mathbb{X}$ and $\mathbb{Y}$ the same lower bound also holds for $\P[\mathcal{A}_j^{\mathbb{Y}}]$. The proposition follows by taking a union bound over $\mathbb{X}$, $\mathbb{Y}$ and all $j\geq 0$.
\end{proof}

%To complete the proof we need the following proposition.

\begin{proposition}
\label{p:compactness1}
Fix $J\in \N$. On $\bigcap_{J\geq j\geq 0} (\mathcal{A}^{\mathbb{X}}_j\cap \mathcal{A}^{\mathbb{Y}}_j)$, there exists a map $\Phi=\Phi_{J}:[-L_J,L_j]^2\rightarrow [-L_J,L_J]^2$ satisfying the following conditions.
\begin{enumerate}
\item[\rm i.] $\Phi(0)=0$ and $\Phi$ is identity on the boundary.
\item[\rm ii.] $\Phi$ is bi-Lipschitz with Lipschitz constant $10$.
\item[\rm iii.] For each level $0$ bad $\mathbb{Y}$-block $Y^0(u')$ contained in $[-\frac{1}{2}L_J, \frac{1}{2}L_J]^2$, there is a  $\mathbb{X}$-block $X^0(u)$ at level $0$ such that $\Phi(u)=u'$ and $X^0(u)\hookrightarrow Y^0(u')$.  
\end{enumerate}
\end{proposition}

We postpone the proof of Proposition \ref{p:compactness1} for the moment.

\begin{proposition}
\label{p:compactness2}
Suppose for all $J\in \N$, there exists a $\Phi_{J}$ satisfying the conditions in Proposition \ref{p:compactness1}. Then there exists a $20M_0$-Lipschitz injection $\phi$ from $\Z^2\to \Z^2$ such that $X_{\iota+v}=Y_{\iota+\phi(v)}$ for all $v\in \Z^2$.
\end{proposition}

\begin{proof}
Fix $J\in \N$. Define $\phi^J:[-\frac{1}{4}L_j,\frac{1}{4}L_J]^2\cap \Z^2 \to \Z^2$ as follows.

{\bf Case 1:}
For $u\in \Z^2$ suppose $\Phi_J(u)=v=(v_1,v_2)$ be such that $Y^0_{v}$ is a bad level $0$ block of $\mathbb{Y}$. Clearly, there exists $v'\in \Z^2$ such that $\iota + v'\in [v_1M_0,(v_1+1)M_0]\times [v_2M_0,(v_2+1)M_0]$ and $X_{\iota+u}=Y_{\iota +v'}$. Choose such a $v'$ arbitrarily and set $\phi^J(u)=v'$.

{\bf Case 2:} If for $u\in \Z^2$ Case 1 does not hold then there exists $v\in \Z^2$ such that $||v-\Phi_{J}(u)||\leq 1$ and $Y^0(v)$ is a good block. Clearly there are many (at least $M_0^2/3$ in number) $\iota+v'\in [v_1M_0,(v_1+1)M_0]\times [v_1M_0,(v_1+1)M_0]$ such that $X_{\iota+u}=Y_{\iota+v'}$. Choose one such $v'$ arbitrarily and set $\phi^J(u)=v'$. Since the number of sites $u$ that correspond to $v$ in the above manner is limited it follows that such a $\phi^J$ can be chosen to be an injection. 

Notice that $\phi^J$ is $20M_0$-Lipschitz and also observe that as $\Phi^J$ is identity at the origin it follows that $||\phi^J(0)||\leq M_0$. It now follows by a compactness argument that there exists an injective map $\phi:\Z^2 \rightarrow \Z^2$ which is $20M_0$-Lipschitz and such that $X_{\iota+v}=Y_{\iota+\phi(v)}$ for all $v\in \Z^2$.      
\end{proof}

It remains to prove Proposition \ref{p:compactness1} which will follow from the next lemma.

\begin{lemma}
\label{l:compactinduction}
Assume the set-up of Proposition \ref{p:compactness1}. Fix $0<j\leq J$. Suppose there exists a map $\phi_j:[-L_J,L_J]^2\rightarrow [-L_J,L_J]^2$ satisfying the following conditions.
\begin{enumerate}
\item[\rm i.] $\phi_j(0)=0$ and $\phi_j$ is identity on the boundary.We take $\phi_{J}$ to be the identity map.
\item[\rm ii.] $\phi_{j}$ is bi-Lipschitz with Lipschitz constant $C_j$.
\item[\rm iii.] There exists $\{X^j_{U}\}_{U\in I_1}$ with respective domains $\{\hat{U}_{X}\}_{U\in I_1}$ containing all bad level $j$ blocks of $\mathbb{X}$ contained in $[-L_J(1-10^{-(j+1)}),L_J(1-10^{-(j+1)})]^2$. Also there exists $\{Y^j_{U'}\}_{U'\in I_2}$ with respective domains $\{\hat{U'}_{Y}\}_{U'\in I_2}$ containing all bad level $j$ blocks of $\mathbb{Y}$ contained in $[-L_J(1-10^{-(j+1)}),L_J(1-10^{-(j+1)})]^2$. 

Also all $U\in I_1$, there exists $f(U)$ such that $\phi_{j}$ restricted to $\hat{U}_{X}$ is a canonical map to the domain $\widehat{f(U)}_{Y}$ of $Y^j_{f(U)}$ and such that such that $X^j_{U}\hookrightarrow Y^j_{f(U)}$. Further for each $U'\in I_2$, there exists $f^{-1}(U)$ such that $\phi_j^{-1}(\hat{U'}_{Y})$ is the domain $\widehat{f^{-1}(U')}_{X}$ of the block $X^j_{f^{-1}(U)}$ and such that $\phi_j$ restricted to $\widehat{f^{-1}(U')}_{X}$ is a canonical map and $X^j_{f^{-1}(U)}\hookrightarrow Y^j_{U'}$.  
%
%
%
%For each level $j$ bad $\mathbb{X}$-component $X^j(u)$ contained in $[-L_j(1-10^{-(j+1)}),L_j(1-10^{-(j+1)})]^2$ (resp.\ for each level $j$ bad $\mathbb{Y}$-component $Y^j(u')$ contained in $[-L_j(1-10^{-(j+1)}),L_j(1-10^{-(j+1)})]^2$) we have union of level $j$  $\mathbb{Y}$-components $Y^j(u')$ (resp.\ we have union of level $j$  $\mathbb{X}$-components $X^j(u)$) such that $X^j(u)\hookrightarrow Y^j(u')$ and $\phi_j(V^j(u))=V^j(u')$ for all such pairs of components and union of components.
\end{enumerate}
Then there exists a function $\psi_{j-1}:[-L_J,L_J]^2\rightarrow [-L_J,L_J]^2$ with $\psi_{j-1}(0)=0$, $\psi_{j-1}$ identity on the boundary, bi-Lipschitz with Lipschitz constant $(1+10^{-(j+9)})$ such that $\phi_{j-1}:=\psi_{j-1}\circ \phi_j$ satisfies all the above conditions with $j$ replaced by $(j-1)$ (with setting $C_{j-1}=C_j(1+10^{-(j+9)})$), that is $\phi_{j-1}$ matches up all the bad $(j-1)$ level components in a Lipschitz manner.  
\end{lemma}

Notice that Proposition \ref{p:compactness1} follows from Lemma \ref{l:compactinduction} using induction and definition of good block and embedding. Now we prove Lemma \ref{l:compactinduction}.

\begin{proof}[Proof of Lemma \ref{l:compactinduction}]
We shall construct $\psi_{j-1}$ satisfying the requirements of the lemma. The strategy we adopt is the following. Denote $\mathcal{B}= \{\hat{U}_{X}:U\in I_1; \widehat{f^{-1}(U')}_X: U'\in I_2\}$. Set $B=[-L_J,L_J]^2\setminus \cup_{A\in \mathcal{B}} A$. For each $A\in \mathcal{B}\cup \{B\}$, we shall construct a function $\psi^{A}: \phi_j(A)\to \phi_j(A)$ that is a Lipschitz bijection with Lipschitz constant $(1+10^{-(j+9)})$ and is identity on the boundary of $A$. We shall take $\psi_{j-1}$ to be the unique function on $[-L_J,L_J]^2$such that its restriction to $A$ equals $\psi^A$ for each $A\in \mathcal{B}\cup \{B\}$. We shall then verify that $\psi_{j-1}$ constructed as such satisfies the conditions of the lemma.  
 
First fix $A\in \mathcal{B}$. We describe how to construct $\psi^A$. Without loss of generality, take $A=\hat{U}_{X}$ and hence $\phi_j(A)=\widehat{f(U)}_{Y}$. By definition of $X^j_{U}\hookrightarrow Y^j_{f(U)}$, there exists an $\alpha$-canonical map $G_{\theta}$ from $\hat{U}_{X}\to \widehat{f(U)}_{Y}$ that gives the embedding. Take $\psi^A=\theta$. Clearly $\psi^A$ is identity on the boundary and is a Lipschitz bijection with Lipschitz constant $(1+10^{-(j+9)})$, and also satisfies the hypothesis of the Lemma for all $(j-1)$ level bad blocks in $X^j_{U}$ and $Y^{j}_{f(U)}$ by Definition \ref{d:embedding}.    

So it suffices to define $\psi^{B}$ in such a way that all the bad components at level $(j-1)$ that are not contained in any $j$-level bad sub-block are matched up. 

Let $\{X^{j-1}_{W}\}_{W\in I'_1}$ denote the set of all $(j-1)$-level bad components of $\mathbb{X}$ contained in $[-L_j(1-10^{-j}),L_j(1-10^{-j})]^2$ and not contained in $A$ for any $A\in \mathcal{B}$, let $\hat{W}_{X}$ denote their respective domains. Similarly let $\{Y^{j-1}_{W'}\}_{W'\in I'_2}$ denote the set of all $(j-1)$-level bad components of $\mathbb{Y}$ contained in $[L_j(1-10^{-j}),L_j(1-10^{-j})]^2$ and not contained in $\phi_j(A)$, let $\hat{W'}_{Y}$ denote their respective domains. We shall only describe how to match up the $X^{j-1}_{W}$'s; the $\mathbb{Y}$-components can be taken care of similarly.

Notice that since all these are contained in good $j$-level blocks, these are all away from the boundaries of $B$ and $\phi_j(B)$ respectively. Also it follows from the definition of good blocks that there cannot be too many bad components close together. It is easy to see that one can find squares $S_{k}\subseteq \phi_j(B)$ such that $\phi_{j}(W_{X})$ are all contained in the union of $S_k$, are at distance at least $L_{j-1}^3$ from the boundary of $S_k$'s and such that $S_k$ does not intersect any of the bad $(j-1)$-level components of $\mathbb{Y}$. Also it can be ensured that for a fixed $k$, the total size of components $X^{j-1}_{W}$ such that $\phi_j(W_{X})$ is contained in $S_k$ is not more than $k_0$. Since by definition of good block the components $X^{(j-1)}_{W}$ are all semi-bad, and the region $\phi_j(S_k)$ contains enough airports it follows that it is possible to define a map $\psi^{S_k}: \phi(S_k)\to \phi(S_k)$ which is identity on the boundary and $\psi^{S_k}$ gives an embedding  $X^{j-1}_{W}\hookrightarrow Y^{j-1}_{g(W)}$ for all the bad $j-1$ level components contained in $S_k$. Now gluing together all such maps we get the required $\psi^{B}$ that matches up all the bad $j-1$ level components contained in $B$. We omit the details, see the proof of Proposition \ref{p:canonical1} for a similar construction. This completes the proof. 

\end{proof}

\vspace{1in}

The remainder of the paper is devoted to the proof of the estimates in the induction statement. Throughout these sections we assume that the estimates (\ref{e:tailx1h}), (\ref{e:taily1h}), (\ref{e:goodvgoodh}), (\ref{e:sizex1h}), (\ref{e:sizey1h}), (\ref{e:goodconditionx}) and (\ref{e:goodconditiony}) hold for some level $j\geq 0$ and then prove the estimates at level $j+1$.  Combined they will complete the proof of Theorem~\ref{inductionh}.

From now on, in every Theorem, Proposition and Lemma we state, we would
implicitly assume the hypothesis that all the recursive estimates hold upto level $j$, the parameters satisfy the constraints described in \S~\ref{s:parameters} and $L_0$ is sufficiently large.

\section{Geometric Constructions}
\label{s:construction}

To show the existence of embeddings we need to construct $\alpha$-canonical maps having different properties. In this section we develop different geometric constructions which shall imply the existence of $\alpha$-canonical maps in different cases. We start with the following simple case where the blocks are not moved around but only the boundaries of domains are adjusted. More specifically, our aim is the following. Consider a potential domain $\tilde{U}$ at level $(j+1)$. We want to construct an $\alpha$-canonical map from $\tilde{U}$ to itself that takes a given set of potential domains of $j$-level multi-cells contained in $\tilde{U}$ that are away from the boundary to any other given set of potential domains of the same multi-cells. We have the following proposition.

\begin{proposition}
\label{p:starcanonical}
Fix a lattice animal $U\subseteq \Z^2$. Fix a level $(j+1)$ potential domain $\tilde{U}$ corresponding to the multi-cell $B^{j+1}_{U}$, i.e., $C=C(\tilde{U})$, the boundary of $\tilde{U}$ is a potential boundary curve through the buffer zone of $B^{j+1}_{U}$. Let $U_1\subseteq \Z^2$ be such that $B^j_{U_1}$ denotes the collection of all level $j$-cells that intersect $\tilde{U}$. Let $U_2\subseteq U_1$ be the set of all vertices in $U_1$ such that the distance from the boundary of $U_1$ is at least $\frac{L_j^3}{2}$. Let $\mathcal{T}=\{T_1,T_2,\ldots T_{\ell}\}$ be a set of disjoint subsets of $U_2$. Let $\{U_{T_{i}}\}_{i\in [\ell]}$ (resp.\ $\{\hat{U}_{T_{i}}\}_{i\in [\ell]}$) denote a set of potential domains corresponding to the $j$-level multi-cells $B^{j}_{T_{i}}$ that are compatible i.e., there exists a canonical map $\Gamma$ (resp.\ $\hat{\Gamma}$) at level $j$ such that $\Gamma(B^j_{T_i})=U_{T_i}$ (resp.\ $\hat{\Gamma}(B^j_{T_i})=\hat{U}_{T_i}$). Then there exists an $\alpha$-canonical map $\Upsilon$ with respect to $\mathcal{T}$ and $\emptyset$ such that $\Upsilon(T_i)=T_i$ for each $i\in [\ell]$.     
\end{proposition}

\begin{proof}
Notice that the canonical map from $\tilde{U}$ to itself is the identity map. Also observe that without loss of generality we can assume that $U=\{\mathbf{0}\}$ and $\mathcal{T}=\{\{u\}:u\in U_2\}$. Let $U_2\subseteq U_3\subseteq \Z^2$ be such that $U_3$ contains all sites at a distance $2$ from $U_2$. It is clear from the construction of potential domains that $\Gamma$ and $\hat{\Gamma}$ as in the statement of the proposition can be chosen such that both $\Gamma$ and $\hat{\Gamma}$ are identity outside $B^j_{U_3}$. Define the map $\Upsilon$ on $\tilde{U}$ by $\Upsilon:=\hat{\Gamma}\circ \Gamma^{-1}$. It follows from definition that 

\begin{enumerate}
\item[\rm i.] $\Upsilon$-is identity on $\tilde{U}\setminus B^j_{U_3}$, in particular on the boundary of $\tilde{U}$.
\item[\rm ii.] $\Upsilon$ is bi-Lipschitz with Lipschitz constant $(1+10^{-(j+5)})$.
\item[\rm iii.] $\Upsilon(U_{T_i})=\hat{U}_{T_i}$ for all $i$.
\end{enumerate}
It follows then from the definition of $\alpha$-canonical maps that $\Upsilon=\Upsilon\circ Id$ is an $\alpha$-canonical map from $\tilde{U}$ to itself with respect to $\mathcal{T}$ and $\emptyset$ such that  $\Upsilon(T_i)=T_{i}$ for all $i$. This completes the proof. 
\end{proof}

An $\alpha$-canonical map as above will be referred to as a $*$-canonical map. Figure \ref{f:starcano} illustrates this construction.

\begin{figure}[h]
\begin{center}
\includegraphics[width=\textwidth]{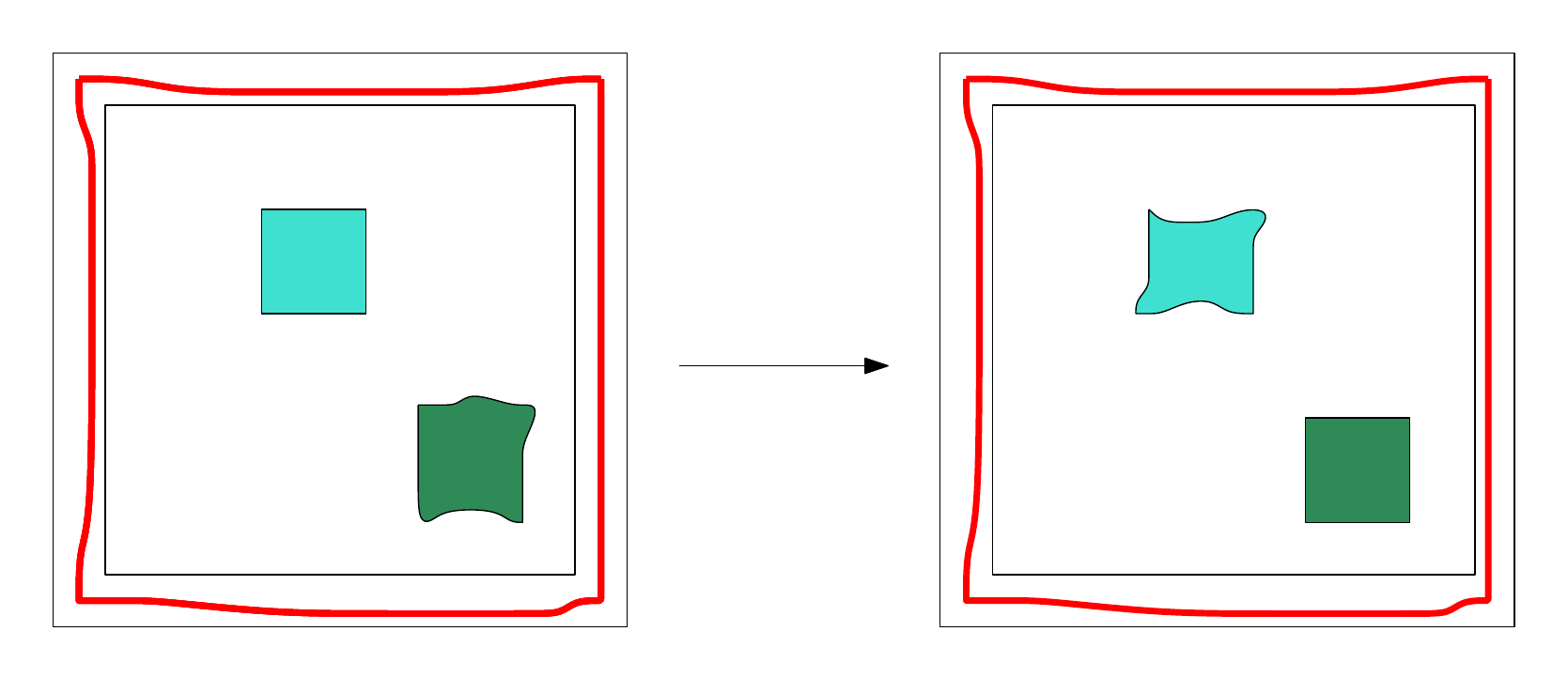}
\end{center}
\caption{A $*$-canonical map}
\label{f:starcano}
\end{figure}

Now we want to move to a more complicated construction of $\alpha$-canonical maps, where we want to match up a non-trivial subset of bad blocks in both $\mathbb{X}$ and $\mathbb{Y}$. We have the following proposition.

\begin{proposition}
\label{p:canonical1}
Fix $U\subseteq \Z^2$. Consider $\tilde{U}_1$ and $\tilde{U}_2$ to be any two potential domains corresponding to the $j+1$-level multi-cell $B^{j+1}_{U}$. Let $U_{1,1}\subseteq \Z^2$ (resp.\ $U_{1,2}\subseteq \Z^2$) be such that $B^j_{U_{1,1}}$ (resp.\ $B^j_{U_{1,2}}$) denotes the collection of all level $j$-cells that intersect $\tilde{U}_1$ (resp.\ $\tilde{U}_2$). Let $U_{2,1}\subseteq U_{1,1}$ (resp.\ $U_{2,2}\subseteq U_{1,2}$) be the set of all vertices in $U_{1,1}$ (resp.\ $U_{1,2}$) such that the distance from the boundary of $U_{1,1}$ (resp.\ $U_{1,2}$) is at least $\frac{L_j^3}{2}$. Let $\mathcal{T}=\{T_1,T_2,\ldots T_{\ell_1}\}$ and $\mathcal{T}'=\{T'_1,T'_2,\ldots T'_{\ell_2}\}$  be a set of disjoint and non-neighbouring subsets of $U_{2,1}$ and $U_{2,2}$ respectively such that $\sum |T_i|\leq v_0k_0$ and $\sum |T'_i|\leq v_0k_0$. Then there exists a sequence of $\alpha$-canonical maps $\{\Upsilon_{h_1,h_2}\}_{(h_1,h_2)\in [L_j^2]^2}$ from $\tilde{U}_1$ to $\tilde{U}_2$ with respect to $\mathcal{T}$ and $\mathcal{T}'$ satisfying the following conditions.
\begin{enumerate}
\item[\rm i.] For each $i\in [\ell_1]$, $\Upsilon_{h_1,h_2}(T_i)=(h_1-1,h_2-1)+\Upsilon_{1,1}(T_i)$ and for each $i\in [\ell_2]$, $\Upsilon^{-1}_{h_1,h_2}(T'_i)=-(h_1-1,h_2-1)+\Upsilon_{1,1}^{-1}(T'_i)$.
\item[\rm ii.] For all $h=(h_1,h_2)$ for all $i\in [\ell_1], i'\in [\ell_2]$ we have $T_i$ and $\Upsilon_{h}^{-1}(T'_{i'})$ are disjoint and non-neighbouring.
\end{enumerate}
\end{proposition}

\begin{figure}[ht!]
\begin{center}
\includegraphics[width=\textwidth]{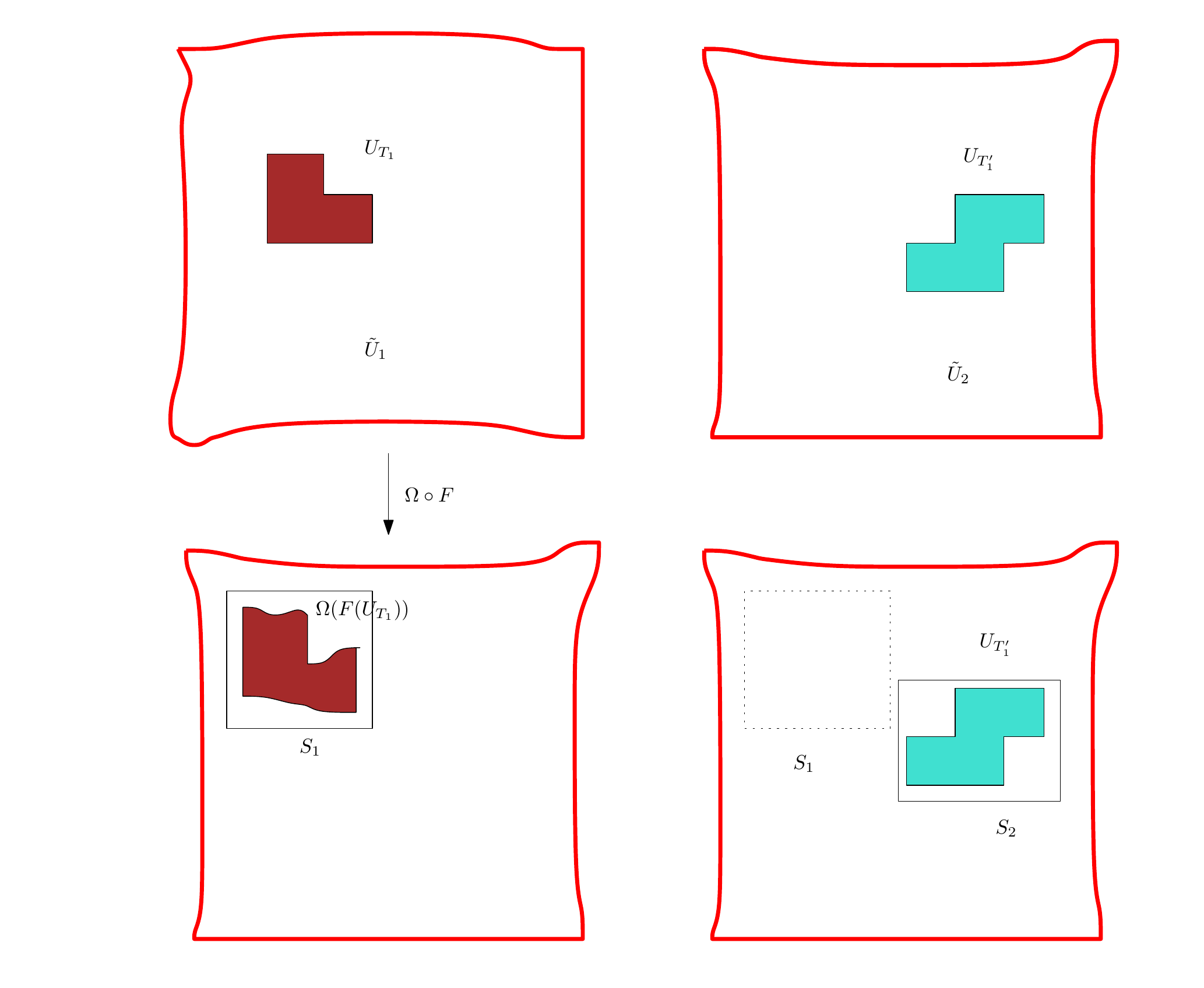}
\end{center}
\caption{Construction of $S_i$ as described in the proof of Proposition \ref{p:canonical1}}
\label{f:cano11}
\end{figure}

\begin{proof}
For $i\in [\ell_1]$ (resp.\ $i'\in [\ell_2]$) Let  $U_{T_i}$ (resp.\ $U_{T'_{i'}}$) be the domain corresponding to $B^j_{T_i}$ (resp.\ $B^j_{T'_{i'}}$). Also let $F$ denote the canonical map at level $(j+1)$ that takes $\tilde{U}_{1}$ to $\tilde{U}_{2}$. Since $\tilde{U}_{2,1}$ and $\tilde{U}_{2,2}$ are away from the boundaries of $U_{1,1}$ and $U_{1,2}$ respectively (by a distance of order $L_j^3$) and the total sizes of $\mathcal{T}$ and $\mathcal{T}'$ are bounded (independent of $L_j$) it follows that for $L_j$ sufficiently large there exists a function $\Omega: \tilde{U}_2 \to \tilde{U}_2$ satisfying the following properties.

\begin{enumerate}
\item[\rm i.] $\Omega$ is identity on the boundary of $\tilde{U}_2$, and bi-Lipschitz with Lipschitz constant $(1+10^{-(j+10)})$.
\item[\rm ii.] There exists squares $S_{1},S_2,\ldots S_{k} \subseteq \tilde{U}_2$ with the following properties.
\begin{itemize}
\item We have that 
$$\biggl(\cup_{i} \Omega \circ F(U_{T_i})\biggr) \bigcup \left(\cup_{i'} U_{T'_{i}}\right) \subseteq \cup_{i=1}^{k} S_i.$$
\item For a fixed $i\in [k]$, $S_i$ intersects at most one of $\cup_{i} \Omega \circ F(U_{T_i})$ and $\cup_{i'} U_{T'_{i'}}$.
\item Distance between $S_{i}$ and $S_{i'}$ for $i\neq i'$ is at least $L_j^{7/2}$.
\item The distance between the boundary of $S_i$ and the sets $F(U_{T_{\ell}})$ or $U_{T'_{\ell'}}$ contained in it is at least $L_j^{7/2}$.
\end{itemize}
\end{enumerate}

See Figure \ref{f:cano11} for the above construction. For $(h_1,h_2)\in [L_j^2]^2$  We shall construct functions $\rho_{h_1,h_2}: \tilde{U}_2 \to \tilde{U}_2$ such that each $\rho_{h_1,h_2}$ is identity except on the interior of $\cup_{i} S_i$. Eventually we shall show that $\Upsilon_{h_1,h_2}:=\rho_{h_1,h_2}\circ \Omega \circ F$ will be the sequence of $\alpha$-canonical maps satisfying the conditions in the statement of the proposition. Without loss of generality we describe below how to construct $\rho_{h_1,h_2}$ on $S_1$, similar constructions work for the other $S_i$. 

Now consider $S_1$. Without loss of generality assume that $S_1$ contains only $\Omega\circ F(U_{T_1})$, more general cases can be handled in a similar manner.  Fix $W$, a translate of $T_1$, such that the distance of $F(U_{T_1})$ from $B^j_{W}$ is at most $2L_j$. For $h=(h_1,h_2)$ set $W_{h}=(h_1,h_2)+W$. Let the domain corresponding to the multi-cell $B^j_{W_{h}}$ be denoted by $\tilde{W}_{h}$. Let $G_{h}$ denote the canonical map from $U_{T_1}$ to $\tilde{W}_{h}$.  On $\Omega \circ F(U_{T_1})$, set $\rho_h:=G_h\circ F^{-1}\circ \Omega ^{-1}$. Clearly $\rho_h$ is bi-Lipschitz with Lipschitz constant $1+10^{-(j+7)}$, also since $F(U_{T_1})$ and $\tilde{W}_{h}$ are sufficiently far from the boundary of $S_1$, it follows that $\rho_{h}$ can be extended to $S_1$ in such a way that $\rho_{h}$ is bi-Lipschitz on $S_1$ with Lipschitz constant $1+10^{-(j+7)}$ and is identity on the boundary of $S_1$. See Figure \ref{f:cano12}.

\begin{figure}[h]
\begin{center}
\includegraphics[width=\textwidth]{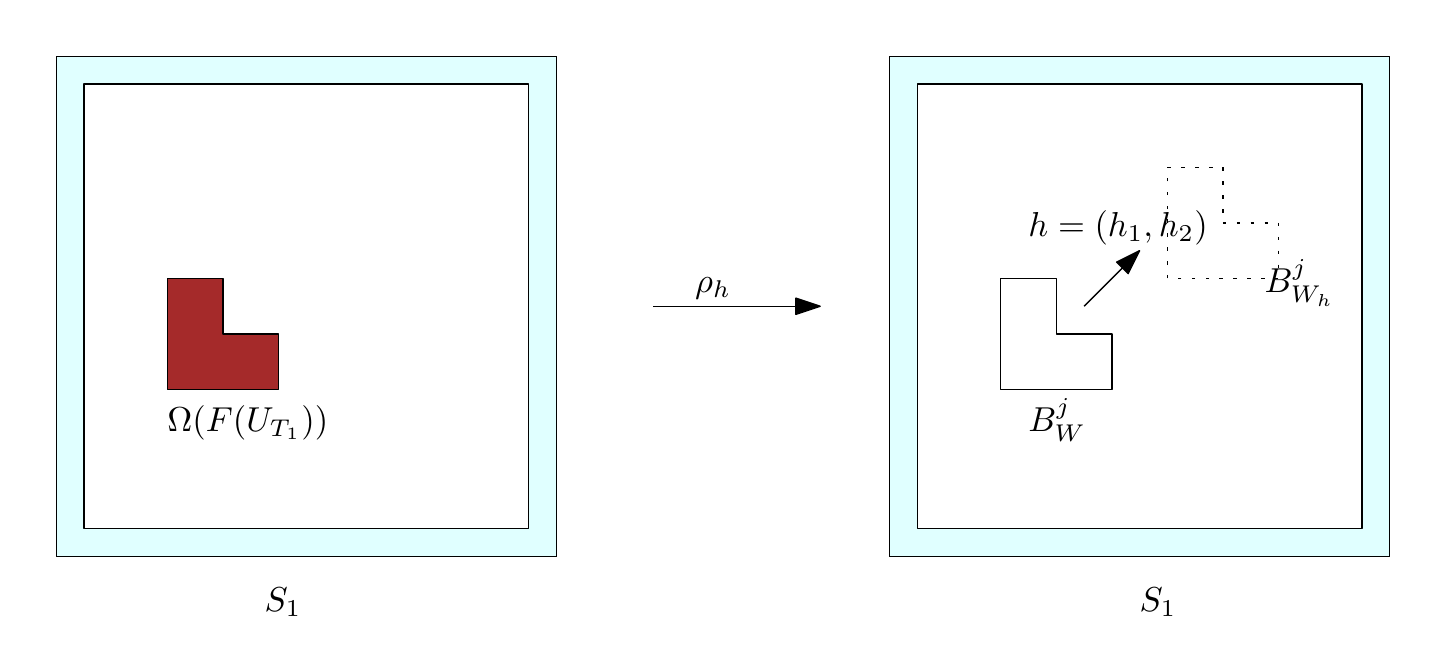}
\end{center}
\caption{Construction of $\rho_{h}$ on $S_1$ in the proof of Proposition \ref{p:canonical1}}
\label{f:cano12}
\end{figure}

It is now easy to check that $\Upsilon_h$ as defined above does indeed produce a sequence of $\alpha$-canonical maps satisfying the conditions in the proposition. This completes the proof.
\end{proof}

Finally we want to construct $\alpha$-canonical maps that match up bad sub-components and ensures that interior of the corresponding multi-cell is mapped into the interior of the multi-cell itself. This property is needed to make sure certain boundaries are valid; see Lemma \ref{l:A1Map2h}. We have the following proposition.

\begin{proposition}
\label{p:canonical2}
Fix $U\subseteq \Z^2$. Let $\tilde{U}_1$ be any potential domain corresponding to the $j+1$-level multi-cell $B^{j+1}_{U}$. Let $U_{1}\subseteq \Z^2$ be such that $B^j_{U_{1}}$ denotes the collection of all level $j$-cells that intersect $\tilde{U}_1$. Let $U_{2}\subseteq U_{1}$ be the set of all vertices in $U_{1}$ such that the distance from the boundary of $U_{1}$ is at least $\frac{L_j^3}{2}$. Let $\mathcal{T}=\{T_1,T_2,\ldots T_{\ell_1}\}$  be a set of disjoint and non-neighbouring subsets of $U_{2}$ such that $\sum |T_i|\leq v_0k_0$. Let $U_3\subseteq U_2$ be such that $B^j_{U_3}=B^{j+1,{\rm int}}_{U}$. Then there exists a sequence of $\alpha$-canonical maps $\{\Upsilon_{h_1,h_2}\}_{(h_1,h_2)\in [L_j^2]^2}$ from $\tilde{U}_1$ to $B^{j+1}_{U}$ with respect to $\mathcal{T}$ and $\emptyset$ satisfying the following conditions.
\begin{enumerate}
\item[\rm i.] For each $i\in [\ell_1]$, $\Upsilon_{h_1,h_2}(T_i)=(h_1-1,h_2-1)+\Upsilon_{1,1}(T_i)$.

\item[\rm ii.] For $T_i$ not contained in $U_2\setminus U_3$, and for all $h=(h_1,h_2)$ we have $\Upsilon_h(T_i)\subseteq U_3$.
\end{enumerate}
\end{proposition}

\begin{figure}[ht!]
\begin{center}
\includegraphics[width=\textwidth]{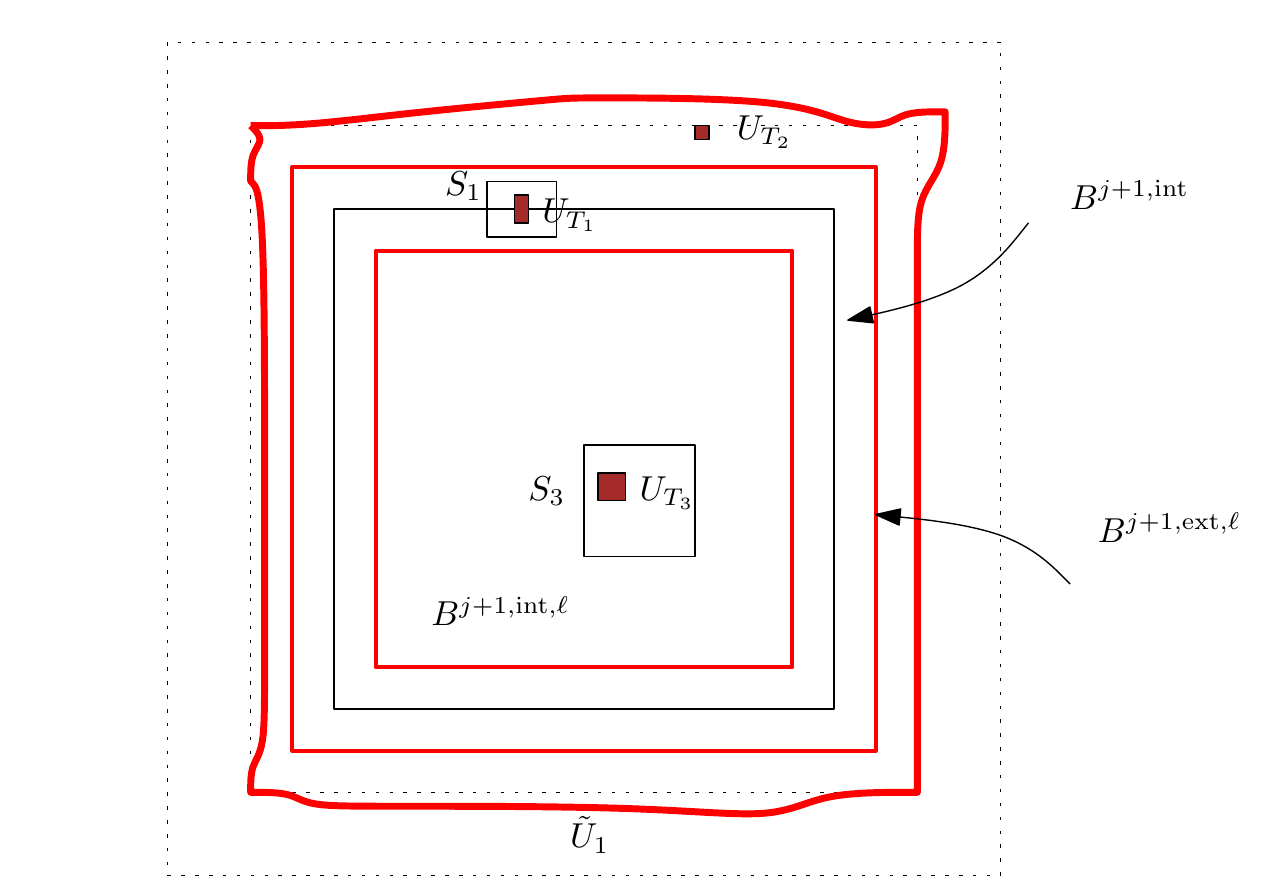}
\end{center}
\caption{Construction of $S_i$ as described in the proof of Proposition \ref{p:canonical2}}
\label{f:cano21}
\end{figure}

\begin{proof}
This proof is similar to the proof of Proposition \ref{p:canonical1} except that we have to do some extra work to ensure condition {\rm ii.} above. We use the same notations for domains as in the proof of Proposition \ref{p:canonical1}. Let $F$ be the canonical map that takes $\tilde{U}_1$ to $B^{j+1}_{U}$. Without loss of generality, we take $U$ to be the singleton $\{\mathbf{0}\}$. Define the squares $B^{j+1, {\rm int}, \ell}= [L_j^5+\ell L_j^4, L_{j+1}-L_j^5-\ell L_j^4]^2$ and $B^{j+1, {\rm ext}, \ell}=[L_j^5-\ell L_j^4, L_{j+1}-L_j^5+\ell L_j^4]^2$  such that $0<\ell <100k_0v_0$ and such that the distance of $F(U_{T_i})$'s from the boundaries of  $B^{j+1, {\rm int}, \ell}$ and $B^{j+1, {\rm ext}, \ell}$ is at least $L_j^4$. See Figure \ref{f:cano21}. Observe that by construction of canonical maps  $F$ is identity on $B^{j+1, {\rm ext}, \ell}$. Now as in the proof of Proposition \ref{p:canonical1}, it is not hard to see that there exist squares $S_{1},S_2,\ldots S_{k} \subseteq \tilde{U}_2$ with the following properties.
\begin{itemize}
\item We have that 
$$\cup_{i} F(U_{T_i}) \subseteq  \cup_{i=1}^{k} S_i.$$
%\item For a fixed $i\in [k]$, $S_i$ intersects at most one of $\cup_{i} F(U_{T_i})$ and $\cup_{i'} U_{T'_{i'}}$.
%\item Distance between $S_{i}$ and $S_{i'}$ for $i\neq i'$ is at least $L_j^{7/2}$.
\item The distance between the boundaries  of $S_i$ and the sets $F(U_{T_{\ell}})$ contained in it is at least $L_j^{7/2}$. 
\item The distance between $S_i$ and the boundaries of $B^{j+1, {\rm int}, \ell}$ and $B^{j+1, {\rm ext}, \ell}$ is at least $L_j^{7/2}$.
\end{itemize}
For $h=(h_1,h_2)\in [L_j^2]^2$, as before our strategy is to construct $\rho_{h}$ that is identity except on the interiors of $S_i$ and such that $\Upsilon_{h}=\rho_{h}\circ F$ is an $\alpha$-canonical map satisfying the conditions of the proposition. We construct $\rho_{h}$ separately on squares $S_i$. If $S_i\subseteq \tilde{U}_2\setminus B^{j+1, {\rm ext}, \ell}$ or $S_i\subseteq B^{j+1, {\rm int}, \ell}$ then the construction of $\rho_{h}$ proceeds as in the proof of Proposition \ref{p:canonical1}. We specify below the changes we need to consider if $S_i\subseteq B^{j+1, {\rm ext}, \ell}\setminus B^{j+1, {\rm int}, \ell}$. Without loss of generality take $S_1\subseteq B^{j+1, {\rm ext}, \ell}\setminus B^{j+1, {\rm int}, \ell}$ and also that $U_{T_1}$ is the only one (among $U_{T_i}$'s) that is contained in $S_1$.  

Recall that we only need to worry about condition {\rm ii.} in the statement of the proposition being violated if $T_1$ is not contained in $U_2\setminus U_3$. Let us assume that to be the case.  Notice that the assumptions on $S_1$ and $U_{T_1}$ imply that there exists $W$ which is a translate of $T_1$ such that $B^j_{W}$ has distance at most $10L_j^3$ from $U_{T_1}=F(U_{T_1})$ for all $h\in [L_j^2]^2$ and  $W_{h}:=h+W$ we have that $W_h\subseteq U_3$. See Figure \ref{f:cano22}. With this choice of $W$, construct $\rho_h$ exactly as in the proof of Proposition \ref{p:canonical1} and it is easy to verify that $\Upsilon_h=\rho_h\circ F$ satisfies the conclusion of the proposition. This completes the proof.
\end{proof}

\begin{figure}[h]
\begin{center}
\includegraphics[width=\textwidth]{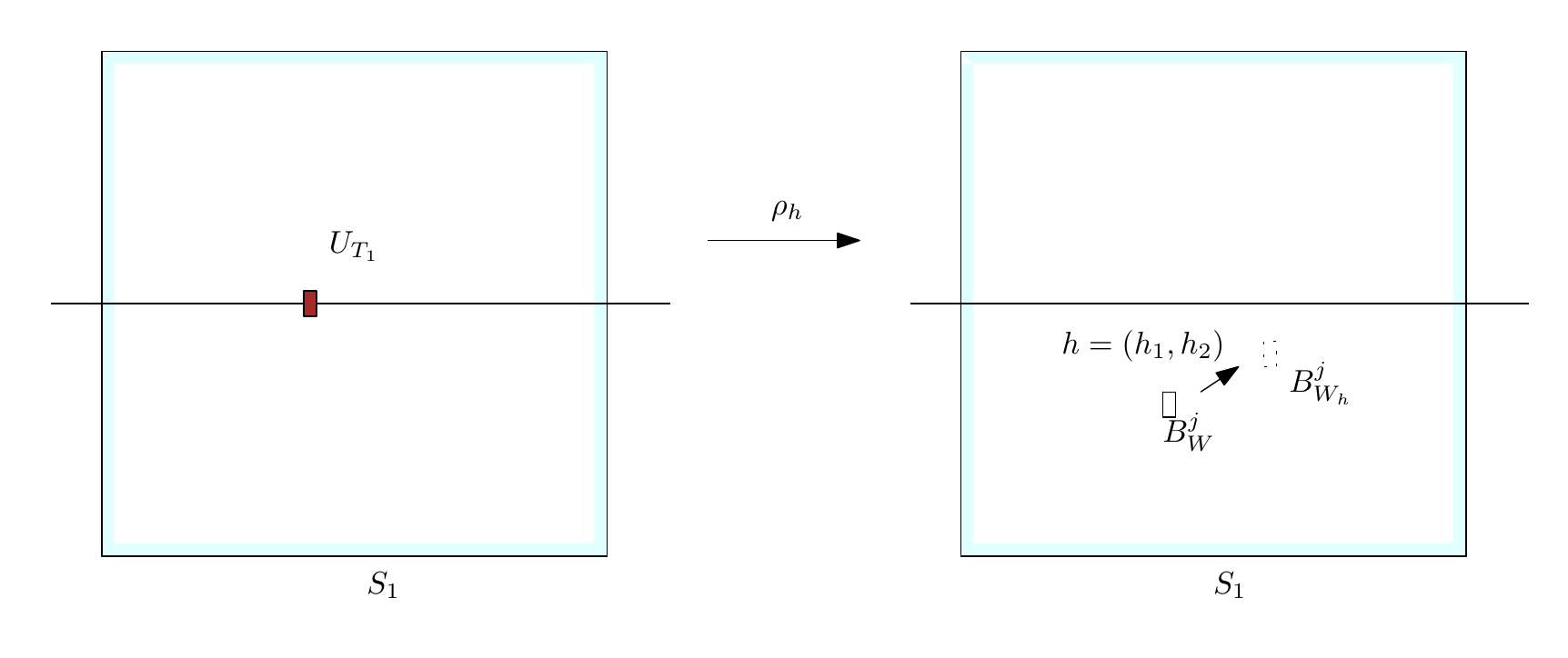}
\end{center}
\caption{Construction of $\rho_{h}$ on $S_1$ in the proof of Proposition \ref{p:canonical2}}
\label{f:cano22}
\end{figure}

\section{Tail Estimates}
\label{s:tailestimateh}
We now start with proving our recursive estimates at level $(j+1)$. The most important of our inductive hypotheses is the following recursive estimate. Let $X=X^{j+1}_{U}$ and $Y=Y^{j+1}_{U'}$ be random  $(j+1)$-level components in $\mathbb{X}$ and $\mathbb{Y}$ having laws $\mu_{j+1}^{\mathbb{X}}$ and $\mu_{j+1}^{\mathbb{Y}}$ respectively. Let $V_{X},V_{Y}$ denote the sizes of $X$ and $Y$ respectively.
%
%
%Let the shape of $X$ be denoted as $H(X)$ and let $V_{X}$ deonte the size of $H(X)$, that is $V_X$ denotes the number of ideal blocks spanned by $X$. Let $X_{\ell_1}$. $X_{\ell_2}$, $X_{\ell_{K_X}}$ denote the $j$ level bad components contained in $X$. Let $U_{X}$ deonote the the total size of these bad components, i.e., $U_{X}=\sum_{i=1}^{K_X} V_{X_{\ell_i}}$. We define $V_Y$, $U_Y$ similarly. 
We have the following theorem establishing \eqref{e:tailx1h} and \eqref{e:taily1h} at level $j+1$.

\begin{theorem}\label{t:tailh}
In the above set-up, we have for all $v\geq 1$ and all $p\leq 1-L_{j+1}^{-1}$
%Assume that the inductive hypotheses hold up to level $j$. Let $X$ and $Y$ be random  $(j+1)$-level components in $\mathbb{X}$ and $\mathbb{Y}$ respectively. Then
\begin{align*}
\mathbb{P}(S^\mathbb{X}_{j+1}(X)\leq p, V_{X}\geq v)\leq p^{m_{j+1}} L_{j+1}^{-\beta} L_{j+1}^{-\gamma (v-1)}; \\
\mathbb{P}(S^\mathbb{Y}_{j+1}(Y)\leq p, V_{Y}\geq v)\leq p^{m_{j+1}} L_{j+1}^{-\beta}L_{j+1}^{-\gamma (v-1)}
\end{align*}
where  $m_{j+1}=m+2^{-(j+1)}$.
\end{theorem}

Due to an obvious symmetry between our $X$ and $Y$ bounds, we shall state all our bounds in terms of $X$ and $S^\mathbb{X}_{j+1}$ but will similarly hold for $Y$ and $S^\mathbb{Y}_{j+1}$. We shall drop the superscript $\mathbb{X}$ for the rest of this section. 

As a consequence of translation invariance we can assume without loss of generality that $X=X^{j+1}_{U}=X^{*,j+1}(\mathbf{0})$, i.e., $U$ is the lattice component containing the origin. Let $U_{X}$ denote the domain of $X$. Also let $\tilde{U}\subseteq \Z^2$ be such that $X^{j}_{\tilde{U}}=X^{j+1}_{U}=X$. Let $X^{j}_{U_1}, X^{j}_{U_2}, \ldots, X^j_{\tilde{U}_{N_{X}}}$ denote the $j$-level bad subcomponents of $X$. Let $K_{X}$ denote the total size of the bad subcomponents, i.e., $K_{X}=\sum_{i=1}^{N_{X}}|U_i|$.
%
% is the $(j+1)$-level block of $\mathbb{X}$ containing the block corresponding to the ideal block $(1,1)$. 
Our first order of business is to obtain a bound on the probability that a component $X$ has either large $V_X$, or large $K_X$ or small $\prod_{i=1}^{N_{X}} S_j(X^j_{U_i})$. The following proposition is the key estimate of the paper.

\begin{proposition}\label{l:totalSizeBoundh}
Let $X$ be as above. For all $v'\geq 1, k,x\geq 0$ we have that
\[
\P\left[V_X \geq v', K_X \geq k, -\log \prod_{i=1}^{N_X}S_j(X^j_{U_i}) > x\right] \leq 500 L_j^{-\gamma k /10}\exp(-x m_{j+1}) L_{j+1}^{-9\gamma (v'-1)}.
\]
\end{proposition}

%\begin{proof}
For brevity of notation we shall write $S^*(X)=\prod_{i=1}^{N_X}S_j(X^j_{U_i})$. For $v\geq 1$, let $\mathcal{H}_{v}$ denote the set of all lattice animals of size $v$ containing $\mathbf{0}$. Clearly, we have 
\begin{equation}
\label{e:tailsum1}
\P\left[V_X \geq v', K_X \geq k, -\log S^*(X)> x\right] = \sum_{v=v'}^{\infty}\sum_{H\in \mathcal{H}_v} \P\left[U=H, K_X \geq k, -\log S^*(X)> x\right].
\end{equation}

To begin with, let us analyse the event $\{U=H\}$. Let $\hat{W}\subseteq \Z^2$ be such that $B^j_{\hat{W}}=B^{j+1,{\rm ext}}_{H}$, i.e., $hat{W}$ corresponds to the $j$-level cells contained in the blow-up of the level $(j+1)$ ideal multi-block $B^{j+1}_{H}$. Observe that on $\{U=H\}$, there exists a subset $H^*\subseteq H$ with at least $\lceil \frac{v}{25} \rceil$ vertices that are non neighbouring (in the closed packed lattice of $\Z^2$) and such that for all $h\in H^*$, the ideal multi-block containing $B^{j+1}_{H}$ must correspond to a bad block. Hence, for each $h\in H^*$, one of the following there events must hold for the cell $B^{j+1}(h)$: $(a)$ it has a conjoined buffer zone or the total size of $j$ level bad components contained in its blow up is at least $k_0$, $(b)$ it contains a really bad $j$-level subblock, $(c)$ it fails the airport condition. Hence at least one of these conditions must must hold for at least $\frac{v}{75}$ many $(j+1)$-level cells among the cells corresponding to the vertices of $H^*$. Hence 
$$\{U=H\}\subseteq A_1\cup A_2 \cup A_3$$ where $A_i$ are defined as follows. 
\begin{enumerate}
\item[$\bullet$] Let $A_1$ denote the event that total size of $j$-level bad components contained in $X^j_{\hat{W}}$ is at least $\frac{k_0v}{75}$.

\item[$\bullet$] Let $A_2$ denote the event that the total number of really bad components contained in $X^j_{\hat{W}}$ is at least $\frac{v}{75}$.

\item[$\bullet$] Finally let $A_3$ denote the event that there exists a subset $H'\subseteq H$ of non-neighbouring vertices with $|H'|= \frac{v}{75}$ such that $\cap_{h\in H'} \mathcal{S}_{h}$ holds where $S_{h}$ is the following event. For $h\in H'$, let $G_{h}$ be such that $B^j_{G_{h}}=B^{j+1, {\rm ext}}_{h}$. Then $S_{h}$ denotes the event that the following two conditions hold.
\begin{itemize}
\item[\rm i.] The total size of bad components at level $j$ contained in $B^j_{G_h}$ is at most $k_0$.
\item[\rm ii.] There exists a square $S\subseteq G_{h}$ of size $L_{j}^{3/2}$ such that $B^{j}_{S}$ is not an airport at level $j$.  
\end{itemize}
\end{enumerate}

Fix $v\geq v'$ and $H\in \mathcal{H}_{v}$ for now. The corresponding term in the right hand side of \eqref{e:tailsum1} can be upper bounded by 
$$\sum_{i=1}^{3} \P[-\log S^*(X)>x, K_{X}\geq k, A_i].$$
We shall treat the three cases separately.

\begin{lemma}
\label{l:tailsumcase1}
In the above set-up, we have 
$$\P[-\log S^*(X)>x, K_{X}\geq k, A_1] \leq  2\exp(-xm_{j+1})L_j^{-\gamma k/10}L_{j+1}^{-10\gamma(v-1)}.$$
\end{lemma}

\begin{proof}
Fix $k'\geq k$. Fix a collection $\mathcal{T}_{k'}=\{T_1,T_2, \cdots, T_n\}$ of non-neighbouring subsets of $\hat{W}$ with $\sum_{i}|T_i|=k'$. Let $\cf_{\mathcal{T}_{k'}}$ denote the event that $X^{j}_{T_i}$ is a $j$-level bad component of $\mathbb{X}$ for each $i$. It follows that we have 
\begin{equation}
\label{e:tailfixlocation}
\P\left[U=H, K_X \geq k, -\log S^*(X)> x\right]\leq \sum_{k'=k}^{\infty}\sum_{\mathcal{T}_{k'}} \P[-\log \prod_{i=1}^{n}S^j(X^j_{T_i})>x, \cf_{\mathcal{T}_{k'}}, U=H].
\end{equation}
Notice that on the event $\cf_{\mathcal{T}_{k'}}$, $X^j_{T_i}$ are independent. Observe that on $A_1$, we have $K_X\geq \frac{k_0v}{24}$. Now fix $\mathcal{T}_{k'}$. Set $t_{i}=|T_i|$. Let $\mathscr{V}_i,i=1,2,\ldots ,n$ be a sequence of independent random variables with  $\mbox{Ber}(L_{j}^{-\gamma t_i/2})$ distribution. Let $\mathscr{R}_i, i=1,2,\ldots ,n$ be a sequence of i.i.d.\ $\mbox{exp}(m_j)$ random variables independent of $\{\mathscr{V}_i\}$. It follows from the recursive estimates that 
$$ -\log S_j(X^j_{T_i})1_{\{X^j_{T_i}~\text{bad component}\}} \preceq \mathscr{V}_i(1+\mathscr{R}_i)$$
for all $i$ where $\preceq$ denotes stochastic domination. It follows that  
\begin{eqnarray}
\label{e:keylemmacase1}
\P[-\log \prod_{i=1}^{n}S_j(X^j_{T_i})>x, \cf_{\mathcal{T}_{k'}}, A_1] &\leq & \P[\mathscr{V}_i=1 \forall i]\P[\sum_{i=1}^{n} (1+\mathscr{R}_i)>x]\nonumber\\
&\leq & L_j^{-\gamma k'/2}\P[\sum_{i=1}^{n} \mathscr{R}_{i}> x-n]. 
\end{eqnarray}

Now observe that $\sum_{i=1}^{n} \mathscr{R}_i$ has a $\mbox{Gamma}(n, m_{j})$ distribution and hence
\begin{equation}
\label{e:gammadisth}
\P[\sum_{i=1}^{n} \mathscr{R}_i> x-n] =\int_{(x-n)\vee 0}^{\infty} \frac{m_j^{n}}{(n-1)!} y^{n-1} \exp(-y m_j) dy.
\end{equation}

Following the proof of Lemma 7.3 in \cite{BS14} it follows from this that
\begin{equation}
\label{e:gammadist2h}
\P[\sum_{i=1}^{n} \mathscr{R}_i> x-n] \leq (m_j 2^{j+1}e^{m_{j+1}})^{n} \exp(-x m_{j+1}).
\end{equation}

Since $L_j$ grows doubly exponentially and $n\leq k'$ and $k'>vk_0/24$ it follows from (\ref{e:keylemmacase1}) that for $L_0$ sufficiently large we have 
\begin{equation}
\label{e:keylemma1finalh}
\P[-\log \prod_{i=1}^{n}S_j(X_{T_i})>x, \cf_{\mathcal{T}_{k'}}, A_1]\leq L_j^{-\gamma k'/4}\exp(-xm_{j+1}) \leq L_j^{-\gamma k'/8}\exp(-xm_{j+1})L_{j+1}^{-10\gamma(v-1)}  
\end{equation}
as $k_0> 6000\alpha\gamma$.

Now notice that total number of choices for $\mathcal{T}_{k'}$ is bounded by $16^{v}3^{k'}L_{j+1}^{k'}$ hence summing over all such choices and then summing over all $k'$ from $k$ to $\infty$ we get the desired result as $\gamma>40\alpha$ and $L_j$ is sufficiently large.  
\end{proof}

\begin{lemma}
\label{l:tailsumcase2}
In the set-up of Lemma \ref{l:tailsumcase1}, we have 
$$\P[-\log S^*(X)>x, K_{X}\geq k, A_2] \leq  \exp(-xm_{j+1})L_j^{-\gamma k/10}L_{j+1}^{-10\gamma(v-1)}.$$
\end{lemma}

\begin{proof}
Fix $k'\geq k$ and $\mathcal{T}_{k'}$ as in the proof of Lemma \ref{l:tailsumcase2}. Fix a subset $\mathcal{N}$ of $[n]$ with $|\mathcal{N}|=\frac{v}{24}$. Now, for $i\in \mathcal{N}$, $X^j_{T_{i}}$ can be a really bad component in one of two ways: (a) $t_{i_{\ell}}\geq v_0$ and (b) $S^j(X_{T_{i_{\ell}}})\leq 1-L_{j}^{-1}$. Observe that it follows from the recursive estimates that for all $i\in \mathcal{N}$ we have 
$$ -\log S_j(X^j_{T_i})1_{\{X^j_{T_{i}}~\text{really bad component}\}} \preceq \mathscr{W}_{i}(1+\mathscr{R}_i)$$
where $\{\mathscr{W}_{i}\}_{i\in \mathcal{N}}$ is a sequence of i.i.d.\  
$\mbox{Ber}(L_{j}^{-\gamma t_{i}/4- (\gamma v_0/4 \wedge \beta/2)})$ distribution. It follows that
\begin{eqnarray}
\label{e:keylemmacase2}
\P[-\log \prod_{i=1}^{n}S_j(X^j_{T_i})>x, \cf_{\mathcal{T}_{k'}}, A_2] &\leq & \sum_{\mathcal{N}} \P[\mathscr{V}_{i}=1 \forall i\in [n]\setminus \mathcal{N}, \mathscr{W}_{i}=1 \forall i\in \mathcal{N}]\P[\sum_{i=1}^{n} (1+\mathscr{R}_i)>x]\nonumber\\
&\leq & \sum_{\mathcal{N}} L_j^{-\gamma k'/4- \frac{v}{300}(\gamma v_0 \wedge 2\beta)}\P[\sum_{i=1}^{n} \mathscr{R}_i> x-n]\nonumber\\
& \leq & \binom{k'}{\frac{v}{75}}L_j^{-\gamma k'/4- \frac{v}{300}(\gamma v_0 \wedge 2\beta)}\P[\sum_{i=1}^{n} U_i> x-n]. 
\end{eqnarray}
Doing the same calculations as in the proof of Lemma \ref{l:tailsumcase1}, we obtain that 
\begin{eqnarray}
\label{e:keylemma2final}
\P[-\log \prod_{i=1}^{k}S_j(X^j_{T_i})>x, \cf_{\mathcal{T}_{k'}}, A_2] &\leq & 2^{k'} L_j^{-\gamma k'/5}\exp(-xm_{j+1})L_j^{-\frac{v}{300}(\gamma v_0 \wedge 2\beta)}\nonumber\\
&\leq & L_j^{-\gamma k'/8}\exp(-xm_{j+1})L_{j+1}^{-10\gamma(v-1)}  
\end{eqnarray}
since $\gamma v_0 \wedge 2\beta >3000\alpha\gamma$.

As before, summing over all $\mathcal{T}_{k'}$ and $k'$ from $k$ to $\infty$ gives the result. 
\end{proof}

\begin{lemma}
\label{l:tailsumcase3}
In the set-up of Lemma \ref{l:tailsumcase1}, we have 
$$\P[-\log S^*(X)>x, K_{X}\geq k, A_3] \leq  \exp(-xm_{j+1})L_j^{-\gamma k/10}L_{j+1}^{-10\gamma(v-1)}.$$
\end{lemma}

\begin{proof}
First fix $H'\subseteq H$ as in the definition of $A_3$. Fix $k'\geq k$ and $\mathcal{T}_{k'}$ as before. Now fix $h\in H'$. Set $I_{h}=(\cup_{i=1}^n T_i) \cap G_{h}$ and observe that by hypothesis $|I_h|\leq k_0$. It is not too hard to see that there exists an event $S'_{h}$ such that $S_{h}\subseteq S'_{h}$ and $S'_{h}$ is independent of $X^j_{I_h}$ and $P[S'_{h}]\leq L_{j}^{-10\beta}$. Indeed, that a square is an airport can be verified, even without checking a limited number of cells, and this can be established using arguments identical to the proof of Lemma \ref{l:airportfinal}, we omit the details.
 
Repeating the same calculations as in the proofs of Lemma \ref{l:tailsumcase1} and Lemma \ref{l:tailsumcase2} it then follows that   

\begin{eqnarray*}
\P[-\log S^*(X)>x, K_{X}\geq k, A_3] &\leq &  \sum_{H'}\sum_{k'=k}^{\infty}\sum_{\mathcal{T}_{k'}} \P[\mathcal{F}_{\mathcal{T}_{k'}}]L_j^{-10\beta v/75}\\
&\leq & 2\binom{v}{\frac{v}{75}}\exp(-xm_{j+1})L_{j}^{-\gamma k/10}L_j^{-10\beta v/75}\\
&\leq & \exp(-xm_{j+1})L_{j}^{-\gamma k/10}L_{j+1}^{-10\gamma(v-1)}
\end{eqnarray*}
as $L_j$ is sufficiently large and $\beta > 75\alpha \gamma$.
  
%
%It is not too hard to see that there exists an event  
%
%Notice that for $h\in H'$, the events $S_{h'}$ can be checked on $B^{j+1}_{h}$ and hence are independent. It is possible to show that for each $h\in H'$,
%$$\P[S_{h}]\leq L_j^{-\beta k_0}.$$
%See the proof of Lemma \ref{l:airportfinal} for the argument. Hence it follows that for each $k\leq \frac{vk_0}{24}$
\end{proof}

Putting together all the cases we are now ready to prove Proposition \ref{l:totalSizeBoundh}.

\begin{proof}[Proof of Proposition \ref{l:totalSizeBoundh}]
Notice that for a fixed $v$, we have $|\mathcal{H}_{v}|\leq 8^v$. We now get from \eqref{e:tailsum1}, Lemmas \ref{l:tailsumcase1}, \ref{l:tailsumcase2}, \ref{l:tailsumcase3} by summing over all $H\in \mathcal{H}_{v}$ and then finally summing over all $v$ from $v'$ to $\infty$ that 

\begin{eqnarray*}
\P\left[V_X \geq v', K_X \geq k, -\log S^*(X)> x\right] &\leq & \sum_{v=v'}^{\infty} 8^v50L_{j}^{-\gamma k/10}\exp(-xm_{j+1})L_{j+1}^{-10\gamma(v-1)}\\
&\leq & 500\exp(-xm_{j+1})L_{j}^{-\gamma k/10}L_{j+1}^{-9\gamma(v'-1)},
\end{eqnarray*}
this completes the proof of the proposition.
\end{proof}

We now move to the proof of Theorem \ref{t:tailh}. Our proof will be divided into four cases depending on the size of $X$, the total size of its bad components and how bad the bad components are. In each one we will use  different $\alpha$-canonical map or maps to get good lower bounds on the probability that $X=X^{j+1}_{U} \hookrightarrow Y^{j+1}_{U}$. We now present our four cases.
%
%our three cases.  

\subsection{Case 1}

The first case is the generic situation where the components are of small size, have small total size of bad sub-components whose embedding probabilities are not too small. For a $(j+1)$-level component $X$, let $N_{X}$ denote the number of bad $j$ level components contained in $X$ and let $X^j_{T_1}, X^j_{T_2},\ldots, X^j_{T_{N_X}}$ denote the bad subcomponents. Let $K_{X}=\sum_{i=1}^{N_X}|T_i|$ denote the total size of bad subcomponents in $X$. We define the class of blocks $\A{1}_{X,j+1}$  as
\[
\A{1}_{X,j+1} := \left\{X: V_{X}\leq v_0, K_X\leq k_0v_0,  \prod_{i=1}^{N_X}S_j(X^j_{T_i}) \geq L_j^{-1/3} \right\}.
\]

First we show that this case holds with extremely high probability. 

\begin{lemma}\label{l:A1Sizeh}
The probability that $X\in \A{1}_{X,j+1}$ is bounded below by
\[
\P[X\not\in \A{1}_{X,j+1}] \leq L_{j+1}^{-3\beta}L_{j+1}^{-\gamma(v_0-1)}.
\]
\end{lemma}
\begin{proof}
This follows from Proposition \ref{l:totalSizeBoundh} by noting $8\gamma(v_0-1)> 3\alpha\beta$, $m\geq 9\alpha \beta +3\alpha \gamma v_0$ and  $\gamma k_0v_0> 300\alpha \beta + \alpha \gamma v_0$. We omit the details.
\end{proof}

Next we show that $S_{j+1}(X)$ is at least $1/2$ for all $X\in \A{1}_{X,j+1}$.

\begin{lemma}\label{l:A1Map1h}
Condition on $X=X^{j+1}_{U}\in \A{1}_{X,j+1}$ where $U\subseteq \Z^2$ and $|U|\leq v_0$. Let the bad $j$ level components of $X$ be $X^j_{T_1}, X^j_{T_2},\ldots , X^j_{T_{N_X}}$ such that $\sum_{i=1}^{N_{X}}|T_i|\leq v_0k_0$. Then we have  
\[
S_{j+1}(X) \geq \frac{1}{2}.
\]
\end{lemma}

\begin{proof}
Let $U_{X}$ denote the domain of $X$, $C_{X}$ denote the boundary of $X$, and let $\tilde{U}\subseteq \Z^2$ be such that $X=X^j_{\tilde{U}}$. By Proposition \ref{p:canonical1}, there exist $L_j^4$ $\alpha$-canonical maps at $j+1$-th level $\{\Upsilon^{j+1}_h=\Upsilon_h: h\in [L_j^2]^2\}$ from $U_{C}$ to $B^{j+1}$ with respect to $\mathcal{T}={T_1,T_2,\ldots, T_{K_{X}}}$ and $\emptyset$ such that $\Upsilon_{h}(T_i)$ are different for all $h\in [L_j^2]^2$.  
%\end{enumerate} 

Clearly there exists a subset $\mathcal{H}\subset [L_j^2]^2$ with $|\mathcal{H}|= L_j<\lfloor L_j^4/ 100v_0^4k_0^4\rfloor$ so that for all $i_1\neq i_2$ and $h_1,h_2\in\mathcal{H}$ we have that $\Upsilon_{h_1}(T_{i_1})$ and  $\Upsilon_{h_2}(T_{i_2})$ are disjoint and non-neighbouring. We will estimate the probability that one of these maps work.

For $h\in \mathcal{H}$ and $i\in [N_X]$, let $\cd^i_h$ denote the event

\[
\cd^i_h=\left\{Y^j_{\Upsilon_{h}(T_i)}~\text{valid}, X^j_{T_i} \hookrightarrow Y^j_{\Upsilon_{h}(T_i)} \right\}.
\]

Since we are only trying out non-neighbouring components, these events are conditionally independent given $X$ and setting
\[
\cd_h=\bigcap_{1\leq i \leq N_X} \cd^i_h
\]
we get 
$$\P[\cd_h\mid X]= \prod_{i=1}^{N_X}S_j(X_{T_i})\geq L_j^{-1/3}.$$ 
By construction of $\ch$, we also get that $\{\cd_{h}:h\in \ch\}$ are mutually independent given $X$ and hence setting $\cd=\cup_{h\in \ch} \cd_{h}$ we have 

\begin{equation}
\label{e:A1Map1h1}
\P[\cd \mid X]\geq 1-(1-L_{j}^{-1/3})^{L_j}\geq 1-L_{j+1}^{-3\beta}.
\end{equation}

Let $\mathcal{B}$ denote the event that $Y^{j+1}_U$ is valid and has domain $B^{j+1}_{U}$. Let $\tilde{U}_2\subseteq \Z^2$ be such that $\tilde{U}_2$ corresponds to the $j$-level cells contained in the blow-up of $B^{j+1}_{U}$, i.e., $B^{j}_{\tilde{U}_2}=B^{j+1,{\rm ext}}_{U}$. Let $\cj$ denote the event
\[
\cj=\left\{Y^j(\ell) \hbox{ is good for all } \ell\in \tilde{U}_2\right\}.
\]
By Lemma \ref{l:embedcond} on $\cd \cap \cj \cap \mathcal{B}$, there exists an embedding and hence $S_{j+1}(X)\geq \P[\cd \cap \cj \cap \mathcal{B}\mid X]$. Using \eqref{e:goodconditiony} at level $j$ we get
\begin{equation}\label{e:A1MapDh}
\P[\cj \mid X] \geq \left(1- L_j^{-\gamma}\right)^{4v_0L_j^{2\alpha-2}} \geq 9/10
\end{equation}
as $\gamma>2\alpha$ and $L_j$ is sufficiently large. Notice that on $\cj$, $Y^{j+1}_{U}$ is valid and $B^{j+1}_{U}$ is a valid potential domain of $Y^{j+1}_{U}$ and by construction
\begin{equation}
\label{e:A1Map1h2}
\P[\mathcal{B} \cap \cj \mid X]\geq \P[\cj\mid X]\P[\mathcal{B}\mid \cj, X] \geq \frac{9}{10} (1-10^{-(j+10)})^{4v_0}\geq \frac{3}{5}. 
\end{equation}
The lemma follows from \eqref{e:A1Map1h1} and \eqref{e:A1Map1h2} since $L_0$ is sufficiently large. 
\end{proof} 
 
Now to improve upon the above estimate, we want to relax the condition that $Y$ does not contain any bad-components by weaker conditions that define a generic block. We proceed as follows. 

Let $\mathcal{S}=\mathcal{S}_{v_0}$ denote the set of all lattice animals containing the set $\{0\}$ and having size at most $v_0$. For $S\in \mathcal{S}_{v_0}$, let $\mathcal{U}_{S}$ denote the set of all potential domains for $X^{j+1}_{S}$ (or $Y^{j+1}_{S}$). Also set $S_0=\cup_{S\in \mathcal{S}}S$. Let $T_1,T_2, \ldots , T_{N_{X}}$ be subsets of $\Z^2$ such that $\{X^j_{T_i}:i\in [N_X]\}$ are the $j$-level bad subcomponents in $X^{j+1}_{S_0}$. Similarly let $T'_1,T'_2, \ldots , T'_{N'_{Y}}$ be subsets of $\Z^2$ such that $\{Y^j_{T_i}:i\in [N'_Y]\}$ denote the $j$-level bad subcomponents in $Y^{j+1}_{S_0}$. Fix $S\in \mathcal{S}$ and $\tilde{U}\in \mathcal{U}_{S}$. Let $\tilde{U}\subseteq \Z^2$ denote the set such that $X^{j+1}_{S}=X^j_{\hat{U}}$ on the event that $\tilde{U}$ is the domain of $X^{j+1}_{S}$. Let $B_{\tilde{U},X}=B_{\tilde{U},S,X}=\{i\in [N_X]: T_i\subseteq \hat{U}\}$ and let $B_{\tilde{U},Y}$ be defined similarly. Let 
$$\ce_{\tilde{U},X}=\left\{X^{j+1}_S~\text{valid},\tilde{U}~\text{valid},\sum_{i\in B_{\tilde{U},X}}|T_i|\leq k_0v_0, \prod_{i\in B_{\tilde{U},X}}S_j(X^j_{T_i}) \geq L_j^{-1/3}\right\}$$   
and
$$\ce_{\tilde{U},Y}=\left\{Y^{j+1}_S~\text{valid},\tilde{U}~\text{valid},\sum_{i\in B_{\tilde{U},Y}}|T'_i|\leq k_0v_0, \prod_{i\in B_{\tilde{U},Y}}S_j(Y^j_{T'_i}) \geq L_j^{-1/3}\right\}.$$ 

Finally let $\mathcal{B}_{\tilde{U},X}$ (resp.\ $\mathcal{B}_{\tilde{U},Y}$) denote the event that the domain of $X^{j+1}_{S}$ (resp.\ $Y^{j+1}_{S}$) is $\tilde{U}$. We have the following lemma.

\begin{lemma}\label{l:A1Map2h}
We have that
\begin{equation}\label{e:A1MapE1h}
\sum_{S\in \mathcal{S}}\sum_{\tilde{U}_1\in \mathcal{U}_S}\sum_{\tilde{U}_2\in \mathcal{U}_S} \P[X^{j+1}_{S}\not\hookrightarrow Y^{j+1}_{S}, \ce_{\tilde{U}_1,X}, \ce_{\tilde{U}_2,Y}, \mathcal{B}_{\tilde{U}_2,X}, \mathcal{B}_{\tilde{U}_2,Y}] \leq L_{j+1}^{-3\beta}L_{j+1}^{-\gamma(v_0-1)}.
\end{equation}
\end{lemma}
\begin{proof}
Since $|\mathcal{S}|\leq 8^{v_0}$ and for all $S\in \mathcal{S}$ we have $\mathcal{U}_S\leq (8k0)^{16k_0v_0^2}$, it suffices to prove that for each fixed $S$, $\tilde{U}_1$ and $\tilde{U}_2$ we have
$$\P[X^{j+1}_{S}\not\hookrightarrow Y^{j+1}_{S}, \ce_{\tilde{U}_1,X}, \ce_{\tilde{U}_2,Y}, \mathcal{B}_{\tilde{U}_1,X}, \mathcal{B}_{\tilde{U}_2,Y}] \leq L_{j+1}^{-4\beta}L_{j+1}^{-\gamma(v_0-1)}.$$

Now fix $S\in \mathcal{S}$ and $\tilde{U}_1,\tilde{U}_2\in \mathcal{U}_{S}$. Notice that the total number of ways it is possible to choose disjoint sets $S_1, S_2, \ldots, S_{\ell_1}\subseteq \tilde{U}_{1}$ and $S'_1, S'_2, \ldots , S'_{\ell_2}\subseteq \tilde{U}_{2}$ such that $\sum_{i}|S_i|\leq v_0k_0$ and $\sum_{i}|S'_i|\leq k_0v_0$ is $L_j^{4\alpha k_0v_0}$. For $\mathscr{S}_1=\{S_1,S_2,\ldots S_{\ell_1}\}$ and $\mathscr{S}_2=\{S'_1,S'_2,\ldots S'_{\ell_2}\}$, let  
$$\mathscr{I}(\mathscr{S}_1,\mathscr{S}_2)=\left\{\mathscr{S}_1= \{T_i:i\in B_{\tilde{U}_1,X}\},\mathscr{S}_2= \{T'_i:i\in B_{\tilde{U}_2,Y}\}\right\}.$$

Clearly then it suffices to show that for each choice of $\mathscr{S}_1$ and $\mathscr{S}_2$ as above, we have 
\begin{equation}
\label{e:s1s2}
\P[X^{j+1}_{S}\not\hookrightarrow Y^{j+1}_{S}, \ce_{\tilde{U}_1,X}, \ce_{\tilde{U}_2,Y}, \mathcal{B}_{\tilde{U}_1,X}, \mathcal{B}_{\tilde{U}_2,Y}, \mathscr{I}(\mathscr{S}_1,\mathscr{S}_2)] \leq L_{j+1}^{-4\beta-8k_0v_0-\gamma v_0}.
\end{equation}

Fix $\mathscr{S}_1$ and $\mathscr{S}_2$ as above. Condition on $\{X^j_{S_i}:S_i\in \mathscr{S}_1\}$ and $\{Y^j_{S'_i}:S'_i\in \mathscr{S}_2\}$, such that they are compatible with $\ce_{\tilde{U}_1,X}$ and $\ce_{\tilde{U}_2,Y}$. Denote this conditioning by $\cf$. Observe that, by Proposition \ref{p:canonical1}, there exist $L_j^4$ $\alpha$-canonical maps at $j+1$-th level $\{\Upsilon^{j+1}_h=\Upsilon_h: h\in [L_j^2]^2\}$ from $\tilde{U}_{1}$ to $\tilde{U}_2$ with respect to $\mathscr{S}_1$ and $\mathscr{S}_2$ satisfying the following conditions.

\begin{enumerate}
\item[\rm i.] $\Upsilon_{h}(S_i)$ are different for all $h\in [L_j^2]^2$.
\item[\rm ii.] $\Upsilon_{h}^{-1}(S'_i)$ are different for all $h\in [L_j^2]^2$.
\item[\rm iii.] $\Upsilon_h(S_{i_1})\neq S'_{i_2}$ for any $i_1,i_2,h$. 
\end{enumerate} 
As before there exists a subset $\mathcal{H}\subset [L_j^2]^2$ with $|\mathcal{H}|= L_j<\lfloor L_j^4/ 100v_0^4k_0^4\rfloor$ so that the sets $\{\Upsilon_{h}(S_{i_1}):h\in \ch, S_{i_1}\in \mathscr{S}_1\}$ are disjoint and non-neighbouring and also disjoint and non-neighbouring with the sets in $\mathscr{S}_2$. Also the collection of sets $\{\Upsilon_{h}^{-1}(S'_{i_2}):h\in \ch, S'_{i_2}\in \mathscr{S}_2\}$ are disjoint and non-neighbouring and also disjoint and non-neighbouring with the sets in $\mathscr{S}_1$.

For $h\in \ch$, let $\cd_{h}$ denote the event

$$\cd_h=\left\{X^j_{S_{i_1}}\hookrightarrow Y^j_{\Upsilon_h(S_{i_1})},X^j_{\Upsilon_h^{-1}(S'_{i_2}}\hookrightarrow Y^j_{S'_{i_2}}~\forall S_{i_1}\in \mathscr{S}_1~\forall S'_{i_2}\in \mathscr{S}_2 \right\}.$$

Arguing as before we have 

$$\P[\cd_h \mid \cf]= \prod_{S_i\in \mathscr{S}_1} S_j^{\mathbb{X}}(X^j_{S_i})\prod_{S'_i\in \mathscr{S}_2} S_j^{\mathbb{X}}(X^j_{S'_i})\geq L_j^{-2/3}.$$

Since these events are independent for $h\in \ch$ it follows that 

$$\P[\cup_{h\in \ch} \cd_h \mid \cf] \geq 1-L_{j+1}^{-4\beta-8k_0v_0-\gamma v_0}$$
since $L_0$ is sufficiently large. Now observing that on $$\mathcal{B}_{\tilde{U}_1,X}\cap \mathcal{B}_{\tilde{U}_2,Y}\cap \mathscr{I}(\mathscr{S}_1\cap \mathscr{S}_2) \cap \biggl(\cup_{h\in \ch} \cd_h\biggr)$$
we have $X^{j+1}_{S}\hookrightarrow Y^{j+1}_{S}$ and removing the conditioning we get \eqref{e:s1s2} which in turn completes the proof of the lemma. 
\end{proof}

\begin{lemma}\label{l:A1finalh}
When $\frac12\leq p \leq 1-L_{j+1}^{-1}$
\[
\P(S_{j+1}(X)\leq p, V_X\geq v)\leq p^{m_{j+1}} L_{j+1}^{-\beta}L_{j+1}^{-\gamma(v-1)}.
\]
\end{lemma}

\begin{proof}
Clearly it is enough to show that 
\begin{equation}
\label{e:A1reduced}
\P[S_{j+1}(X)\leq 1-L_{j+1}^{-1}, V_X\geq v]\leq 2^{-m_{j+1}}L_{j+1}^{-\beta}L_{j+1}^{-\gamma(v-1)}.
\end{equation}

For $v\geq v_0$ this follows from Proposition \ref{l:totalSizeBoundh} and $8\gamma(v_0-1)>\beta$, so it suffices to prove that

\begin{equation}
\label{e:A1reduced2}
\P[S_{j+1}(X)\leq 1-L_{j+1}^{-1}, V_X\leq v_0]\leq 2^{-m_{j+1}}L_{j+1}^{-\beta}L_{j+1}^{-\gamma(v_0-1)}.
\end{equation}

Using Markov's inequality and Lemma \ref{l:A1Map2h} we get that 
$$\P[S_{j+1}(X)\leq 1-L_{j+1}^{-1}, V_X\leq v_0]\leq \P[X\notin \A{1}_{X,j+1}]+ L_{j+1}\left(L_{j+1}^{-3\beta}L_{j+1}^{-\gamma(v_0-1)}+\P[\ce_{Y}]\right)$$
where 
$$\P[\ce_{Y}]=\sum_{S\in \mathcal{S}}\sum_{\tilde{U}\in \mathcal{U}_S} \P[(\ce_{C,Y})^c].$$
It can be shown as in Lemma \ref{l:A1Sizeh} that $\P[(\ce_{Y})^c]\leq L_{j+1}^{-3\beta}L_{j+1}^{-\gamma(v-1)}$ and this completes the proof of the lemma. 
\end{proof}

\subsection{Case 2}

The next case involves components which are not too large and do not contain too many bad sub-components but whose bad sub-components may have very small embedding probabilities. For a $(j+1)$-level component $X$, let $N_{X}$ denote the number of bad $j$ level components contained in $X$ and let $X^j_{T_1}, X^j_{T_2},\ldots, X^j_{T_{N_X}}$ denote the bad subcomponents. Let $K_{X}=\sum_{i=1}^{N_X}|T_i|$ denote the total size of bad subcomponents in $X$. We define the class of blocks $\A{2}_{X,j+1}$  as
\[
\A{2}_{X,j+1} := \left\{X: V_{X}\leq v_0, K_X\leq k_0v_0,  \prod_{i=1}^{N_X}S_j(X^j_{T_i}) \leq L_j^{-1/3} \right\}.
\]

\begin{lemma}\label{l:A2Maph}
Condition on $X=X^{j+1}_{U}\in \A{2}_{X,j+1}$ where $|U|\leq v_0$. Let the bad $j$ level components of $X$ be $X^j_{T_1}, X^j_{T_2},\ldots , X^j_{T_{N_X}}$ such that $\sum_{i=1}^{N_{X}}|T_i|\leq v_0k_0$. Then we have  
\[
S_{j+1}(X) \geq \min\left\{\frac12, \frac{1}{10} L_j \prod_{i=1}^{N_X}S_j(X^j_{T_i}) \right\}.
\]
\end{lemma}

\begin{proof}
Let us make some notations first. Let $U_{X}$ denote the domain of $X$. Let $\tilde{U}\subseteq \Z^2$ be such that $X=X^j_{\tilde{U}}$. Let $\tilde{U}_1\subseteq \tilde{U}_0\subseteq \tilde{U}_2\subseteq \Z^2$ be defined as follows. 
$$B^j_{\tilde{U}_0}=B^{j+1}_{U};$$
$$B^j_{\tilde{U}_1}=B^{j+1,{\rm int}}_{U};$$
$$B^j_{\tilde{U}_2}=B^{j+1,{\rm ext}}_{U}.$$
%Let $\tilde{U}_0$ be the set corresponding to the $j$-level cells contained in the ideal multi-block $B^{j+1}_{U}$, and let $\tilde{U}_1$ (resp.\ $\tilde{U}_2$) denote the set corresponding to the $j$-level cells in the interior (resp.\ blow-up) of $B^{j+1}_{U}$.
Now by Proposition \ref{p:canonical2}, there exist $L_j^4$ $\alpha$-canonical maps at $j+1$-th level $\{\Upsilon^{j+1}_h=\Upsilon_h: h\in [L_j^2]^2\}$ from $U_{X}$ to $B^{j+1}_{U}$ with respect to $\mathcal{T}=\{T_1,T_2,\ldots, T_{N_{X}}\}$ satisfying the following conditions:

\begin{enumerate}
\item[\rm i.] For each $i\in [N_X]$ such that $T_{i}$ is not contained in $\tilde{U}_2\setminus \tilde{U}_1$ and for all $h\in [L_j^2]^2$ we have $\Upsilon_h(T_i)\subseteq \tilde{U}_1$.
\item[\rm ii.] For all $h\in [L_j^2]^2$ and for all $i\in [N_X]$ we have $\Upsilon_h(T_i)$ is at least at a distance $L_j^3$ from the boundaries of $U_0$.
\item[\rm iii.] $\Upsilon_{h}(T_i)$ are different for all $h\in [L_j^2]^2$.  
\end{enumerate} 

It is easy to see that there exists a subset $\mathcal{H}\subset [L_j^2]^2$ with $|\mathcal{H}|= L_j^{3/2}<\lfloor L_j^4/ 100v_0^4k_0^4\rfloor$ so that for all $i_1\neq i_2$ and $h_1,h_2\in\mathcal{H}$ we have that $\Upsilon_{h_1}(T_{i_1})$ and  $\Upsilon_{h_2}(T_{i_2})$ are disjoint and non-neighbouring. We will estimate the probability that one of these maps work.

In trying out these $L_j^{3/2}$ different mappings there is a delicate conditioning issue since a map failing may imply that $Y_{\Upsilon_{h}(T_i)}$ is bad. To avoid this we condition on an  event $\cd_h \cup \cg_h $ which holds with high probability.
For $h\in \mathcal{H}$ and $i\in [N_X]$, let $\cd^i_h$ denote the following event. If $T_i\subseteq \tilde{U}_2\setminus \tilde{U}_1$, then
\[
\cd^i_h=\left\{X^j_{T_i} \hookrightarrow Y^j_{\Upsilon_{h}(T_i)}; Y^j_{\ell} \hbox{ is good for all } \ell\in T_i \right\}.
\]
Otherwise set 
\[
\cd^i_h=\left\{X^j_{T_i} \hookrightarrow Y^j_{\Upsilon_{h}(T_i)} \right\}.
\]
Define
\[
\cd_h=\bigcap_{1\leq i \leq N_X} \cd^i_h.
\]
Also let
\[
\cg^i_h=\left\{Y^j_{\ell} \hbox{ is good for all } \ell\in T_i \right\}
\]
and 
\[
\cg_h=\bigcap_{1\leq i \leq N_X} \cg^i_h.
\]

Then using \eqref{e:goodconditiony} at level $j$ we get for $h\in \ch$
\[
\P[\cd_h \cup \cg_h \mid X] \geq \P[ \cg_h \mid X] \geq (1-L_j^{-\gamma})^{k_0v_0}\geq 1-2k_0v_0L_j^{-\gamma}.
\]
Since $\cd_h \cup \cg_h, h\in \ch$ are conditionally independent given $X$, we have 
\begin{equation}\label{e:A2MapAh}
\P[\cap_{h\in\ch}(\cd_h \cup \cg_h) \mid X] \geq  (1-L_j^{-\gamma})^{v_0k_0 L_j^{3/2}}\geq 9/10
\end{equation}
for $L_j$ sufficiently large. Now
\[
\P[\cd_h\mid X,( \cd_h \cup \cg_h)] \geq \P[\cd_h\mid X] = \prod_{i=1}^{N_X} \left(\frac12 \wedge S_j(X^j_{T_i})\right).
\]
Indeed, observe that if $T_i\subseteq \tilde{U}_2\setminus \tilde{U}_1$, $X_{T_i}$ is semi-bad and hence 
$$\P[\cd_h^i]\geq S_j(X^j_{T_i})-v_0L_j^{-\gamma}\geq \frac12.$$

Also observe that since none of the multi-blocks that are tried (over all $h$ and $i\in [N_X]$) are non-neighbouring it follows that $\{\cd_h:h\in \ch\}$ is independent conditionally on $X$ and $\cap_{h\in\ch}(\cd_h \cup \cg_h)$ and hence

\begin{align}\label{e:A2MapBh}
\P[\cup_{h\in\ch} \cd_h \mid X, \cap_{h\in\ch}(\cd_h \cup \cg_h)] &\geq 1-\left(1- (\frac{1}{2})^{k_0v_0}\prod_{i=1}^{N_X} S_j(X^j_{T_i})\right)^{L_j^{3/2}}\nonumber\\
&\geq \frac{9}{10} \wedge \frac14 L_j \prod_{i=1}^{N_X} S_j(X^j_{T_i})
\end{align}
since $1-e^{-x}\geq x/4\wedge 9/10$ for $x\geq 0$ and $L_j^{1/2}> 2^{k_0v_0}$ for $L_j$ sufficiently large.

Further, set
\[
\cm=\left\{\exists h_1\neq h_2 \in\ch: \cd_{h_1}\setminus \cg_{h_1}, \cd_{h_2}\setminus \cg_{h_2} \right\}.
\]
We then have
\begin{align}\label{e:A2MapCh}
\P[\cm \mid X, \cap_{h\in\ch}(\cd_h \cup \cg_h)]
&\leq {L_j \choose 2} \P[\cd_{h}\setminus \cg_{h}\mid X,\cap_{h\in\ch}(\cd_h \cup \cg_h)]^2\nonumber\\
&\leq {L_j \choose 2}2\left(\prod_{i=1}^{N_X} S_j(X^j_{T_i}) \wedge 2v_0k_0L_j^{-\gamma} \right)^2\nonumber\\
&\leq 2k_0v_0L_j^{-(\gamma-2)} \prod_{i=1}^{N_X} S_j(X^j_{T_i}).
\end{align}

Finally let $\cj$ denote the event
\[
\cj=\left\{Y^j_{k} \hbox{ is good for all } k\in \tilde{U}_2\setminus \cup_{h\in\ch, 1\leq i \leq N_X}\{\Upsilon_h(T_i)\}\right\}.
\]
Then using \eqref{e:goodconditiony} again
\begin{equation}\label{e:A2MapDh}
\P[\cj \mid X, \cup_{h\in\ch} \cd_h, \cap_{h\in\ch}(\cd_h \cup \cg_h), \neg \cm] \geq \left(1- L_j^{-\gamma}\right)^{4v_0L_j^{2\alpha-2}} \geq 9/10.
\end{equation}

Now let $\mathcal{B}$ denote the event that none of the external buffer zones of $Y^{j+1}_{U}$ are conjoined and $B^{j+1}_{U}$ is the domain of $Y^{j+1}_{U}$. Observe that on $\cap_{h\in \ch}(\cd_h \cup \cg_{h}) \cap \cj$, $B^{j+1}_{U}$ is a valid potential domain for $Y^{j+1}_{U}$ and hence we have that 

\begin{equation}\label{e:A2MapEh}
\P[\mathcal{B} \mid X,  \cup_{h\in\ch} \cd_h, \cap_{h\in\ch}(\cd_h \cup \cg_h), \cj, \neg \cm] \geq (1-10^{-(j+10)})^{4v_0} \geq 9/10.
\end{equation}  
%\textbf{NEED TO MAKE SURE $v_0$ is not very large}.

If $\mathcal{B},\cj,\cup_{h\in\ch} \cd_h$ and $\cap_{h\in\ch}(\cd_h \cup \cg_h)$ all hold and $\cm$ does not hold then by definition $Y^{j+1}_{U}$ is valid  and there is $h_0\in\ch$ such that $\cd_{h_0}$ holds and $\cg_{h'}$ holds for all $h'\in\ch\setminus\{h_0\}$. The $\alpha$-canonical map $\Upsilon_{h_0}$ then gives rise to an embedding of $X$ into $Y=Y^{j+1}_{U}$. It follows from \eqref{e:A2MapAh}, \eqref{e:A2MapBh}, \eqref{e:A2MapCh}, \eqref{e:A2MapDh} and \eqref{e:A2MapEh}  that
%
%Then by Proposition~\ref{admissiblemap2} we have that $X\hookrightarrow Y$. Hence by~\eqref{e:A2MapA}, \eqref{e:A2MapB}, \eqref{e:A2MapC}, and~\eqref{e:A2MapD} and the fact that $\cj$ is conditionally independent of the other events that
\begin{eqnarray}
S_{j+1}(X) &\geq & \P[\cup_{h\in\ch} \cd_h, \cap_{h\in\ch}(\cd_h \cup \cg_h), \cj, \mathcal{B}, \neg \cm \mid X] \nonumber\\
&= &\P[\cj \cap \mathcal{B} \mid X, \cup_{h\in\ch} \cd_h, \cap_{h\in\ch}(\cd_h \cup \cg_h), \cj, \neg \cm]\nonumber\\
&~& \P[\cup_{h\in\ch} \cd_h, \neg \cm \mid X,\cap_{h\in\ch}(\cd_h \cup \cg_h)]\P[\cap_{h\in\ch}(\cd_h \cup \cg_h)\mid X]\\
&\geq &\frac{7}{10}\left[ \left(\frac{9}{10}  \wedge \frac14 L_j \prod_{i=1}^{N_X} S_j(X_{\ell_i})\right) - 2v_0k_0L_j^{-(\gamma-2)} \prod_{i=1}^{N_X} S_j(X_{\ell_i}) \right]\\
&\geq &\frac{1}{2} \wedge \frac{1}{10} L_j \prod_{i=1}^{N_X} S_j(X_{\ell_i}).
\end{eqnarray}
%\end{align*}
This completes the proof.
%Combining with~\eqref{e:A2MapCEbound} we have that
%\begin{align*}
%\P[X\hookrightarrow Y\mid X]
%&\geq \frac{1}{2} \wedge \frac1{10} L_j \prod_{i=1}^{K_X} S_j(X_{\ell_i}),
%\end{align*}
%which completes the proof.
\end{proof}

\begin{lemma}\label{l:A2Boundh}
When $0<p< \frac12$ and $v\geq 1$,
\[
\mathbb{P}(X\in \A{2}_{X,j+1}, S_{j+1}(X)\leq p, V_X\geq v)\leq \frac15 p^{m_{j+1}} L_{j+1}^{-\beta}L_{j+1}^{-\gamma(v-1)}.
\]
\end{lemma}

\begin{proof}
We have that
\begin{align}
\mathbb{P}(X\in \A{2}_{X,j+1}, S_{j+1}(X)\leq p, V_{X}\geq v) &\leq \P\left[\frac1{10} L_j \prod_{i=1}^{N_X}S_j(X_{T_i}) \leq p, V_{X}\geq v\right]\nonumber\\
&\leq  500\left(\frac{10 p}{L_j}\right)^{m_{j+1}}L_{j+1}^{-\gamma(v-1)} \leq \frac15 p^{m_{j+1}} L_{j+1}^{-\beta}L_{j+1}^{-\gamma(v-1)}
\end{align}
where the first inequality holds  by Lemma~\ref{l:A2Maph}, the second by Proposition~\ref{l:totalSizeBoundh} and the third holds for large enough $L_0$  since $m_{j+1}>m>\alpha\beta$.
\end{proof}

\subsection{Case 3}
Case 3 involves components with very large size. The class of components $\A{3}_{X,j+1}$ is defined as
\[
\A{3}_{X,j+1} := \left\{X: V_{X}> v_0\right\}.
\]

\begin{lemma}\label{l:A4Maph}
Condition on $X=X^{j+1}_{U}\in \A{3}_{X,j+1}$ with $|U|=v>v_0$. Let the bad $j$ level components of $X$ be $X^j_{T_1}, X^j_{T_2},\ldots , X^j_{T_{N_X}}$. Then we have  
\[
S_{j+1}(X) \geq (8k_0)^{-16vk_0^2}100^{-4(j+10)v}2^{-v}2^{-4k_0v}\prod_{i=1}^{N_X}S_j(X^j_{T_i})
\]
\end{lemma}

\begin{proof}
Let $\hat{U}=U_{X}\subseteq \R^2$, $\tilde{U}\subseteq \Z^2$, $\tilde{U}_1\subseteq \tilde{U}_0 \subseteq \tilde{U}_2 \subseteq \Z^2$, be defined as in the proof of Lemma \ref{l:A2Maph}. For $i\in [N_X]$, let $\cd_i$ denote the following event. 
%$$\cd_i=\{Y^{j}_{T_{i}} {\rm valid}, X^{j}_{T_i}\hookrightarrow Y^j_{T_i}\}.$$

If $T_i\subseteq \tilde{U}_2\setminus \tilde{U}_1$, then
\[
\cd_i=\left\{X^j_{T_i} \hookrightarrow Y^j_{T_i}; Y^j_{\ell} \hbox{ is good for all } \ell\in T_i \right\}.
\]
Otherwise set 
\[
\cd_{i}=\left\{X^j_{T_i} \hookrightarrow Y^j_{\upsilon_{h}(T_i)} \right\}.
\]
Let
$$\cd= \cap_{i=1}^{N_{X}}\cd_i.$$

For $\ell \in \tilde{U}_2\setminus (\cup_{i=1}^{N_X} T_i)$ let 
$$\cg_{\ell}=\{Y^j_{\ell}~\text{is good}\}$$
and set 
$$\cg=\cap_{\ell\in \tilde{U}_2\setminus (\cup_{i=1}^{N_X} T_i)} \cg_{\ell}.$$ 
Observe that $|\tilde{U}_2|\leq 4vL_j^{2\alpha}$. 
Finally let 
$\mathcal{B}$ denote the event that $Y=Y^{j+1}_{U}$ is valid and $\hat{U}$ is the domain of $Y^{j+1}_{U}$. Let $\Upsilon$ be the $*$-canonical map from $\hat{U}$ to itself with respect to $\mathcal{T}=\{T_1,T_2,\ldots, T_{N_X}\}$ which exists by Proposition \ref{p:starcanonical}. On 
$\cd \cap \cg \cap \mathcal{B}$, we get an embedding of $X$ into $Y=Y^{j+1}_{U}$ given by $\Upsilon$. Hence it follows that 
$$S_{j+1}(X) \geq \P[\cd]\P[\cg\mid \cd]\P[\mathcal{B}\mid \cg, \cd].$$

Now observe that the total size of $T_i$'s contained in $\tilde{U}_2\setminus \tilde{U}_1$ must be at most $4k_0v$ and hence arguing as in the proof of Lemma \ref{l:A2Maph} we get that

$$\P[\cd]= \left(\frac 12\right)^{4k_0v}\prod_{i=1}^{N_X}S_j(X^j_{T_i}).$$ 

Also using the recursive hypothesis (\ref{e:goodconditiony}) we get that 
$$\P[\cg\mid \cd]\geq (1-L_{j}^{-\gamma})^{4vL_j^{2\alpha}}\geq 2^{-v}$$  
as $\gamma>\alpha$ and $L_0$ is sufficiently large. Finally observe that on $\cd \cap \cg$, the curve corresponding to the boundary of $\hat{U}$ is a valid level $(j+1)$ boundary curve for $Y$ and hence,
$$\P[\mathcal{B}\mid \cg, \cd]\geq (8k_0)^{-16vk_0^2}100^{-4(j+10)v}.$$

Putting all these together we get the lemma.  
\end{proof}

\begin{lemma}\label{l:A4Boundh}
When $0<p\leq \frac12$ and $v\geq 1$,
\[
\mathbb{P}(X\in \A{3}_{X,j+1}, S_{j+1}(X)\leq p, V_{X}\geq v)\leq \frac15 p^{m_{j+1}} L_{j+1}^{-\beta}L_{j+1}^{-\gamma(v-1)}.
\]
\end{lemma}

\begin{proof}
Without loss of generality we can take $v\geq v_0$. Then we have using Lemma \ref{l:A4Maph} and Proposition \ref{l:totalSizeBoundh}
\begin{eqnarray}
\mathbb{P}(X\in \A{3}_{X,j+1}, S_{j+1}(X)\leq p, V_{X}\geq v) &=& \sum_{v'=v}^{\infty}\P[S_{j+1}(X)\leq p, V_{X}=v']\nonumber\\
&\leq & \sum_{v'=v}^{\infty}\P\biggl[(8k_0)^{-16v'k_0^2}100^{-4(j+10)v'}2^{-v'-4k_0v'}\nonumber\\
&~& \prod_{i=1}^{N_X}S_j(X^j_{T_i})\leq p, V_{X}=v'\biggr]\nonumber\\
&\leq & \sum_{v'=v}^{\infty} 500p^{m_{j+1}}\times (2000k_0)^{64k_0^2v'(j+10)m_{j+1}}L_{j+1}^{-9\gamma(v'-1)}\nonumber\\
&\leq & 500p^{m_{j+1}}L_{j+1}^{-\beta}L_{j+1}^{-\gamma(v-1)}\left(\sum_{v'=v}^{\infty} (2000k_0)^{64k_0^2(j+10)m_{j+1}}L_{j+1}^{-5\gamma v'}\right)\nonumber\\
&\leq & \frac15 p^{m_{j+1}} L_{j+1}^{-\beta}L_{j+1}^{-\gamma(v-1)}  
\end{eqnarray}
where the penultimate inequality follows from $
\gamma (v_0-1) >\beta$ and  $v_0>5$ and the last inequality follows by taking $L_0$ sufficiently large.
\end{proof}

\subsection{Case 4}
The final case is the case of components of size not too large, but with a large size of bad subcomponents.
The class of blocks $\A{4}_{X,j+1}$ is defined as
\[
\A{4}_{X,j+1} := \left\{X:K_X \geq V_{X}k_0 , V_{X}\leq v_0\right\}.
\]

\begin{lemma}\label{l:A5Maph}
Condition on a $(j+1)$ level component $X=X^{j+1}_{U}\in \A{5}_{X,j+1}$ with $|U|=v\leq v_0$. Let the bad $j$ level components of $X$ be $X^j_{T_1}, X^j_{T_2},\ldots , X^j_{T_{N_{X}}}$ such that $\sum_{i=1}^{N_{X}}|T_i|\geq vk_0$. Then we have  
\[
S_{j+1}(X) \geq (8k_0)^{-16vk_0^2}100^{-4(j+10)v}2^{-v}2^{-4k_0v}\prod_{i=1}^{N_X}S_j(X^j_{T_i})
\]
\end{lemma}

Proof of Lemma \ref{l:A5Maph} is identical to the proof of Lemma \ref{l:A4Maph}, i.e.\ we once again get the result by considering the  $*$-canonical map from the domain of $X$ to itself and asking tor $X$ and $Y$ to have the same domain. We omit the details.

To complete the analysis of this case we have the following lemma.

\begin{lemma}\label{l:A5Boundh}
When $0<p\leq \frac12$ and $v\geq 1$,
\[
\mathbb{P}(X\in \A{4}_{X,j+1}, S_{j+1}(X)\leq p, V_{X}\geq v)\leq \frac15 p^{m_{j+1}} L_{j+1}^{-\beta}L_{j+1}^{-\gamma(v-1)}.
\]
\end{lemma}

\begin{proof}
Fix $p\leq \frac{1}{2}$ and $v\geq 1$. By definition of $\A{4}_{X,j+1}$ and using Lemma \ref{l:A5Maph} and Proposition \ref{l:totalSizeBoundh}, we get that 

\begin{eqnarray}
\mathbb{P}(X\in \A{4}_{X,j+1}, S_{j+1}(X)\leq p, V_{X}\geq v) &=& \sum_{v'=v}^{v_0}\P[X\in \A{5}_{X,j+1}, S_{j+1}(X)\leq p, V_{X}=v']\nonumber\\
&\leq & \sum_{v'=v}^{v_0}\P\biggl[\prod_{i=1}^{N_X}S_j(X^j_{T_i})\leq (2000k_0)^{64k_0^2v'(j+10)}p, V_{X}=v', K_{X}\geq vk_0\biggr]\nonumber\\
&\leq& \sum_{v'=v}^{v_0}500p^{m_{j+1}}(2000k_0)^{64k_0^2v'(j+10)m_{j+1}}L_{j}^{-\gamma v'k_0/10}\nonumber\\
&\leq & 500p^{m_{j+1}}(2000k_0)^{64k_0^2v_0(j+10)m_{j+1}}\sum_{v'=v}^{v_0}L_{j}^{-\gamma v'k_0/10}\nonumber\\
&\leq & 2000p^{m_{j+1}}(2000k_0)^{64k_0^2v_0(j+10)m_{j+1}}L_{j}^{-\gamma v k_0/10}\nonumber\\
& \leq & \frac15 p^{m_{j+1}} L_{j+1}^{-\beta}L_{j+1}^{-\gamma(v-1)}  
\end{eqnarray}
where the final inequality follows because $\gamma k_0 > 10\alpha \beta$ and $k_0> 10 \alpha \gamma$ and taking $L_0$ sufficiently large. 
This completes the proof.
\end{proof}

\subsection{Proof of Theorem~\ref{t:tailh}}
We now put together the four cases to establish the tail bounds.

\begin{proof}[Proof of Theorem~\ref{t:tailh}]
The case of $\frac12\leq p \leq 1-L_{j+1}^{-1}$ is established in Lemma~\ref{l:A1finalh}.  By Lemma~\ref{l:A1Map1h} and Lemma \ref{l:A1Sizeh} we have that $S_{j+1}(X) \geq \frac12$ for all $X\in \A{1}_{X,j+1}$ since $L_0$ is sufficiently large.  Hence we need only consider $0<p<\frac12$ and cases 2 to 4.  By Lemmas~\ref{l:A2Boundh}, \ref{l:A4Boundh} and~\ref{l:A5Boundh} then
\begin{align*}
\mathbb{P}(S_{j+1}(X)\leq p) \leq \sum_{l=2}^4 \mathbb{P}(X\in \A{l}_{X,j+1}, S_{j+1}(X)\leq p)\leq  p^{m_{j+1}} L_{j+1}^{-\beta}L_{j+1}^{-\gamma(v-1)}.
\end{align*}
The bound for $S_{j+1}^{\mathbb{Y}}$ follows similarly.
\end{proof}

\section{Estimates on Size of Components}
Our objective here is to bound the probability the $(j+1)$-level components have large size, i.e., we want to prove recursive estimates \eqref{e:sizex1h} and \eqref{e:sizey1h} at level $(j+1)$. We only prove the following theorem, the corresponding bound for $\mathbb{Y}$ components is identical.

\begin{theorem}
\label{t:sizeh}
Let $X$ be a component of $\mathbb{X}$ at level $(j+1)$ having law $\mu_j^{\mathbb{X}}$. Let $V_{X}$ denote the size of $\mathbb{X}$. Then we have for all $v\geq 1$,
$$\P[V_{X}\geq v] \leq L_{j+1}^{-\gamma(v-1)}.$$
\end{theorem}

\begin{proof}
This follows immediately from Proposition \ref{l:totalSizeBoundh} and observing that $L_0$ and hence $L_j$ is sufficiently large.
\end{proof}

\section{Estimates for good blocks}\label{s:goodh}
We prove the remaining recursive estimates, i.e., the ones for the good blocks, in this final section thereby completing the induction. We start with showing that most blocks are good.
\subsection{Most Blocks are good}

%\begin{theorem}
%\label{goodprobabilityh}
%Fix $u\in \Z^2$. Let $X=X^{j+1}(u)$ denote the corresponding $\mathbb{X}$-block at level $(j+1)$. Then $\mathbb{P}(X~\text{is good})\geq 1-L_{j+1}^{-\delta}$. Similarly for $\mathbb{Y}$-block $Y=Y^{j+1}(u)$ at level $(j+1)$, $\mathbb{P}(Y \text{is good})\geq 1-L_{j+1}^{-\delta}$.
%\end{theorem}

First we prove the recursive estimates (\ref{e:goodconditionx}) and (\ref{e:goodconditiony}) at level $(j+1)$. We shall only prove the estimate \eqref{e:goodconditionx} as the other one follows in a similar manner. 

\begin{theorem}
\label{t:goodconditionh}
For $u\in \Z^2$, let $X=X^{j+1}_u$ denote the corresponding $\mathbb{X}$-block at level $(j+1)$. For $V\subseteq \Z^2\setminus \{u\}$, let $\cf_V=\cf_V^{\mathbb{X}}$ be as defined in \S~\ref{s:rech} (at level $j+1$). Then we have
$$\P[X~\text{is good}\mid \cf_V] \geq 1-L_{j+1}^{-\gamma}.$$ 
\end{theorem}

Let us first set-up some notation before we move towards proving Theorem \ref{t:goodconditionh}. Let $\mathcal{C}_u$ denote the set of all potential boundary curves of $X$ provided $V_{X}=1$, i.e., $\mathcal{C}_{u}$ denotes the set of potential boundary curves through the buffer zone of $B^{j+1}(u)$. Conditional on $\cf_{V}$, let $\mathcal{C}^*_{u,V}\subseteq \mathcal{C}_{u}$ denote the set of all potential boundary curves that are compatible with $\cf_{V}$. By the assumption on $\cf_{V}$, we must have that $\mathcal{C}^*_{u,V}$ is non-empty, e.g., if $V$ contains all the vertices surrounding $u$, then $\mathcal{C}^*_{u,V}$ will be a singleton.  

Now let us fix $C\in \mathcal{C}_{u}$. Let $\hat{U}=\hat{U}(C)$ denote the domain having boundary $C$  Let $\ce_{\hat{U}}$ denote the event that $\hat{U}$ is the domain of $X$. On $\ce_{\hat{U}}$, let $U\subseteq \Z^2$ be such that $X=X^j_U$, i.e., the $(j+1)$ level block $X$ consists of the $j$ level blocks corresponding to $U$. Let $\partial U$ denote the the vertices on the boundary of $U$ (i.e. the vertices in $U$ that have neighbours outside $U$) and $U^*=U\setminus \partial U$. Let $V^*\subseteq \Z^2$ be such that $X^j_{V^*}=X^{j+1}_{V}$. Let $\cf_{V^*}$ denote the conditioning on $X^j_{V^*}$ being valid, i.e., $X^j_{V*}$ being a union of $j$-level blocks. 

Fix $C\in \mathcal{C}_{u}$. Let $U^C_{X}$ denote the total size of bad components in $U^*$ and let $W^C_{X}$ denote the number of really bad components in $U^*$. We have the following lemma.

\begin{lemma}
\label{l:goodconditionbadsemibad}
In the above set-up $\P[\{U^C_{X}\geq k_0\}\cup \{W^C_{X}\geq 1\}]\leq L_j^{-\beta/2}$.
\end{lemma}

\begin{proof}
This follows from the arguments in Proposition \ref{l:totalSizeBoundh} and using that $\beta> 4\alpha+2\gamma$, $\gamma(v_0-1)>2\beta$ and $\gamma k_0>10\beta$ are sufficiently large.
\end{proof}

Next define the following event about a stronger notion of airport. A $(L_{j}^{3/2}-1)\times (L_j^{3/2}-1)$ square $S$ of $j$ cells contained in $X^j_{U^*}$ is called a strong airport if any $L_{j}^{3/2}\times L_j^{3/2}$ square $\tilde{S}$ of $j$-level blocks containing $\tilde{S}$ is an airport. Let $\ce_{\hat{U}}^*$ denote the event that all $(L_{j}^{3/2}-1)\times (L_j^{3/2}-1)$ square of $j$ level cells contained in $X^j_{U^*}$ are strong airports.

Now we have the following Lemma about airports. First observe the following. Fix a square $S_1$ of size $L_j^{3/2}$ and a square $S_2$ of size $L_j^{3/2}-1$. Further fix a lattice animal $S$ of size at most $v_0$. Let $N(S,S_1)$ (resp.\ $N(S,S_2)$) denote the number of subsets of $S_1$ (resp.\ $S_2$) that are translates of $S$. It is easy to see that 
$|N(S,S_2)|\geq (1-L_j^{-1})|N(S,S_1)|$.

\begin{lemma}
\label{l:airport1}
Let $S\subseteq \Z^2$ be a fixed square of size $(L_{j}^{3/2}-1)$. Consider the set of blocks $X^{j}_{S}$. Fix a $j$-level semi-bad $\mathbb{Y}$-component $Y=Y^j_{S'}$. Let $\mathcal{S}$ denote the set of subsets of $S$ that are translates of $S'$. Let $H$ denote the event that 
$$\#\{\tilde{S}\in \mathcal{S}; A^{X^j_{\tilde{S}}}_{\rm valid}, X^j_{\tilde{S}}\hookrightarrow Y\} \geq (1-v_0^{-3}k_0^{-4}100^{-j})|\mathcal{S}|.$$
Then we have 
$\P[H\mid Y] \geq 1-e^{-cL_j^{5/2}}$ for some constant $c$ not depending on $L_j$.
\end{lemma}

\begin{proof}
Observe the following. As $|S'|\leq v_0$ it follows that $\mathcal{S}$ can be partitioned into $4v_0^2$ subsets $\mathcal{S}_{i}$, $i\in [4v_0^2]$ such that $\tilde{S}_1, \tilde{S}_2\in \mathcal{S}_{i}$ for some $i$ implies that $\tilde{S}_1$ and $\tilde{S}_2$ are non neighbouring. By using a Chernoff bound and $S_j(Y)\geq (1-v_0^{-5}k_0^{-4}100^{-j})$ it follows that for each $i$,
$$\P[\#\{\tilde{S}\in \mathcal{S}_i; \neg A^{X^j_{\tilde{S}}}_{\rm valid}~\text{or}~ X^j_{\tilde{S}}\not\hookrightarrow Y\} \geq v_0^{-3}k_0^{-4}100^{-j}|\mathcal{S}|]\leq  e^{-cL_j^{5/2}}$$
for some constant $c$ not depending on $L_j$. Taking a union bound over all $i$ we get the lemma.   
\end{proof}

\begin{lemma}
\label{l:airport2}
In the set-up of Lemma \ref{l:airport1}, we have 
$$\P[X^j_{S}~\text{is a strong airport}]\geq 1-e^{-c'L_j^{9/4}}$$ for some constant $c'>0$ not depending on $L_j$.
\end{lemma}

\begin{proof}
Observe that, conditioned on $X^j_{\tilde{S}}$  the event $X^j_{\tilde{S}}\hookrightarrow Y$ is determined by the $0$ level structure of $Y$, i.e., by looking at whether each $0$ level block contained in $Y$ is $\mathbf{0}, \mathbf{1}$ or good. Hence for our purposes, the different number of semi-bad $Y$ at level $j$ is at most $8^{v_0}3^{4v_0L_j^2}$. The lemma now follows from Lemma \ref{l:airport1} by taking a union bound over all semi-bad $Y$ as $L_j$ is sufficiently large, and from the observation immediately preceding Lemma \ref{l:airport1}. 
\end{proof}
 
\begin{lemma}
\label{l:airportfinal}
Fix $C\in \mathcal{C}_{u}$ and let $\hat{U}$ denote the domain enclosed by $C$. Then we have $\P[\ce^*_{\hat{U}}]\geq 1-e^{-c'L_j^{9/4}}$ for some constant $c'>0$ not depending on $L_j$.
\end{lemma}

\begin{proof}
Follows in a similar manner to Lemma \ref{l:airport2} and noting that the number of $L_j^{3/2}\times L_j^{3/2}$ squares in $U^*$ are $O(L_j^{2\alpha})$ and taking a union bound over all of them. 
\end{proof}  

\begin{lemma}
\label{l:interiorbad}
On $\ce_{\hat{U}} \cap \ce^*_{\hat{U}} \cap \{U^C_{X}<k_0\} \cap \{W^C_{X}=0\}$, we have that $X$ is good.
\end{lemma}

\begin{proof}
Noticing that on $\ce_{\hat{U}}$, the $j$ level $\mathbb{X}$-blocks corresponding to $\partial U$ are all good, so this lemma follows immediately from the definition of good blocks.
\end{proof}

Now we are ready to prove Theorem \ref{t:goodconditionh}.

\begin{proof}[Proof of Theorem \ref{t:goodconditionh}]
Observe that 
$$\P[X~\text{is bad}\mid \cf_{V}]\leq \max_{C\in \mathcal{C}^*_{u}} \P[X~\text{is bad}\mid \cf_{V}, \ce_{C}].$$

Now fix $C\in \mathcal{C}^*_{u}$. Define $U=U_{C}, \hat{U}=\hat{U}_{C}$ and $U^*=U^*_{C}$ as above. Set $\tilde{V}=\Z^2\setminus U$. Let $\cf^*_{C}$ denote the event that $C$ is a valid level $(j+1)$ boundary.

\begin{eqnarray*}
\label{e:goodcondition1}
\P[X~\text{is bad}\mid \cf_{V}, X^{j}_{\tilde{V}}, \ce_{C}, \cf^*_{C}] &=& \dfrac{\P[X \text{is bad}, \ce_{C} \mid  X^{j}_{\tilde{V}}, \cf^*_{C}]}{\P[\ce_{\hat{U}} \mid X^{j}_{\tilde{V}}, \cf^*_{C}]}\\
&\leq & (8k_0)^{16k_0^2}100^{j+10}L_{j+1}^{-2\gamma} 
\end{eqnarray*}
where the last inequality follows from Lemma \ref{l:intbad} below and the construction of boundaries at level $j+1$. The theorem follows by averaging over the distribution of $j$ level blocks outside $V^*$.
\end{proof}

It remains to prove the following lemma.

\begin{lemma}
\label{l:intbad}
Fix $C\in \mathcal{C}^*_u$ and let $\hat{U}$ be as above. Set $I^C_{\rm bad}=(\neg{\ce^*_{\hat{U}}})\cup \{U^C_{X}\geq k_0\} \cap \{W^C_{X}\geq 1\}$. Consider the above set-up where we condition on $X^{j}_{\tilde{V}}$ such that it is compatible with $\ce_{\hat{U}}$. Then we have
$$\P[X~{\rm{is~bad}}, \ce_{\hat{U}} \mid X^{j}_{\tilde{V}}, \cf^*_{C}]\leq 2\P[I^C_{\rm bad}\mid X^{j}_{\tilde{V}}]\leq 2 L_{j+1}^{-2\gamma}.$$
\end{lemma}

\begin{proof}
It follows from Lemma \ref{l:interiorbad} that 
\begin{eqnarray*}
\P[X~\text{is bad}, \ce_{\hat{U}} \mid X^{j}_{\tilde{V}}, \cf^*_{C}] &= & \dfrac{\P[X \text{is bad}, \ce_{\hat{U}}, \cf^*_{C}\mid X^{j}_{\tilde{V}}]}{\P[\cf^*_{C}\mid X^{j}_{\tilde{V}}]}\\
&\leq & \dfrac{\P[I^C_{\rm bad}\mid X^{j}_{\tilde{V}}]}{\P[\cf^*_{C}\mid X^{j}_{\tilde{V}}]}.
\end{eqnarray*}

Let $\cg_{U}$ denote the event that $j$ level blocks $X^j(u')$ for all $u'\in U$ are good. Observe that since $X^{j}_{\tilde{V}}$ is such that it is compatible with $\ce_{\hat{U}}$, it follows that on $X^{j}_{\tilde{V}}\cap \cg_{U}$ we have $\cf^*_{C}$. Hence we have using the recursive estimate \eqref{e:goodconditionx} at level $j$ that 

$$\P[\cf^*_{C}\mid X^{j}_{\tilde{V}}]\geq \P[\cg_{U}\mid X^{j}_{\tilde{V}}]\geq 1-4L_j^{\alpha-\gamma}\geq \frac{1}{2}.$$
 
Observe that $I^C_{\rm bad}$ only depends on the blocks corresponding to the set $U^*_{C}$ and hence is independent of $X^{j}_{\tilde{V}}$. The lemma now follows from Lemma \ref{l:goodconditionbadsemibad} and Lemma \ref{l:airportfinal} since $\beta > 4\alpha \gamma$ and $L_j$ is sufficiently large.
\end{proof}

\subsection{Good Blocks Embed into Good Blocks}

\begin{theorem}
Let $X=X^{j+1}_u$ and $Y=Y^{j+1}_u$ be level $(j+1)$ good $\mathbb{X}$ and $\mathbb{Y}$ blocks respectively. Then we have $X\hookrightarrow Y$.
\end{theorem}

\begin{proof}
Let $C_{X}$ and $C_{Y}$ denote the boundary curves of $X$ and $Y$ respectively. Let $\tilde{U}_{X}$ and $\tilde{U}_{Y}$ denote the domains bounded by these curves. Let $\hat{U}_{X}, \hat{U}_{Y}\subseteq \Z^2$, such that $X=X^j_{\hat{U}_{X}}$ and $Y=Y^j_{\hat{U}_{Y}}$. Let $\mathcal{T}=\{T_1,T_2,\ldots , T_{N_{X}}\}$ (resp.\ $\mathcal{T}'=\{T'_1,T'_2,\ldots T'_{N_Y}\}$) be the set of subsets of $\hat{U}_X$ (resp.\ $\hat{U}_{Y}$) such that $X^j_{T_i}$ (resp.\ $Y^j_{T'_i}$) are the $j$-level bad subcomponents of $X$ (resp.\ $Y$).

By Proposition \ref{p:canonical1} there exists canonical maps $\Upsilon^{j+1}_{h_1,h_2}$, $(h_1,h_2)\in [L_j^2]\times [L_j^2]$ from $\tilde{U}_{X}$ to $\tilde{U}_{Y}$ with respect to $\mathcal{T}$ and $\mathcal{T}'$  which are bi-Lipschitz with Lipschitz constant $(1+10^{-(j+5)})$ such that for each $i\in [N_{X}]$ we have $\Upsilon_{h_1,h_2}(T_{i})=(h_1-1,h_2-1)+\Upsilon_{1,1}(T_{i})$ and for all $i\in [N_{Y}]$ we have  $\Upsilon^{-1}_{h_1,h_2}(T_{i})=-(h_1-1,h_2-1)+\Upsilon^{-1}_{1,1}(T'_{i})$. Now since all rectangles of $L_j^{3/2}\times L_j^{3/2}$ sub-blocks are airports, it follows that for all $i\in [N_{X}]$
$$\#\{(h_1,h_2)\in [L_j^2]\times [L_j^2]: X^j_{T_i}\not\hookrightarrow Y^{j}_{\Upsilon_{h_1,h_2}(T_i)}\} \leq v_0^{-2}k_0^{-4}100^{-(j-1)}L_j^4$$
and for all  $i\in [N'_{Y}]$
$$\#\{(h_1,h_2)\in [L_j^2]\times [L_j^2]: X^j_{\Upsilon^{-1}_{h_1,h_2}(T'_i)}\not\hookrightarrow Y^{j}_{T'_i}\} \leq v_0^{-2}k_0^{-4}100^{-(j-1)}L_j^4.$$
By taking a union bound it follows that there exists a canonical map $\Upsilon=\Upsilon^{j+1}$ from $\tilde{U}_{X}$ to $\tilde{U}_{Y}$ with respect to $\mathcal{T}$ and $\mathcal{T}'$  which are bi-Lipschitz with Lipschitz constant $(1+10^{-(j+5)})$, such that for all $i\in [N_X]$ for all $i'\in [N'_{Y}]$ we have $X^j_{T_i}\hookrightarrow Y^{j}_{\Upsilon_{h_1,h_2}(T_i)}$. The theorem now follows from definition that $X\hookrightarrow Y$.
\end{proof}

\bibliography{embedh}
\bibliographystyle{plain}

\end{document}